\documentclass[10pt,reqno,oneside]{amsart}
\usepackage{amsfonts,amssymb,amsmath}
\usepackage{authblk}

\parskip        \baselineskip
\usepackage{aliascnt}
\usepackage[utf8]{inputenc}
\usepackage[english]{babel}
\usepackage[foot]{amsaddr}
\usepackage{color}
\usepackage{enumitem}
\usepackage{cite}

\newcommand{\suchthat}{\;\ifnum\currentgrouptype=16 \middle\fi|\;}
\def\R{{\mathbb{R} }}

\DeclareMathOperator{\dvr}{div}

\newcommand{\LPND}[1]{L^{#1}_{\dvr}(\Omega)}
\newcommand{\LND}{\LPND{2}}
\newcommand{\WND}{W^{1,2}_{0,\dvr}(\Omega)}

\def\eR{{\mathbb{R}}}

\newcommand{\N}{\mathbb{N}}
\newcommand{\AD}{\mathcal{A}}
\newcommand{\FF}{\mathcal{F}}
\newcommand{\CC}{\mathcal{C}}
\newcommand{\CS}{\tilde{\mathcal{F}}}
\newcommand{\CA}{\tilde{\mathcal{A}}}
\newcommand{\HH}{\mathbb{H}}

\newcommand{\VV}{\mathbb{V}}
\newcommand{\SSS}{\mathbb{S}}

\newcommand{\UAD}{\mathbb{U}_\mathrm{ad}}
\newcommand{\UR}{\mathbb{U}_r}
\newcommand{\dH}{\widehat H}
\newcommand{\du}{\widehat u}
\newcommand{\dv}{\widehat v}
\newcommand{\dpr}{\widehat p}
\newcommand{\dF}{\widehat F}
\newcommand{\dM}{\widehat M}
\newcommand{\RR}{\mathcal{R}}
\newcommand{\D}{\mathbf{D}}

\newcommand{\dt}{\,\mathrm dt}
\newcommand{\dx}{\,\mathrm dx}

\newcommand{\eps}{\varepsilon}
\newcommand{\fix}{{\ast}}

\newcommand{\norm}[1]{\ensuremath\left\| #1 \right\|}
\newcommand{\tnorm}[1]{\ensuremath\| #1 \|}
\newcommand{\bignorm}[1]{\ensuremath\big\| #1 \big\|}

\newcommand{\wto}{\rightharpoonup}

\theoremstyle{theorem}
\newtheorem{thm}{Theorem}[section]

\newaliascnt{lemma}{thm}
\newtheorem{lem}[thm]{Lemma}
\aliascntresetthe{lemma}

\newaliascnt{proposition}{thm}
\newtheorem{prop}[thm]{Proposition}
\aliascntresetthe{proposition}

\newaliascnt{corollary}{thm}
\newtheorem{corollary}[thm]{Corollary}
\aliascntresetthe{corollary}

\newaliascnt{remark}{thm}
\newtheorem{remark}[thm]{Remark}
\aliascntresetthe{remark}

\newaliascnt{definition}{thm}
\newtheorem{mydef}[thm]{Definition}
\aliascntresetthe{definition}

\numberwithin{equation}{section}

\makeatletter
\renewenvironment{proof}[1][\proofname]{%
	\par\pushQED{\qed}\normalfont%
	\topsep6\p@\@plus6\p@\relax
	\trivlist\item[\hskip\labelsep\bfseries#1\@addpunct{.}]%
	\ignorespaces
}{%
	\popQED\endtrivlist\@endpefalse
}
\makeatother

\definecolor{LinkColor}{rgb}{0,0.5,0}
\usepackage{hyperref}
\hypersetup{%
	colorlinks	=true,
	linkcolor	=red,
	citecolor	=LinkColor,
	urlcolor	=blue,
}

\begin{document}

\title[Analysis and optimal control of a model for magneto-viscoelastic fluids]{Strong well-posedness, stability and optimal control theory for a mathematical model for magneto-viscoelastic fluids}

\date{\today}

\maketitle
\vspace{-1cm}

\begin{center}
    \scshape
    Harald Garcke\footnotemark[1], 
    Patrik Knopf\footnotemark[1], 
    Sourav Mitra\footnotemark[2]
    and
    Anja Schl\"omerkemper\footnotemark[2]
\end{center}

\footnotetext[1]{
    Fakult\"at für Mathematik, 
    Universit\"at Regensburg, 
    93053 Regensburg, 
    Germany \\
    \tt(\href{mailto:harald.garcke@mathematik.uni-regensburg.de}
    {harald.garcke@mathematik.uni-regensburg.de},
	\href{mailto:patrik.knopf@mathematik.uni-regensburg.de}
	{patrik.knopf@mathematik.uni-regensburg.de})
}

\footnotetext[2]{
    Institut f\"ur  Mathematik,  
    Universit\"at  W\"urzburg,  
    97074 W\"urzburg,  
    Germany, \\
    \tt(\href{mailto:sourav.mitra@mathematik.uni-wuerzburg.de}
    {sourav.mitra@mathematik.uni-wuerzburg.de},
	\href{mailto:anja.schloemerkemper@mathematik.uni-wuerzburg.de}
	{anja.schloemerkemper@mathematik.uni-wuerzburg.de})
}

\begin{abstract}
In this article, we study the strong well-posedness, stability and optimal control of an incompressible magneto-viscoelastic fluid model in two dimensions. 
The model consists of an incompressible Navier--Stokes equation for the velocity field, an evolution equation for the deformation tensor, and a gradient flow equation for the magnetization vector.
First, we prove that the model under consideration posseses a global strong solution in a suitable functional framework. Second, we derive stability estimates with respect to an external magnetic field. Based on the stability estimates we use the external magnetic field as the control to minimize a cost functional of tracking-type. We prove existence of an optimal control and derive first-order necessary optimality conditions. Finally, we consider a second optimal control problem, where the external magnetic field, which represents the control, is generated by a finite number of fixed magnetic field coils.
\\[1ex]
\\[1ex]
\end{abstract}
\setlength\parskip{1ex}
\setlength\parindent{0ex}
\allowdisplaybreaks
		
\section{Introduction}
Magnetic materials have a huge variety of technical applications. In this paper we are interested in magneto-viscoelastic materials which have the important property that their elastic behavior can be influenced by magnetic fields and vice versa. They thus belong to the class of smart materials and react to external stimuli in a remarkable way. We are particularly interested in controlling the behavior of magneto-viscoelastic fluids by means of external magnetic fields. In fact, a change of the applied external magnetic field will lead to changes in the magnetization of the material. As a consequence, due to the coupling between magnetic and elastic effects, these changes will be converted to changes in the flow map of the fluid's body. The induced motion within the body can then be used for specific technical applications. 

In \cite{benesovagarcia,Forster} a system of partial differential equations was introduced which describes incompressible magneto-viscoelastic fluids based on a gradient-flow dynamics for the magnetization vector, the incompressible Navier-Stokes equation and an evolution equation for the deformation tensor. The existence of a weak solution to this system was established in \cite[Chapter 3]{Forster} in two and three dimension. The uniqueness of such weak solutions was proved in \cite{anjazab}. 
As this system involves the Navier--Stokes equation, the global strong well-posedness of the three-dimensional model is still an open problem. In fact, this issue is directly related to the Millenium Problem stated by the Clay Mathematics Institute concerning the Navier--Stokes equation.
Since we are interested in strong well-posedness, stability (with respect to an external magnetic field $H$) and optimal control, we thus consider a two-dimensional variant of the model proposed in \cite{benesovagarcia,Forster}. 
However, although we can only prove strong well-posedness and stability in two dimensions, the optimal control theory we develop in this article would also remain valid in the three-dimensional setting, provided that the strong well-posedness and stability results could be verified.

Before we present the model, we first introduce some notation. Let $\Omega\subset \mathbb{R}^{2}$ be a bounded domain with $C^4$-boundary and let $T > 0$ be a given time. We write $Q_T := \Omega\times(0,T)$ to denote the space-time cylinder. Let $v : Q_T\to \mathbb{R}^{2}$ be the \emph{velocity field}, $p: Q_T\to \mathbb{R}$ the \emph{fluid pressure}, $F : Q_T\to  \mathbb{R}^{2\times 2}$
the \emph{deformation tensor}, and $ M : Q_T\to \mathbb{R}^3$ the \emph{magnetization vector}, all described in
Eulerian coordinates. The magnetoelastic material is exposed to an \emph{external magnetic field} 
$H: Q_T \to \mathbb{R}^3$. 
To avoid confusion, we want to make clear that in contrast to standard notation, $H$ is \emph{not} the magnetic field generated by the magnetization $M$ but an independent external field. 

The model we consider (written in a non-dimensional form) reads as follows:
\begin{subequations}\label{diffviscoelastic*}
	 \begin{alignat}{2}
		\partial_{t}v+(v\cdot\nabla)v+\mbox{div}\left((\nabla M\odot\nabla M)-FF^{T}\right)+\nabla p=& \nu\Delta v+ (\nabla H)^T M &&\mbox{ in } Q_{T},\label{diffvlin1}\\
		\dvr v=&0&&\mbox{ in } Q_{T},\label{diffvlin2}\\
		\partial_{t}F+(v\cdot\nabla)F-\nabla v F=&\kappa\Delta F&&\mbox{ in } Q_{T},\label{diffvlin3}\\
		\partial_{t}M+(v\cdot\nabla)M=&\Delta M-\frac{1}{\alpha^{2}}(|M|^{2}-1)M+H&&\mbox{ in } Q_{T},\label{diffvlin4}\\
		v=0,\ F=0,\ \partial_{n}M=&0&&\mbox{ on } \Sigma_{T},\label{diffvlin5}\\
		(v,M,F)(\cdot,0)=&(v_{0},M_{0},F_{0})&& \mbox{ in }\Omega.\label{diffvlin6}
	\end{alignat}
\end{subequations} 

Here, \eqref{diffvlin1} is the \emph{momentum balance equation}, \eqref{diffvlin2} is the \emph{incompressibility constraint}, whereas \eqref{diffvlin3} and \eqref{diffvlin4} describe the evolution of the deformation tensor $F$ (in Eulerian coordinates) and the magnetization vector $M$, respectively. We further use the notation $(\nabla M\odot\nabla M)_{ij}=\sum_{k=1}^{3}(\partial_{i}M_{k})(\partial_{j}M_{k})$, and we assume that the fluid viscosity $\nu>0$ is a constant. The constant $\kappa>0$ appearing in \eqref{diffvlin3} is an artificial regularization parameter whose value is assumed to be small. In \eqref{diffvlin4}, the term $\alpha^{-2} (|M|^2-1)M$ with $\alpha>0$ acts as a penalization corresponding to the saturation condition of the magnetization vector as it punishes any deviation of $|M|$ from one. The system is supplemented with a no-slip boundary condition for the velocity field, a homogeneous Dirichlet boundary condition for the deformation tensor, and a homogeneous Neumann boundary condition for the magnetization (see \eqref{diffvlin5}), as well as initial conditions (see \eqref{diffvlin6}). 

\subsection{Contents and main results}

Before highlighting the ideas behind the derivation of the model \eqref{diffviscoelastic*}, we first outline the structure and the main results of this paper.
The final goal of this article is to investigate optimal control problems where the quantities $v$, $F$ and $M$ are to be optimized in a desired way by adjusting the external magnetic field $H$. To analyze such optimal control problems, we first need to establish some basic results which are important in order to apply methods from the calculus of variations. Therefore, our paper is organized as follows.

\begin{itemize}[leftmargin=*]
    \item \textbf{Strong well-posedness.}
    We first ensure that \eqref{diffviscoelastic*} possesses a unique, sufficiently regular solution.
    We point out that the existence of a unique (Leray type) weak solution to the model \eqref{diffviscoelastic*} has already been established in \cite[Chapter 3, Theorem 9, p.~42]{Forster}. 
    However, especially as we want to derive first-order necessary optimality conditions for our optimal control problems, 
    the regularity of those weak solutions is by far not enough. 
    Therefore, our first main task is to establish the strong well-posedness of system \eqref{diffviscoelastic*}. \\[1ex]
    Our construction of a strong solution is inspired from the proof of the weak existence theory for \eqref{diffviscoelastic*} presented in \cite[Section 3.1]{Forster} (see also \cite{liubenesova}). Roughly speaking, the idea is to discretize only the velocity field via a Galerkin scheme. Then, we solve the equations for the deformation tensor and magnetization with the discretized velocity. Eventually, a suitable fixed point argument can be applied to obtain the existence of a solution to an intermediate problem with a finite dimensional velocity field $v_{m}$. In order to recover a solution for the original problem \eqref{diffviscoelastic*}, we derive uniform estimates in regular Sobolev spaces which allow us to pass to the limit $m\rightarrow\infty$. 
    The proof of these estimates relies on several interpolation results, which are collected in Section~\ref{FSP}, as well as on Gronwall's inequality. 
    The strong well-posedness result is stated in Theorem~\ref{globalstrong}.
    \\
    \item \textbf{Stability with respect to perturbations of the external magnetic field.} We next need to investigate how the solution of \eqref{diffviscoelastic*} reacts to changes of the external field $H$.
    To this end, we prove stability results for strong solutions to system \eqref{diffviscoelastic*} with respect to perturbations of the external magnetic field. A stability estimate with respect to the function spaces corresponding to weak solutions is stated in Theorem~\ref{stabilities}, whereas a stability result with respect to the function spaces corresponding to strong solutions is presented in Theorem~\ref{strngestimates}.
    Especially the stability in the functional framework of strong solutions will be a crucial tool for the analysis of our optimal control problems.
    \\
    \item \textbf{The control-to-state operator and its most important properties.} As a consequence of the strong well-posedness, we can define an operator $\FF$ mapping any admissible control $H$ onto the corresponding strong solution of the system \eqref{diffviscoelastic*}. This operator will be referred to as the \emph{control-to-state operator}. We next prove certain properties of $\FF$ that are needed to apply methods from the calculus of variations. More precisely, we show that $\FF$ is Lipschitz continuous (see Corollary~\ref{COR:LIP}), weakly sequentially continuous (see Proposition~\ref{PROP:WSC}) and Fr\'echet differentiable (see Proposition~\ref{PROP:FD}).
    \\
    \item \textbf{Optimal control via unconstrained external magnetic fields.}
    The motivation behind our optimal control problems is to adjust the external magnetic field $H$ in such a way that the quantities $v$, $F$ and $M$ approximate desired quantities $v_d$, $F_d$ and $M_d$ on a given time interval $[0,T]$ as closely as possible. In the first optimal control problem we investigate, this is to be achieved by minimizing the quadratic tracking-type cost functional 
    \begin{align*}
        \begin{aligned}
        I(v,p,F,M,H) &:= 
        \frac{a_1}{2} \norm{v-v_d}_{L^2(Q_T)}^2
        + \frac{a_2}{2}  \norm{F-F_d}_{L^2(Q_T)}^2
        + \frac{a_3}{2}  \norm{M-M_d}_{L^2(Q_T)}^2
        + \frac{\lambda}{2}  \norm{H}_{\HH}^2
        \end{aligned}
    \end{align*}
    subject to the following side conditions: \medskip
    \begin{itemize}
        \item $H$ is an \emph{admissible control}, $i.e.$, it belongs to a suitable function space $\HH$ (which is defined in \eqref{DEF:HH}); 
        \item $(v,p,F,M)$ is the unique strong solution of the system \eqref{diffviscoelastic*} to the external magnetic field $H$.
    \end{itemize} \medskip
    Here $a_1,a_2,a_3\ge 0$ and $\lambda>0$ are given constants which act as weights for the summands of the cost functional.
    \\[1ex]
    Using the control-to-state operator $\FF$ (see Definition~\eqref{DEF:CSO}) that maps any admissible control onto the corresponding strong solution of \eqref{diffviscoelastic*}, this problem can be reformulated as
    \begin{align}
        \left\{
        \begin{aligned}
            &\text{Minimize} && J(H) := I(\FF(H),H),\\
            &\text{subject to} && H\in\HH,
        \end{aligned}
        \right.
    \end{align}
    which is called the \emph{reduced formulation}.
    Invoking the weak sequential continuity of the control-to-state operator mentioned above, we first show in Theorem~\ref{THM:EX:1} that this optimal control problem has at least one global minimizer. This can be done by employing the direct method of the calculus of variations.\\[1ex]
    After that, we use the Fr\'echet differentiability of the control-to-state operator to characterize local minimizers by a first-order necessary optimality condition. We further show that any local minimizer actually possesses a higher regularity than prescribed, and satisfies a certain semilinear elliptic equation. These results are stated in Theorem~\ref{THM:NOC}.
    \\
    \item \textbf{Optimal control via fixed magnetic field coils.}
    The control problem introduced above is based on the rather idealistic assumption that any (locally) optimal control $H$ can actually be generated. However, since in real technical applications magnetic fields are usually generated by magnetic field coils, it might not be possible to reproduce any theoretically (locally) optimal external magnetic field in a satisfactory manner.
    \\[1ex]
    To this end, we investigate a second optimal control problem, where the external magnetic field is generated by a finite number $n\in\N$ of magnetic field coils. 
    We assume that the geometry ($i.e.$, the shape and the position) of these field coils is fixed, and only the intensity of their generated magnetic field can be adjusted. 
    Under this assumption, by linear superposition, the total external magnetic field can be expressed as
    \begin{align*}
        H(x,t) 
        = 
        \sum_{i=1}^n u_i(t)\, h_i(x) 
        =: \CC(u)(x,t)
        \quad\text{for all $x\in\Omega$, $t\in[0,T]$,}
    \end{align*} 
    where for any $i\in\{1,...,n\}$, the factor $u_i$ is proportional to the intensity of the magnetic field of the $i$-th coil. In the optimal control problem, the functions $h_i$ are assumed to be prescribed,
    and instead of the whole external magnetic field $H$, the vector-valued function of control parameters $u = (u_1,...,u_n)^T$ is now to be adjusted. In the spirit of the optimal control problem presented above, we now want to minimize the cost functional
    \begin{align*}
        \tilde J(u) := \tilde I\big(\FF(\CC(u)),u\big),
    \end{align*}
    where $\FF$ denotes again the control-to-state operator, and the functional $\tilde I$ is defined as
    \begin{align*}
        \tilde I(v,p,F,M,u) &:= 
        \frac{a_1}{2} \norm{v-v_d}_{L^2(Q_T)}^2
        + \frac{a_2}{2} \norm{F-F_d}_{L^2(Q_T)}^2 
        + \frac{a_3}{2} \norm{M-M_d}_{L^2(Q_T)}^2
        + \frac{\lambda}{2} \norm{u}_{L^2(0,T;\R^n)}^2.
    \end{align*}
    The side condition for the minimization problem is that $u$ must belong to a certain set of admissible control parameters $\UAD$ which is chosen as a box restricted subset of $L^2(0,T;\R^n)$. 
    \\[1ex]
    As for the first optimal control problem, we first prove the existence of a global minimizer by means of the direct method of the calculus of variations (see Theorem~\ref{THM:EX:2}). Next, in Theorem~\ref{THM:NOC:2}, we establish a variational inequality as a first-order necessary optimality condition, and we further show that any locally optimal control can be expressed by a certain projection formula.  
    \\[1ex]
    
\end{itemize}

\subsection{Comments on the derivation of the model} 
The model \eqref{diffviscoelastic*} is essentially derived in \cite{Forster}. Hence without going into details we will just comment on the ideas and both physical and mathematical motivations behind. 

The system \eqref{diffviscoelastic*} is derived by an energetic variational approach. The starting point is to consider a Helmholtz free energy of the system which reads as
\begin{align}\label{Helm}
\Psi(F,M)=\frac{1}{2}\int_{\Omega}|\nabla M|^{2}+\frac{1}{4\alpha^{2}}\int_{\Omega}\left(|M|^{2}-1\right)^{2}-\int_{\Omega}M\cdot H+\frac{1}{2}\int_{\Omega}|F|^{2}.
\end{align}
The magnetic contribution to the energy ($i.e.$, the first three terms of \eqref{Helm}) is motivated from micromagnetics, see, $e.g.$, the recent review \cite{DiFratta-etal2019} and references therein. For simplicity, we only consider the so-called exchange energy contribution $\frac{1}{2}\int_{\Omega}|\nabla M|^{2}$,
which reflects the tendency of the magnetization to align. In micromagnetics, the saturation condition is taken into account, which means that the modulus of the magnetization is constant. As it is typical in the mathematical literature, we set this saturation constant equal to one. In our model, we include the saturation condition by means of a penalization term which punishes the deviation of $|M|$ from 1 (the second integral on the right-hand side of \eqref{Helm}), see, $e.g.$, \cite[Section 1.2]{Kurzke}, \cite{chipotshafrir} or \cite{anjazab}. This penalization is also referred to as Ginzburg-Landau approximation. The third integral in \eqref{Helm} is the Zeeman energy associated with the external magnetic field $H$. The fourth integral in \eqref{Helm} represents the elastic energy, which for simplicity is assumed of this quadratic form. For a discussion of more general forms of the elastic energy and related mathematical difficulties see \cite{KKS}.

The derivation of \eqref{diffvlin1} relies on the least action principle and an energy dissipation law, i.e., on a variational energetic approach, cf., e.g., \cite{gigaetal}.
One first introduces the action functional $\int_0^t  \mathcal{K} - \Psi(F,M),$ where $\mathcal{K}$ represents the kinetic energy. Its variation with respect to the flow map yields the evolution equation for the linear momentum. In particular, the source term $(\nabla H)^T M$ in the momentum equation \eqref{diffvlin1} stems from a variation of the Zeeman energy with respect to the flow map. The term $\nu\Delta v$ in \eqref{diffvlin1} results from taking the first variation of the dissipation term $\nu\int_{\Omega}|\nabla v|^{2}$ 
with respect to divergence free vector fields. The pressure $p$ in \eqref{diffvlin1} can be interpreted as a Lagrange multiplier corresponding to the incompressibility constraint \eqref{diffvlin2}. The details of this derivation can be found in \cite[Section 2.7]{Forster}. 

The derivation of \eqref{diffvlin3} with $\kappa=0$ can be found in \cite[Section A.1]{Forster}. If $\kappa>0$, the term $\kappa\Delta F$  is to be understood as an artificial regularization term that is added only for mathematical reasons, which was introduced in \cite[p.~1461]{Linliu}. The authors of \cite{Linliu} explain the importance of the regularization in order to obtain global weak solutions of Leray type to a viscoelastic model, see also \cite{liubenesova}. In the case $\kappa=0$, the compactness obtained for the approximation of $F$ is only good enough to pass to the limit in the nonlinearity $FF^{T}$ (which appears in the weak formulation of the momentum equation) up to a positive Radon measure. In a similar spirit, the existence of dissipative weak solutions for  viscoelastic and magneto-viscoelastic models are proved in \cite{kaloushek} and \cite{KalouAnja} respectively.

The derivation of the magnetization equation \eqref{diffvlin4} relies on a gradient flow approach and can be found in \cite[Section 2.8.1]{Forster}.
	
\subsection{Bibliographical remarks} To the best of our knowledge, the present article is the first one to study optimal control of the magneto-viscoelastic model \eqref{diffviscoelastic*}. In fact, the literature corresponding to the well-posedness of \eqref{diffviscoelastic*} is quite recent. The global existence of weak solutions (more specifically the Theorem \ref{weaksolution}) first appeared in the thesis \cite{Forster}. The uniqueness of global weak solutions to \eqref{diffviscoelastic*} in two dimensions was proved in \cite{anjazab}. The authors in \cite{anjazab} further proved a Prodi--Serrin type criteria for the uniqueness of weak solutions in dimension three. 

As an important tool for the analysis of optimal control problems for \eqref{diffviscoelastic*} via an external magnetic field, we prove two stability results (in different functional frameworks) in Section \ref{Stabilityestimates}. The weak type stability result presented in Theorem \ref{stabilities} is an extension of the uniqueness result in two dimensions established in \cite{anjazab}. 

We emphasize the fact that the artificial regularization term in the equation for the deformation tensor plays a crucial role in order to prove the global existence of weak solutions to \eqref{diffviscoelastic*}. Even without the evolution of the magnetization, the global existence (for general initial data and without any restriction on the interval of existence) of weak solutions to an incompressible viscoelastic fluid model not involving any artificial regularization of the equation for the deformation tensor is a longstanding open question. 

In the article \cite{HuLIn}, the authors study the existence
of a global in time weak solution to an incompressible viscoelastic model (without regularizing the equation for deformation tensor) in $\mathbb{R}^{2}$, provided that the initial deformation tensor is close to the
identity matrix and the initial velocity is small. The weak-strong uniqueness of the same model is proved in \cite{HUWu}. In a torus (in dimension two and three), the local-in-time existence of strong solutions for a similar non-regularized model is established in \cite{Zhao}. There, the author also establishes a blow-up criterion in terms of the temporal integral of the maximum norm of the velocity gradient. A different approach is used to prove the global existence of dissipative weak solutions for viscoelastic fluids and magneto-viscoelastic fluids in \cite{kaloushek} and \cite{KalouAnja}, respectively. Existence of Struwe-like solutions is proven for a two-dimensional magneto-viscoelastic system on the torus without the regularizing term in the $F$-equation \cite{DeAnnaetal}. We would also like to mention \cite{LiuWalk}, where the authors consider a model for the flow of a fluid (described in Eulerian coordinates) containing visco-hyperelastic solid particles and prove the global existence of weak solutions for a small strain approximation of the deformation tensor. 

Concerning the study of optimal control problems for the Navier--Stokes equation, there already exists an extensive literature. For instance, we refer the readers to \cite{Fattorini,Fursikov1,Fursikov2,hinze-kunisch,hinze-kunisch2,barbu,bewley,casas} and the references therein. Although the optimal control of the system \eqref{diffviscoelastic*} has not been studied before, there are some recent articles on optimal control problems for the Ericksen-Leslie system which describes incompressible nematic liquid crystal flows (see, e.g., \cite{Lin89,LINLIU2}) and is related to our model. To the best of our knowledge, the article \cite{Cavaterra} is the first one to study optimal control problems for the approximation of the original Ericksen–Leslie model (in dimension two) introduced in \cite{Lin89} by using a boundary control to influence the averaged macroscopic/continuum molecular orientation. The analysis in \cite{Cavaterra} is inspired from the well-posedness and stability results proved in \cite{Bosia} and \cite{Grasselli}. The optimal control of this model is studied in \cite{Liu} using a distributed control entering the momentum balance equation. In yet another recent article \cite{Liuwangzhangzhou}, the authors investigate the optimal boundary control of a different Ericksen–Leslie system where the director field satisfies a length constraint. Roughly speaking, the model considered in \cite{Liuwangzhangzhou} couples the non-homogeneous incompressible Navier-Stokes equations and the transported flow of harmonic maps for the director field. The authors of \cite{Liuwangzhangzhou} extend the theory developed in \cite{Linlinwang} to the situation of a time dependent Dirichlet boundary condition for the director field, and they establish the existence of a global weak solution that is smooth except for finitely many singular times.
Moreover, the existence of a unique global strong solution that is smooth for $t>0$ is also established under the assumption that the image of boundary data is contained in a hemisphere. These results are then applied to study the optimal boundary control of the system they consider.

\section{Functional spaces and Preliminaries}\label{FSP}
	Before we state and prove the main results, we introduce some notation that will be used throughout this article. We use the standard notation for Lebesgue spaces and Sobolev spaces on a domain $\Omega\subset\eR^2$, $i.e.$, we write $L^{p}(\Omega)$ and $W^{s,p}(\Omega)$ for $p\in[1,\infty]$ and $s\in(0,\infty)$. For convenience, we do not distinguish between a Banach space $X$ of scalar functions and a space of a vector-valued functions with $m$ components, where each of them belongs to $X$. In particular, this means we will write $ \|\cdot \|_{L^p(\Omega)}$, $ \|\cdot \|_{W^{s,p}(\Omega)}$, etc. also when vector-valued functions are considered. 
	
	For Banach spaces $X,Y$ we denote by  $X\hookrightarrow Y$ ($X\stackrel{C}{\hookrightarrow} Y$) the continuous (compact) embedding of $X$ into $Y$. The dual space of a Banach space $X$ is denoted by $X'$, and for any $x\in X$ and $\phi\in X'$, we write $\left\langle \phi,x\right\rangle_X$ to denote the duality pairing. Moreover, $C_w([0,T];X)$ stands for the subspace of $L^\infty(0,T;X)$ consisting of such $f$ for which the mapping $t\mapsto\left\langle \phi, f(t)\right\rangle_{X}$ is continuous on $[0,T]$ for every $\phi\in X'$. 
	
	Let us also introduce the following spaces  
	\begin{align*}
		\LND&=\overline{\{v\in C^\infty_c(\Omega) 
		\suchthat \dvr v=0\text{ in }\Omega\}}^{\|\cdot\|_{L^2}},\\
		\WND&=\overline{\{v\in C^\infty_c(\Omega)
		\suchthat \dvr v=0\text{ in }\Omega\} }^{\|\cdot\|_{W^{1,2}}},\\
		V(\Omega)&=\WND\cap W^{2,2}(\Omega),\\
		W^{2,2}_n(\Omega)&=\{u\in W^{2,2}(\Omega) \suchthat \partial_n u=0\text{ on }\partial\Omega\}.
	\end{align*} 
	Here, the first three spaces consist of solenoidal vector fields and the last one is used for both scalar and tensor valued functions.
	
	Next, we introduce the Leray projector $\mathbb{P}_{\dvr}:L^{2}(\Omega)\rightarrow L^{2}_{\dvr}(\Omega)$ as 
	\begin{align}\label{Leray1}
			\mathbb{P}_{\dvr}(f)=f-\nabla p
			\quad\mbox{for any vector field}\quad 
			f\in L^{2}(\Omega),
	\end{align}
	where $p\in W^{1,2}(\Omega)$ with $\int_{\Omega} p=0$
	solves the weak Neumann problem
	\begin{align}\label{Leray}
			\int_\Omega \nabla p \cdot \nabla\varphi = \int_\Omega f \cdot \nabla\varphi \quad\mbox{for all $\varphi\in C^{\infty}(\overline{\Omega})$.}
	\end{align}
	
	\pagebreak[2]
    In the following lemma, we collect several useful interpolation inequalities for Sobolev spaces. We remark that the symbol $C$ can denote different constants.
    \begin{lem}\label{Leminterpole}
	Let $\Omega\subset\R^2$ be a bounded domain with $C^{4}$-boundary. Then, there exist positive constants $C$ (depending on $\Omega$) such that the following estimates hold:
	\begin{alignat}{2}
	\label{interpolation}
	\|\Delta f\|_{L^{4}(\Omega)}
	&\leq C\|\Delta f\|^{\frac{1}{2}}_{L^{2}(\Omega)}\big(\|\Delta f\|^{2}_{L^{2}(\Omega)}+\|\nabla\Delta f\|^{2}_{L^{2}(\Omega)}\big)^{\frac{1}{4}},
	&&\quad\mbox{$f\in W^{3,2}(\Omega)$},\\
    \label{interpolation2}
	\|f\|_{L^{\infty}(\Omega)}
	&\leq C\|f\|^{\frac{1}{2}}_{L^{2}(\Omega)}\big(\|f\|^{2}_{L^{2}(\Omega)}+\|\Delta f\|^{2}_{L^{2}(\Omega)}\big)^{\frac{1}{4}},
	&&\quad\mbox{$f\in W^{2,2}_{n}(\Omega)$},\\
	\label{interpolation3}
	\|\nabla f\|_{L^{\infty}(\Omega)}
	&\leq C\|\nabla f\|^{\frac{1}{2}}_{L^{2}(\Omega)}
	\big(\|\nabla f\|^{2}_{L^{2}(\Omega)}
	+\|\Delta f\|^{2}_{L^{2}}
	+\|\nabla\Delta f\|^{2}_{L^{2}(\Omega)}\big)^{\frac{1}{4}}, 
	&&\quad\mbox{$f\in W^{3,2}(\Omega)\cap W^{2,2}_{n}(\Omega)$},
	\\
    \label{smoreinterpole1}
	\|f\|_{L^{4}(\Omega)}
	&\leq C\big(\| f\|_{L^{2}(\Omega)}+\| f\|^{\frac{1}{2}}_{L^{2}(\Omega)}\|\nabla f\|^{\frac{1}{2}}_{L^{2}(\Omega)}\big),
	&&\quad\mbox{$f\in W^{1,2}(\Omega)$},\\
	\label{smoreinterpole2}
	\|\nabla f\|_{L^{4}(\Omega)}
	&\leq C\|\nabla f\|^{\frac{1}{2}}_{L^{2}(\Omega)}\big(\|\nabla f\|^{2}_{L^{2}(\Omega)}+\| \Delta f\|^{2}_{L^{2}(\Omega)}\big)^{\frac{1}{4}},
	&&\quad\mbox{$f\in W^{2,2}_{n}(\Omega)$}, \\
    \label{L40bnd}
	\|f\|_{L^{4}(\Omega)}
	&\leq C\|f\|^{\frac{1}{2}}_{L^{2}(\Omega)}\|\nabla f\|^{\frac{1}{2}}_{L^{2}(\Omega)},
	&&\quad\mbox{$f\in W^{1,2}_{0}(\Omega)$}, \\
    \label{interpoledir}
	\|f\|_{L^{\infty}(\Omega)}
	&\leq C\|f\|^{\frac{1}{2}}_{L^{2}(\Omega)}\|f\|^{\frac{1}{2}}_{W^{2,2}(\Omega)}
	\leq C\|f\|^{\frac{1}{2}}_{L^{2}(\Omega)}\|\Delta f\|^{\frac{1}{2}}_{L^{2}(\Omega)},
	&&\quad\mbox{$f\in W^{2,2}(\Omega)\cap W^{1,2}_{0}(\Omega).$}
\end{alignat}
\end{lem}

\textit{Comments on the proof of Lemma \ref{Leminterpole}.} All the inequalities can be derived from the Gagliardo-Nirenberg inequality with the help of elliptic regularity theory. Due to the fact that the inequalities \eqref{interpolation}, \eqref{interpolation3} and \eqref{smoreinterpole2} are somewhat more involved we refer the reader to \cite[p.~22]{liubenesova} for their proofs. Similar results can also be found in \cite[pp.~216,~226]{Carbou}.\hfill$\Box$

\bigskip

In the present article, we will use $\eps$ to denote small positive parameters. When using Young's inequality during the estimates of product terms, a possibly large constant will appear which might depend on this parameter $\eps.$ For the sake of simplicity of our notation, we will simply denote this constant by $C$ instead of $C_{\eps}.$ The use of this notation will become clear from the context. 

\section{Global existence}\label{globalex}
	
\subsection{Global existence of weak solutions}
    We point out that the choice of the parameter $\kappa>0$ does not have any impact on the mathematical analysis. From now on, we will thus set $\kappa=1$ to provide a cleaner presentation.
	In this section, the generic positive constant $C$  may depend on $\Omega$, the final time $T$, $\|H\|_{L^{2}(0,T;W^{1,2}(\Omega))}$, and the initial data.
	\begin{thm}\label{weaksolution}
		  For any $T>0,$ $v_{0}\in L^{2}_{\dvr}(\Omega),$ $F_{0}\in L^{2}(\Omega),$ $M_{0}\in W^{1,2}(\Omega)$ and $H\in L^{2}(0,T;W^{1,2}(\Omega)),$ the system \eqref{diffviscoelastic*} has a unique weak solution $(v,p,F,M)$ with the following regularity properties:
			\begin{align}\label{funframeweak}
			\left\{ \begin{aligned}
				&  v\in W^{1,\frac{4}{3}}(0,T;(W^{1,2}_{0,\dvr}(\Omega))') \cap 
				L^{\infty}(0,T;L^{2}_{\dvr}(\Omega))\cap L^{2}(0,T;W^{1,2}_{0,\dvr}(\Omega)),\\
				&  p\in W^{-1,\infty}(0,T;L^2(\Omega)),\,\,\int_{\Omega}p=0, \\
				&  F\in W^{1,\frac{4}{3}}(0,T;(W^{1,2}(\Omega))') \cap 
				L^{\infty}(0,T;L^{2}(\Omega))\cap L^{2}(0,T;W^{1,2}_{0}(\Omega)),\\
				&  M\in W^{1,\frac{4}{3}}(0,T;L^{2}(\Omega)) \cap 
				L^{\infty}(0,T;W^{1,2}(\Omega))\cap L^{2}(0,T;W^{2,2}_{n}(\Omega)).
			\end{aligned}\right.
		\end{align}
	    Here, $W^{-1,\infty}(0,T;L^2(\Omega))$ denotes the space of distributions that can be expressed as the distributional time derivative of a function in $L^\infty(0,T;L^2(\Omega))$.
		    
		Furthermore, the energy estimate 
		\begin{align}\label{weakestimate}
				&\|v(t)\|^{2}_{L^{2}(\Omega)}+\|F(t)\|^{2}_{L^{2}(\Omega)}+\|M(t)\|^{2}_{W^{1,2}(\Omega)}
				\notag\\
				& +\int_{0}^{t}\left(\|v(\tau)\|^{2}_{W^{1,2}(\Omega)}+\|F(\tau)\|^{2}_{W^{1,2}(\Omega)}+\|M(\tau)\|^{2}_{W^{2,2}(\Omega)}\right)d\tau
				\leq c_\text{w}
		\end{align}
		holds for all $t\in[0,T].$ Here, $c_\text{w}$ is a positive constant depending only on $\|v_{0}\|_{L^{2}(\Omega)},$ $\|M_{0}\|_{W^{1,2}(\Omega)},$ $\|F_{0}\|_{W^{1,2}(\Omega)}$, $\|H\|_{L^{2}(0,T;W^{1,2}(\Omega))}$, $\Omega$ and the final time $T$.
	\end{thm}
    \begin{proof}
	The existence of a weak solution to system (\ref{diffviscoelastic*}) is established in \cite[Chapter~3, Theorem~9]{Forster}. The only difference of this result compared to the one presented in Theorem~\ref{weaksolution} is that in \cite{Forster}, the time regularities 
	$$\partial_{t}v\in L^{\frac{4}{3}}(0,T;(W^{1,2}_{0,\dvr}(\Omega))'), 
	\quad\partial_{t}M\in L^{\frac{4}{3}}(0,T;L^{2}(\Omega)),
	\quad\partial_{t}F\in L^{\frac{4}{3}}(0,T;(W^{1,2}(\Omega))')$$
	are not stated explicitly and moreover, the pressure is not recovered in a suitable functional framework.
	
	However, the time regularities are hidden in the proof that is given in  \cite[Section~3.1.4.1]{Forster}. They are established for the sequence of approximate solutions (constructed by a Galerkin scheme) and are uniform with respect to the approximation parameter. Hence, the same time regularities can be recovered for the limit functions meaning that $\partial_{t}v\in L^{\frac{4}{3}}(0,T;W^{-1,2}(\Omega))$ (in particular, $\partial_{t}v\in L^{\frac{4}{3}}(0,T;(W^{1,2}_{0,\dvr}(\Omega))')$), $\partial_{t}M\in L^{\frac{4}{3}}(0,T;L^{2}(\Omega))$ and $\partial_{t}F\in L^{\frac{4}{3}}(0,T;(W^{1,2}(\Omega))')$ hold.
	
    Moreover, the recovery of the pressure $p$ is a standard but not straightforward line of argument. We will just sketch the approach and refer to \cite{Boyer} for the technical details.
    
    First, as in \cite[Section~1.5, pp.~368--369]{Boyer}, we derive the following relation from the weak formulation of the momentum balance equation \eqref{diffvlin1}:
    \begin{align}\label{Gvarphi}
    \langle G(t),\varphi\rangle_{W^{-1,2}(\Omega),W^{1,2}_{0}(\Omega)}=0
    \quad\mbox{for all}\,\,\varphi\in W^{1,2}_{0,\dvr}(\Omega)\,\,\mbox{and all}\,\, t\in[0,T].
    \end{align}	
    Here, the function $G$ is given as
    \begin{align}\label{exGt}
     G(t)=v(t)-v(0)+\int_{0}^{t}\left((v\cdot\nabla)v-\nu\Delta v+\dvr(\nabla M\odot\nabla M)-\dvr(FF^{T})-(\nabla H)^TM\right)d\tau.
    \end{align}
    Invoking $\eqref{funframeweak}$, we infer that the integrand in \eqref{exGt} belongs to $L^{\frac{4}{3}}(0,T;W^{-1,2}(\Omega))$; the detailed computations can be imitated from \cite[Section~3.1.4.1]{Forster} and the fact that $$\|(\nabla H)^TM\|_{L^{2}(0,T;L^{3/2}(\Omega))}\leq C\|\nabla H\|_{L^{2}(Q_{T})}\|M\|_{L^{\infty}(0,T;L^{6}(\Omega))}.$$ 
    Hence, the map
    $$t\mapsto \int_{0}^{t}\left((v\cdot\nabla)v-\nu\Delta v+\dvr(\nabla M\odot\nabla M)-\dvr(FF^{T})-(\nabla H)^TM\right)d\tau$$
    is absolutely continuous. This implies that the function $G$ is continuous in $[0,T]$ with values in $W^{-1,2}(\Omega).$ Now, by de Rham’s theorem,
    for all $t\in[0,T],$ there exists a unique $\pi(t)\in L^{2}(\Omega)$ with $\int_{\Omega}\pi(t)=0$ such that
    \begin{align}\label{Gpit}
       G(t)=-\nabla\pi(t) 
    \end{align}
    holds in the sense of distributions. 
    
    Once again, following the arguments from \cite[Section~1.5, p.~369]{Boyer}, we show that  $\pi(t)\in C_{w}([0,T];L^{2}(\Omega)).$ In particular, the map $t\mapsto\pi(t)$ belongs to $L^{\infty}(0,T;L^{2}_{0}(\Omega))$ where $L^{2}_{0}(\Omega)$ denotes the space of $L^{2}(\Omega)$-functions with average zero. Now, let us introduce the distribution $p=\partial_{t}\pi\in W^{-1,\infty}(0,T;L^{2}_{0}(\Omega)).$ Finally, by taking test functions of the form $\partial_{t}\vartheta$ ($\vartheta\in C^{\infty}_{c}(Q_{T})$) in the weak form of equation \eqref{Gpit}, we easily verify that $(v,p,F,M)$ solves the momentum balance equation \eqref{diffvlin1} in the sense of distributions. 
    
    The uniqueness of the weak solution was established in \cite[Theorem~3]{anjazab}. This means that all assertions are established and thus, the proof of Theorem \ref{weaksolution} is complete.
	\end{proof}
	
	\subsection{Global existence of strong solutions}
	This section is devoted to the strong well-posedness of the model \eqref{diffviscoelastic*}.
	\begin{thm}\label{globalstrong}
		For any $T>0,$ $v_{0}\in W^{1,2}_{0,\dvr}(\Omega),$ $F_{0}\in W^{1,2}_{0}(\Omega),$ $M_{0}\in W^{2,2}_{n}(\Omega)$ and $H\in L^{2}(0,T;W^{1,2}(\Omega)),$ the system \eqref{diffviscoelastic*} has a unique strong solution $(v,p,F,M)$ with the following regularity properties:
		\begin{align}\label{strngsol}
			\left\{ \begin{aligned}
				&  v\in   
				L^{\infty}(0,T;W^{1,2}_{0,\dvr}(\Omega))\cap L^{2}(0,T;W^{2,2}(\Omega))\cap W^{1,2}(0,T;L^{2}(\Omega)),\\
				&  p \in L^2(0,T;W^{1,2}(\Omega)),\,\, \int_{\Omega}p=0, \\
				&  F\in   
				L^{\infty}(0,T;W^{1,2}_{0}(\Omega))\cap L^{2}(0,T;W^{2,2}(\Omega))\cap W^{1,2}(0,T;L^{2}(\Omega)),\\
				&  M\in  
				L^{\infty}(0,T;W^{2,2}_{n}(\Omega))\cap L^{2}(0,T;W^{3,2}(\Omega))\cap W^{1,2}(0,T;W^{1,2}(\Omega)).
			\end{aligned}\right.
		\end{align}
		Furthermore, the following inequality holds:
		\begin{align}\label{strongestimate}
				&\|v(t)\|^{2}_{W^{1,2}(\Omega)}+\|F(t)\|^{2}_{W^{1,2}(\Omega)}+\|M(t)\|^{2}_{W^{2,2}(\Omega)}
				\notag\\
				& +\int_{0}^{t}\left(\|v(\tau)\|^{2}_{W^{2,2}(\Omega)}+\|F(\tau)\|^{2}_{W^{2,2}(\Omega)}+\|M(\tau)\|^{2}_{W^{3,2}(\Omega)}\right)d\tau
				\leq C,
		\end{align}
		for all $t\in[0,T].$ Here, $C$ is a positive constant depending only on $\|v_{0}\|_{W^{1,2}(\Omega)},$ $\|M_{0}\|_{W^{2,2}(\Omega)},$ $\|F_{0}\|_{W^{1,2}(\Omega)},$ $\|H\|_{L^{2}(0,T;W^{1,2}(\Omega))},$ $|\Omega|$ and the final time $T.$ 
	\end{thm}
		In order to prove the existence of a unique strong solution we will use a similar strategy as in the proof of the existence of weak solutions (see Theorem \ref{weaksolution}), of course keeping in mind that we need to estimate the unknowns in Sobolev spaces with higher regularity. 
		
		Let $\{\xi_{i} \suchthat  i\in\N \}\subset W^{4,2}(\Omega;\mathbb{R}^{2}) \hookrightarrow C^{2}(\overline{\Omega};\mathbb{R}^{2})$ be an orthonormal basis of $L^{2}_{\dvr}(\Omega)$ and an orthogonal basis of $W^{1,2}_{0,\dvr}(\Omega)$ consisting of eigenfunctions of the Stokes operator. For any $m\in\mathbb{N}$, we define the finite dimensional space
		\begin{align}\label{Hm}
		H_{m}:=\langle \xi_{1},....,\xi_{m}\rangle,
		\end{align}
		along with the orthogonal projection $P_{m}:L^{2}_{\dvr}(\Omega)\to H_{m}.$
		 
		Roughly speaking, the construction of a strong solution $(v,F,M)$ is done in three steps: first, for any $m\in\mathbb{N}$, we construct a local-in-time strong solution $(v^\fix_{m},F^\fix_{m},M^\fix_{m})$ of an approximate system (that is formulated in \eqref{diffviscoelasticdiscrete*}) where $v_{0}$ is replaced by $P_{m}v_{0}$. Here, the function $v^\fix_{m}$ belongs to the set
		\begin{align}\label{Vm0}
			 V_{m}(t^{*}_{0})
			:=\left\{ v(x,t)=\sum\limits_{i=1}^{m}g^{i}_{m}(t)\xi_{i}(x)
			\suchthat 
			\begin{aligned}
			    &\left(\sum\limits_{i=1}^{m}|g^{i}_{m}(t)|^{2}\right)^{\frac{1}{2}}\leq N
			    \,\,\mbox{for}\,\, 0\leq t\leq t^{*}_{0},
			    \\[1ex]
			    & g^{i}_{m}\,\,\mbox{is continuous, and}\,\,g^{i}_{m}(0)=\int_{\Omega} v_{0}(x)\cdot\xi_{i}(x)
			\end{aligned}
			\right\},
		\end{align}
		for $t^{*}_{0}>0$ suitably small, where $N$ is a constant depending only on $\|v_{0}\|_{L^{2}(\Omega)}$ and $m$.
		Next, we show that for each $m$, the local-in-time approximate solution $(v^\fix_{m},F^\fix_{m},M^\fix_{m})$ can be extended onto the whole time interval $(0,T)$.
		In the second step, we derive a priori estimates that are uniform in $m$, and 
		in the third step, we pass to the limit $m\to \infty$ to obtain a strong solution $(v,F,M)$ to \eqref{diffviscoelastic*} in $(0,T).$ 
        Eventually, the uniqueness of the strong solution we constructed follows directly from the uniqueness of weak solutions.   
		
        The following lemma will play a crucial role in the proof of Theorem~\ref{globalstrong}. 
        
		\begin{lem}\label{solveFMm}
			Let $t^{*}_{1}>0$ and let $v_{m}\in L^{\infty}(0,t^{*}_{1};W^{2,\infty}(\Omega))$ satisfy $v_{m}=0$ $a.e.$ on $\Sigma_{t^{*}_{1}}$ and $\dvr\,v_{m}=0$ $a.e.$ on $Q_{t^{*}_{1}}.$ Then, 
			for any $H\in L^{2}(0,t^{*}_{1};W^{1,2}(\Omega))$ and any $(F_{0},M_{0})\in W^{1,2}_{0}(\Omega)\times W^{2,2}_{n}(\Omega),$ the system\vspace{1ex}
			\begin{subequations}
			\label{subsystem}
			\begin{alignat}{2}
			\label{subsystem:1}
			&\partial_{t}F_{m}+(v_{m}\cdot\nabla)F_{m}-\nabla v_{m} F_{m}=\Delta F_{m} 
			&&\mbox{ in } Q_{t^{*}_{1}},\\
			\label{subsystem:2}
			&\partial_{t}M_{m}+(v_{m}\cdot\nabla)M_{m}=\Delta M_{m}-\frac{1}{\alpha^{2}}(|M_{m}|^{2}-1)M_{m}+H
			&&\mbox{ in } Q_{t^{*}_{1}},\\
			\label{subsystem:3}
			&  \partial_{n}M_{m}=0,\ F_{m}=0
			&&\mbox{ on } \Sigma_{t^{*}_{1}},\\
			\label{subsystem:4}
			&(M_{m},F_{m})(\cdot,0)=(M_{0},F_{0})
			&& \mbox{ in }\Omega
			\end{alignat}
			\end{subequations}
			has a 
			weak solution satisfying the estimates
			\begin{align}
			\label{weakest:1}
			\|F_{m}\|_{L^{\infty}(0,t^{*}_{1};L^{2}(\Omega))}+\|F_{m}\|_{L^{2}(0,t^{*}_{1};W^{1,2}_{0}(\Omega))}+\|\partial_{t}F_{m}\|_{L^{2}(0,t^{*}_{1};W^{-1,2}(\Omega))}
			&\leq C(v_{m}),\\[1ex]
			\label{weakest:2}
			\|M_{m}\|_{L^{\infty}(0,t^{*}_{1};L^{2}(\Omega))}+\|M_{m}\|_{L^{4}(0,t^{*}_{1};L^{4}(\Omega))}+\|M_{m}\|_{L^{2}(0,t^{*}_{1};W^{1,2}(\Omega))}
			&\leq C,\\[1ex]
			\begin{split}
			\label{weakest:3}
			\|M_{m}\|_{L^{\infty}(0,t^{*}_{1};W^{1,2}(\Omega))}+\|M_{m}\|_{L^{2}(0,t^{*}_{1};W^{2,2}_{n}(\Omega))}
			+\|M_{m}\|_{H^{1}(0,t^{*}_{1};L^{2}(\Omega))} 
			\\
			+\|M_{m}\|_{L^{\infty}(0,t^{*}_{1};L^{4}(\Omega))}
			&\leq C(v_{m}),
			\end{split}
			\end{align}
			Furthermore, this weak solution is actually a strong solution and the following estimates hold:
			\begin{subequations}
			\label{strngestimates}
			\begin{align}
			\label{strngestimates:1}
			\|F_{m}\|_{L^{2}(0,t^{*}_{1};W^{2,2}(\Omega))}+\|F_{m}\|_{L^{\infty}(0,t^{*}_{1};W^{1,2}_{0}(\Omega))}
			&\leq C(v_{m}),\\
			\label{strngestimates:2}
			\|M_{m}\|_{L^{2}(0,t^{*}_{1};W^{3,2}(\Omega))}+\|M_{m}\|_{L^{\infty}(0,t^{*}_{1};W^{2,2}_{n}(\Omega))}
			&\leq C(v_{m}).
			\end{align}
			\end{subequations}
			In the above estimates, $C$ and $C(v_m)$ are generic positive constants. The constant $C$ depends only on $\norm{H}_{L^2(0,t_1^*;W^{1,2}(\Omega))}$ and the initial data, whereas $C(v_m)$ depends only on $\norm{H}_{L^2(0,t_1^*;W^{1,2}(\Omega))}$, the initial data, and $\norm{v_m}_{L^\infty(0,t_1^*;W^{2,\infty}(\Omega))}$.
		\end{lem} 
		\begin{proof}[Proof of Lemma \ref{solveFMm}] 
			We observe that the equations \eqref{subsystem:1} and \eqref{subsystem:2} are decoupled and can thus be solved independently.
			
			\textit{Step 1: Construction of a solution to \eqref{subsystem:2}.} 
			Let ${\{\eta_{i}\suchthat i\in\N\}}$ be an orthonormal basis of $L^{2}(\Omega)$ and an orthogonal basis of $W^{1,2}(\Omega)$. The functions $\eta_i$ can be chosen, for instance, as $L^{2}(\Omega)$-normalized eigenfunctions to the eigenvalues 
			$0<\mu_{1}\leq\mu_{2}\leq...$
			of the eigenvalue problem
			\begin{align*}\nonumber
			\left\{ \begin{aligned}
			-\Delta\eta +\eta &=\mu \eta\,\,&&\mbox{in}\,\,\Omega,\\
			\partial_{n}\eta &=0\,\,&&\mbox{on}\,\,\partial\Omega.
			\end{aligned}\right.
			\end{align*}
			We further define the operator $\widetilde{P}_{n}:L^{2}(\Omega)\to \langle\eta_{1},...,\eta_{n}\rangle$
			as the orthogonal projection onto the finite dimensional linear subspace $\langle\eta_{1},...,\eta_{n}\rangle$.
			Following the arguments of \cite[p.~62]{Forster} we prove the existence of a time $t^{*}_{2}\in (0, t^{*}_{1}]$ and coefficient functions 
			$h_n^i:[0,t_2^*)\to \R$
			such that the ansatz function
			$$M_{m}^{n}(x,t)=\sum\limits_{i=1}^{n}h^{i}_{n}(t)\eta_{i}(x), \quad x\in\Omega,\; t\in [0,t_2^*)$$
            is a solution of the system
			\begin{subequations}
			\label{Mnm}
			\begin{alignat}{2}
			\label{Mnm:1}
			&\partial_{t}M^{n}_{m}=\widetilde{P}_{n}\left[-(v_{m}\cdot\nabla)M^{n}_{m}+\Delta M^{n}_{m}-\frac{1}{\alpha^{2}}(|M^{n}_{m}|^{2}-1)M^{n}_{m}+H\right]
			&&\mbox{ in } Q_{t^{*}_{2}},\\
			\label{Mnm:2}
			&  \partial_{n}M^{n}_{m}=0,
			&&\mbox{ on } \Sigma_{t^{*}_{2}},\\
			\label{Mnm:3}
			& M^{n}_{m}(\cdot,0)=\widetilde{P}_{n}M_{0} 
			&&\mbox{ in }\Omega.
			\end{alignat}
			\end{subequations}
			
			In order to pass to the limit $n\to\infty,$ we need uniform estimates of $M^{n}_{m}.$ It is shown in \cite[pp.~55--61]{Forster} that $M^{n}_{m}$ fulfills the estimates 
			\begin{align}
			\label{weakest:2*}
			\|M_{m}^n\|_{L^{\infty}(0,t^{*}_{2};L^{2}(\Omega))}+\|M_{m}^n\|_{L^{4}(0,t^{*}_{2};L^{4}(\Omega))}+\|M_{m}^n\|_{L^{2}(0,t^{*}_{2};W^{1,2}(\Omega))}
			&\leq C,\\[1ex]
			\label{weakest:3*}
			\|M_{m}^n\|_{L^{\infty}(0,t^{*}_{2};W^{1,2}(\Omega))}+\|M_{m}^n\|_{L^{2}(0,t^{*}_{2};W^{2,2}_{n}(\Omega))}\quad
			 +\|M_{m}^n\|_{H^{1}(0,t^{*}_{2};L^{2}(\Omega))}
			 &
			 \notag\\
			 +\|M_{m}^n\|_{L^{\infty}(0,t^{*}_{2};L^{4}(\Omega))}
			&\leq C(v_{m}).
			\end{align}
			Recall that the constants $C$ and $C(v_{m})$ may depend on  
            $\norm{H}_{L^2(0,t_1^*;W^{1,2}(\Omega))}$ and the initial data but not on $n$ or $t_2^*$.
			Since $M^{n}_{m}$ is absolutely continuous in $[0,t^{*}_{2})$ and the constants in the right-hand sides of \eqref{weakest:2*} and \eqref{weakest:3*} are independent of $t^{*}_{2}$,
			the solution $M_{m}^{n}$ can be extended onto the whole time interval $[0,t_1^*)$ and the estimates \eqref{weakest:2*} and \eqref{weakest:3*} hold true with $t_*^1$ instead of $t_2^*$.
			Hence, by the Banach--Alaoglu theorem, we infer that the sequence $(M_m^n)_{n\in\N}$ converges to a limit function $M_m$ in the weak-$^*$ sense, at least after an extraction of a subsequence. In particular, the limit $M_m$ satisfies the estimates \eqref{weakest:2} and \eqref{weakest:3}.
			This further allows us to pass to the limit in the weak formulation of \eqref{Mnm} which proves that the limit function $M_{m}$ is a weak solution of \eqref{subsystem:2}.
			
			To further prove the estimate \eqref{strngestimates:2}, we first derive an analogous estimate for the approximate solutions $M_m^n$. To this end,
			test \eqref{Mnm:1} by $\Delta^{2}M^{n}_{m}$, and we use $\partial_{n}M^{n}_{m}=\partial_{n}\Delta M^{n}_{m}=0$ on $\partial\Omega$ (since $\partial_{n}\eta_{i}=\partial_{n}\Delta\eta_{i}=0$ on $\partial\Omega$) along with an integration by parts to obtain
			\begin{align}
			\label{afttesteprob}
			&\frac{1}{2}\frac{\mathrm d}{\mathrm dt}\|\Delta M^{n}_{m}\|^{2}_{L^{2}(\Omega)}+\|\nabla\Delta M^{n}_{m}\|^{2}_{L^{2}(\Omega)}
			\notag\\
			&= \int_{\Omega} \nabla \big[(v_{m}\cdot\nabla) M^{n}_{m}\big] \cdot \nabla\Delta M^{n}_{m} 
			+ \frac{1}{\alpha^2} \int_{\Omega} \nabla\left[\big(|M_m^n|^2-1\big) M_m^n \right]\cdot \nabla \Delta M^{n}_{m} 
			+ \int_{\Omega} (- \nabla H) \cdot \nabla \Delta M^{n}_{m} 
			\notag\\
			&=:\sum\limits_{i=1}^{3}J_{i}.
			\end{align}
			The next step is to estimate the terms $J_{i}$, $i=1,...,3$. For any $\eps>0$, we obtain
			\begin{align}\label{J1}
			|J_{1}| &
			=\left|\int_{\Omega}\nabla \big[(v_{m}\cdot\nabla)M^{n}_{m}\big] \cdot \nabla \Delta M^{n}_{m}\right|
			\notag\\
			& \leq  \|\nabla\Delta M^{n}_{m}\|_{L^{2}(\Omega)}\left(\|\nabla v_{m}\|_{L^{4}(\Omega)}\|\nabla M^{n}_{m}\|_{L^{4}(\Omega)}+\|v_{m}\|_{L^{\infty}(\Omega)}\|M^{n}_{m}\|_{W^{2,2}(\Omega)}\right)
			\notag\\
			& \leq \eps\|\nabla\Delta M^{n}_{m}\|^{2}_{L^{2}(\Omega)}
			+C \|\nabla v_{m}\|^{2}_{L^{4}(\Omega)}\|\nabla M^{n}_{m}\|^{2}_{L^{4}(\Omega)}
			+C\|v_{m}\|^{2}_{L^{\infty}(\Omega)}\|M^{n}_{m}\|^{2}_{W^{2,2}(\Omega)}
			\notag\\
			& \leq \eps\|\nabla\Delta M^{n}_{m}\|^{2}_{L^{2}(\Omega)}+C\|\nabla v_{m}\|^{2}_{L^{4}(\Omega)}\|\nabla M^{n}_{m}\|_{L^{2}(\Omega)}\left(\|\nabla M^{n}_{m}\|^{2}_{L^{2}(\Omega)}+\|\Delta M^{n}_{m}\|^{2}_{L^{2}(\Omega)}\right)^{\frac{1}{2}}
			\notag\\
			&\qquad +C\| v_{m}\|^{2}_{L^{\infty}(\Omega)}\left(\|\Delta M^{n}_{m}\|^{2}_{L^{2}(\Omega)}+\|M^{n}_{m}\|^{2}_{L^{2}(\Omega)}\right)
			\notag\\
			& \leq \eps\|\nabla\Delta M^{n}_{m}\|^{2}_{L^{2}(\Omega)}+C\|\nabla v_{m}\|^{2}_{L^{4}(\Omega)}\left(\|\nabla M^{n}_{m}\|^{2}_{L^{2}(\Omega)}+\|\Delta M^{n}_{m}\|^{2}_{L^{2}(\Omega)}\right)
			\notag\\
			&\qquad+C\| v_{m}\|^{2}_{L^{\infty}(\Omega)}\left(\|\Delta M^{n}_{m}\|^{2}_{L^{2}(\Omega)}+\|M^{n}_{m}\|^{2}_{L^{2}(\Omega)}\right).
			\end{align}
			Here, from the third to the fourth line, we used the interpolation inequality \eqref{smoreinterpole2} to estimate $\|\nabla M^{n}_{m}\|^{2}_{L^{4}(\Omega)}$. To obtain the last two lines, we applied Young's inequality on the term 
			$$\|\nabla M^{n}_{m}\|_{L^{2}(\Omega)}\left(\|\nabla M^{n}_{m}\|^{2}_{L^{2}(\Omega)}+\|\Delta M^{n}_{m}\|^{2}_{L^{2}(\Omega)}\right)^{\frac{1}{2}}.$$
			Next, we estimate $J_{2}$ and $J_{3}$. For any $\eps>0$, we get
			\begin{align}
			\label{J2}
			 |J_{2}|& \leq C\int_{\Omega}\left|\nabla\left[\left(|M^{n}_{m}|^{2}-1\right)M^{n}_{m}\right]\right||\nabla\Delta M^{n}_{m}|
			\notag\\
			& \leq \eps\|\nabla\Delta M^{n}_{m}\|^{2}_{L^{2}(\Omega)}+C\left(\|\nabla M^{n}_{m}\|^{2}_{L^{6}(\Omega)}\|M^{n}_{m}\|^{4}_{L^{6}(\Omega)}+\|\nabla M^{n}_{m}\|^{2}_{L^{2}(\Omega)}\right)
			\end{align}
			and
			\begin{align}
			\label{J3}
			|J_{3}|
			    \leq \eps\|\nabla\Delta M^{n}_{m}\|^{2}_{L^{2}(\Omega)} 
			    + C \|\nabla H\|^{2}_{L^{2}(\Omega)}
			\end{align}
			by means of Young's inequality.
			Choosing $\eps$ sufficiently small, using \eqref{J1}--\eqref{J3} to bound the right-hand side of \eqref{afttesteprob}, applying Gronwall's inequality and invoking the weak estimates \eqref{weakest:2*} and \eqref{weakest:3*}, we find that
			\begin{align}\label{ustestMnm}
			 \|M^{n}_{m}\|_{L^{2}(0,t^{*}_{1};W^{3,2}(\Omega))}+\|M^{n}_{m}\|_{L^{\infty}(0,t^{*}_{1};W^{2,2}(\Omega))}\leq C(v_{m}).
			\end{align}
			Since $C(v_{m})$ is independent of $n$, we conclude that $M_m$ satisfies the estimate \eqref{strngestimates:2} by passing to the limit $n\to\infty$. In particular, this proves that the weak solution $M_m$ of \eqref{subsystem:2} is actually strong.
			
			\textit{Step 2: Construction of a solution to \eqref{subsystem:1}.}
			Let ${\{\zeta_{i}\suchthat i\in\N \}}$ be an orthonormal basis of $L^{2}(\Omega)$ and an orthogonal basis of $W^{1,2}_0(\Omega)$. Here, the functions $\eta_i$ can be chosen, for instance, as $L^{2}(\Omega)$-normalized eigenfunctions to the eigenvalues 
			$0<\mu_{1}\leq\mu_{2}\leq...$
			of the eigenvalue problem
			\begin{align*}\nonumber
			\left\{ \begin{aligned}
			-\Delta\zeta  &=\mu \zeta\,\,&&\mbox{in}\,\,\Omega,\\
			\zeta &=0\,\,&&\mbox{on}\,\,\partial\Omega.
			\end{aligned}\right.
			\end{align*}
			We further define the operator $\overline{P}_{n}:L^{2}(\Omega)\to \langle\zeta_{1},...,\zeta_{n}\rangle$
			as the orthogonal projection onto the finite dimensional linear subspace $\langle\zeta_{1},...,\zeta_{n}\rangle$.
			
			Proceeding as in \cite[pp.~50--53]{Forster}, we prove the existence of a time $t^{*}_{3} \in (0, t^{*}_{1}]$ and coefficient functions 
			${d_n^i:[0,t_3^*)\to \R}$
			such that the ansatz function
			$$F^{n}_{m}(x,t)=\sum\limits_{i=1}^{n}d^{i}_{n}(t)\zeta_{i}(x),\quad x\in\Omega,\, t\in [0,t_e^*)$$
			is a solution of the system
			\begin{subequations}
			\label{Fnm}
			\begin{align}
			\label{Fnm:1}
			&\partial_{t}F^{n}_{m}=\overline{P}_{n}\big[-(v_{m}\cdot\nabla)F^{n}_{m}+\nabla v_{m} F^{n}_{m}+\Delta F^{n}_{m} \big] &\mbox{ in } Q_{t^{*}_{3}},\\
			\label{Fnm:2}
			&   F^{n}_{m}=0&\mbox{ on } \Sigma_{t^{*}_{3}},\\
			\label{Fnm:3}
			& F^{n}_{m}(\cdot,0)=\overline{P}_{n}F_{0}& \mbox{ in }\Omega,
			\end{align}
			\end{subequations}
			which satisfies the estimate
			\begin{align}
			\label{weakest:1*}
			\|F_{m}^n\|_{L^{\infty}(0,t^{*}_{3};L^{2}(\Omega))}+\|F_{m}^n\|_{L^{2}(0,t^{*}_{3};W^{1,2}_{0}(\Omega))}+\|\partial_{t}F_{m}^n\|_{L^{2}(0,t^{*}_{3};W^{-1,2}(\Omega))}
			&\leq C(v_{m}).
			\end{align}
			Since the constant $C(v_{m})$ is independent of $n$ and $t_3^*$,
			we can thus argue as above to conclude that the solution $F_m^n$ can be extended onto the whole time interval $[0,t_1^*)$ and in particular, the uniform estimate \eqref{weakest:1*} holds true with $t_1^*$ instead of $t_3^*$.
			Using the Banach--Alaoglu theorem, we infer that the sequence (at least a subsequence of) $(F^{n}_{m})_{n\in\N}$ converges to a function $F_m$ in the weak-$^*$ sense, and the limit 
			$F_m$ satisfies estimate \eqref{weakest:1}.
			By passing to the limit $n\to\infty$ in the weak formulation of \eqref{Fnm}, we conclude that $F_m$ is a weak solution of \eqref{subsystem:1}.
			
		    In order to show that the obtained weak solution $F_{m}$ satisfies the bound \eqref{strngestimates:1}, we will first establish an analogous estimate for the approximate solutions $F^{n}_{m}$. Therefore, we test \eqref{Fnm:1}  by $-\Delta F^{n}_{m}$. This yields
			\begin{align}
			\label{estFnm}
			\frac{1}{2}\frac{\mathrm d}{\mathrm dt}\|\nabla F^{n}_{m}\|^{2}_{L^{2}(\Omega)}+\|\Delta F^{n}_{m}\|^{2}_{L^{2}(\Omega)}
			&= \int_{\Omega}(v_{m}\cdot\nabla)F^{n}_{m}\Delta F^{n}_{m} + \int\limits_{\Omega}\nabla v_{m}F^{n}_{m}(-\Delta F^{n}_{m})
			\notag \\
			&=:\sum\limits_{i=1}^{2} \mathcal{J}_{i}.
			\end{align}
			For any $\eps>0$, we obtain 
			\begin{align}\label{mathcalJ1}
			|\mathcal{J}_{1}|&=\left|\int_{\Omega}(v_{m}\cdot\nabla)F^{n}_{m}\Delta F^{n}_{m}\right| \leq\eps\|\Delta F^{n}_{m}\|^{2}_{L^{2}(\Omega)}+C\|v_{m}\|^{2}_{L^{\infty}(\Omega)}\|\nabla F^{n}_{m}\|^{2}_{L^{2}(\Omega)}
			\end{align}
			and 
			\begin{align}\label{mathcalJ2}
			|\mathcal{J}_{2}|&=\left|\int\limits_{\Omega}\nabla v_{m}F^{n}_{m}\Delta F^{n}_{m}\right|\leq \eps\|\Delta F^{n}_{m}\|^{2}_{L^{2}(\Omega)}+C\|v_{m}\|^{2}_{W^{1,\infty}(\Omega)}\|F^{n}_{m}\|^{2}_{L^{2}(\Omega)}
			\end{align}
			by means of Young's inequality.
			Choosing $\eps$ sufficiently small, using \eqref{mathcalJ1}--\eqref{mathcalJ2} to estimate the right-hand side of \eqref{estFnm}, using Gronwall's inequality and invoking \eqref{weakest:1*}, we find that
			\begin{align}\label{uestFnm}
			 \|F^{n}_{m}\|_{L^{2}(0,t^{*}_{1};W^{2,2}(\Omega))}+\|F^{n}_{m}\|_{L^{\infty}(0,t^{*}_{1};W^{1,2}_{0}(\Omega))}\leq C(v_{m}).
			\end{align} 
			Since $C(v_m)$ is independent of $n$, we conclude that $F_m$ satisfies the estimate \eqref{strngestimates:1} by passing to the limit $n\to\infty$.
			In particular, this proves that the weak solution $F_m$ of \eqref{subsystem:1} is actually strong.
			
			This means that all assertions are established and thus, the proof is complete.
\end{proof}
		
\begin{proof}[Proof of Theorem \ref{globalstrong}]	
    The proof is split into three steps.
    
    \textit{Step 1. Construction of an approximate solution.} 
    We fix an arbitrary $m\in\N$. Let $t_0^*\in (0,T]$ be some time that will be adjusted later. 
    For any function $v_m \in V_{m}(t^{*}_{0})$, we consider the system
    \begin{subequations}
        \label{diffviscoelasticdiscrete*}
    \begin{alignat}{2}
    	&\partial_{t}\widetilde{v}_{m}=P_{m}\big[-(\widetilde{v}_{m}\cdot\nabla)\widetilde{v}_{m}-\mbox{div}\left((\nabla M_{m}\odot\nabla M_{m})-F_{m}F_{m}^{T}\right) \notag\\
    	\label{diffviscoelasticdiscrete*:1}
    	&\qquad\qquad -\nabla p_{m}+\nu\Delta \widetilde{v}_{m}+(\nabla H)^T M_{m}\big]
    	&& \mbox{ in } Q_{t^{*}_{0}},\\
    	\label{diffviscoelasticdiscrete*:2}
    	&\dvr \widetilde{v}_{m}= 0 
    	&&\mbox{ in } Q_{T},\\
    	\label{diffviscoelasticdiscrete*:3}
    	&\partial_{t}F_{m}+(v_{m}\cdot\nabla)F_{m}-\nabla v_{m} F_{m}=\Delta F_{m} 
    	&&\mbox{ in } Q_{t^{*}_{0}},\\
    	\label{diffviscoelasticdiscrete*:4}
    	&\partial_{t}M_{m}+(v_{m}\cdot\nabla)M_{m}=\Delta M_{m}-\frac{1}{\alpha^{2}}(|M_{m}|^{2}-1)M_{m}+H
    	&&\mbox{ in } Q_{t^{*}_{0}},\\
    	\label{diffviscoelasticdiscrete*:5}
    	& \widetilde{v}_{m}=0,\ \partial_{n}M_{m}=0,\ F_{m}=0
    	&&\mbox{ on } \Sigma_{t^{*}_{0}},\\
    	\label{diffviscoelasticdiscrete*:6}
    	&(\widetilde{v}_{m},M_{m},F_{m})(\cdot,0)=(P_{m}v_{0},M_{0},F_{0}).
    	&& \mbox{ in }\Omega
    \end{alignat}
    \end{subequations}
    
    Let $v_m \in V_{m}(t^{*}_{0})$ be arbitrary.
    Invoking Lemma~\ref{solveFMm}, we infer the existence of a strong solution $(F_{m},M_{m})$ to the subsystem \eqref{diffviscoelasticdiscrete*:3}--\eqref{diffviscoelasticdiscrete*:4} subject to the corresponding initial and boundary conditions stated in \eqref{diffviscoelasticdiscrete*:5} and \eqref{diffviscoelasticdiscrete*:6}. Choosing $t_0^*$ sufficiently small, we proceed as in \cite[pp.~63--64]{Forster} to construct a solution $\widetilde{v}_m \in V_{m}(t^{*}_{0})$ of \eqref{diffviscoelasticdiscrete*:1} written for $v_m$ and the pair $(F_{m},M_{m})$ we just constructed. In summary, we have just obtained a local-in-time solution $(v_{m},F_{m},M_{m})$ of the system \eqref{diffviscoelasticdiscrete*} to the given function $v_m$ existing on the time interval $[0,t_0^*)$.
    
    Since $v_m \in V_{m}(t^{*}_{0})$ was arbitrary, we can define an operator $\mathfrak S_m:V_{m}(t^{*}_{0})\to V_{m}(t^{*}_{0}), \mathfrak S_m(v_m) := \widetilde{v}_m$ that maps any given function $v_m \in V_{m}(t^{*}_{0})$ onto the component $\widetilde{v}_m \in V_{m}(t^{*}_{0})$ of the corresponding solution $(\widetilde{v}_{m},F_{m},M_{m})$ that is constructed as described above.
    
    Proceeding as in \cite[p.~64]{Forster}, and choosing $t_0^*$ as small as necessary, we apply Schauder's fixed point theorem to prove the existence of a fixed point $v_m^\fix$ of the operator $\mathfrak S_m$. By means of Lemma~\ref{solveFMm}, we can find an associated pair $(F^\fix_{m},M^\fix_{m})$ such that the triplet 
    \begin{align}\nonumber
    		(v^\fix_{m},F^\fix_{m},M^\fix_{m})&\in V_{m}(t^{*}_{0})\times \left(L^{2}(0,t^{*}_{0};W^{2,2}(\Omega))\cap L^{\infty}(0,t^{*}_{0};W^{1,2}_{0}(\Omega))\right)\\
    		&\quad \times\left(L^{2}(0,t^{*}_{0};W^{3,2}(\Omega))\cap L^{\infty}(0,t^{*}_{0};W^{2,2}_{n}(\Omega))\right)
    \end{align}
    is a weak solution to system \eqref{diffviscoelasticdiscrete*} written for $v_m = v_m^\fix$ on the time interval $[0,t_0^*)$.
		
	Following the line of argument in \cite[Sections 3.1.3.3--3.1.3.4]{Forster}, we infer that the solution $(v^\fix_{m},F^\fix_{m},M^\fix_{m})$ can be extended onto the whole time interval $[0,T]$. Moreover, it is easy to see that the functions $v^\fix_{m}$, $F^\fix_{m}$ and $M^\fix_{m}$ satisfy the associated regularities stated in \eqref{funframeweak}.
	
	Since, by construction, $v^\fix_{m}$ is smooth with respect to the space variables, we further conclude that $v^\fix_{m}\in L^{\infty}(0,T;W^{2,\infty}(\Omega))$.  As a consequence of Lemma \ref{solveFMm}, the pair $(F^\fix_{m},M^\fix_{m})$ fulfills the estimates \eqref{strngestimates}. This means that the functions $\mathcal{S} v^\fix_{m}$ (where $\mathcal S$ denotes the Stokes operator, cf.~\eqref{Stokesmulti}), $\Delta F^\fix_{m}$ and $\nabla\Delta M^\fix_{m}$ are well defined $a.e.$ on $\Omega\times (0,T)$.
	
	\textit{Step 2. Uniform a priori estimates.} 
	The next step is to derive a priori estimates on the approximate solution $(v^\fix_{m},F^\fix_{m},M^\fix_{m})$ that are uniform with respect to the approximation index $m\in\N$. Eventually, this will allow us to pass to the limit $m\to\infty$ to obtain a strong solution of \eqref{diffviscoelastic*}. 
	
	We claim that
	\begin{align}\label{diffinq}
			\frac{\mathrm d}{\mathrm dt}\mathcal{A}(t)+\mathcal{B}(t)\leq C\left(\mathcal{A}^2(t)+\|\nabla H\|^{2}_{L^{2}(\Omega)}\mathcal{A}(t)+\| H\|^{2}_{W^{1,2}(\Omega)}\right),
	\end{align}
	with
	\begin{align}\label{defs}
			\mathcal{A}(t)&:=\|\nabla v^\fix_{m}(t)\|^{2}_{L^{2}(\Omega)}+\|\nabla F^\fix_{m}(t)\|^{2}_{L^{2}(\Omega)}+\|(\Delta M^\fix_{m}(t)-f(M^\fix_{m}(t))\|^{2}_{L^{2}(\Omega)},\\
			\mathcal{B}(t)&:=\|\mathcal{S}v^\fix_{m}(t)\|^{2}_{L^{2}(\Omega)}+\|\Delta F^\fix_{m}(t)\|^{2}_{L^{2}(\Omega)}+\|\nabla(\Delta M^\fix_{m}(t)-f(M^\fix_{m}(t)))\|^{2}_{L^{2}(\Omega)},
	\end{align}	
	where $\mathcal{S}$ denotes the Stokes operator, $i.e.$,
	\begin{align}\label{Stokesmulti}
			\mathcal{S}v^\fix_{m} :=-\nu\Delta v^\fix_{m}+\nabla p_{m}\in H_{m},
	\end{align}
	and 
	\begin{align}\label{fM}
	        f:\R^2\to\R^2,\quad
			M\mapsto \frac{1}{\alpha^{2}}(|M|^{2}-1)M
	\end{align}
	We point out that the Stokes operator $\mathcal S$ comes into play since for deriving strong a priori estimates, we can use $\mathcal{S}v^\fix_{m}$ as a test function in \eqref{diffviscoelasticdiscrete*:1} but not $-\nu\Delta v^\fix_{m}$ as this function is not necessarily in $H_{m}.$
	
    Using the equations of system \eqref{diffviscoelasticdiscrete*} as well as the identity
    \begin{align}\label{ptwiseidentity}
			\mbox{div}(\nabla M\odot\nabla M)=\frac{1}{2}\nabla|\nabla M|^{2}+(\nabla M)^{T}\Delta M,
	\end{align}
	we derive the following equation:
	\begin{align}\label{computenergy}
			&\frac{1}{2}\frac{\mathrm d}{\mathrm dt}\mathcal{A}(t)+\int_{\Omega}|\mathcal{S} v^\fix_{m}|^{2}+\int_{\Omega}|\Delta F^\fix_{m}|^{2}+\int_{\Omega}|\nabla\left(\Delta M^\fix_{m}-f(M^\fix_{m})\right)|^{2}
			\notag\\
			&=\int_{\Omega}(v^\fix_{m}\cdot\nabla)v^\fix_{m}\cdot\mathcal{S} v^\fix_{m}+\int_{\Omega}(\nabla M^\fix_{m})^{T}\Delta M^\fix_{m}\cdot\mathcal{S} v^\fix_{m}-\int_{\Omega}((\nabla H)^T M^\fix_{m})\cdot\mathcal{S} v^\fix_{m}
			\notag\\
			&\quad+\int_{\Omega}(v^\fix_{m}\cdot\nabla)F^\fix_{m}\cdot\Delta F^\fix_{m}
			-\int_{\Omega}\nabla v^\fix_{m}F^\fix_{m}\cdot\Delta F^\fix_{m}
			 +\int_{\Omega}\nabla(v^\fix_{m}\cdot\nabla)M^\fix_{m}\cdot\nabla(\Delta M^\fix_{m}-f(M^\fix_{m}))
			\notag\\
			&\quad-\int_{\Omega}\nabla H \cdot \left(\nabla(\Delta M^\fix_{m}-f(M^\fix_{m}))\right)-\int_{\Omega}\partial_{t}f(M^\fix_{m})(\Delta M^\fix_{m}-f(M^\fix_{m}))
			\notag\\
			& =: \sum_{i=1}^{8}I_{i}.
	\end{align}

	In the following we will estimate the terms $I_1,...,I_8$ appearing in the right-hand side of \eqref{computenergy}. The estimates will be performed for almost every $t\in[0,T]$ but for the simplicity of notation, we avoid writing the explicit dependence of the functions on $t$. We already know from \cite[Section 3.1.3.4]{Forster} that the triplet $(v^\fix_{m},F^\fix_{m},M^\fix_{m})$ fulfills \eqref{weakestimate}. In the following, the letter $C$ denotes generic positive constants that may depend on the norms $\|v^\fix_{m}\|_{L^{2}(\Omega)},$ $\|F^\fix_{m}\|_{L^{2}(\Omega)}$ and $\|M^\fix_{m}\|_{W^{1,2}(\Omega)}$ (which are bounded uniformly on $[0,T]$), and may change its value from line to line.
	
	Recalling the definition of $f$ in \eqref{fM}, we first observe that
	\begin{align}\label{estfM}
	\|f(M^\fix_{m})\|_{L^{2}(\Omega)}\leq C\|M^\fix_{m}\|^{3}_{L^{6}(\Omega)}+C\|M^\fix_{m}\|_{L^{2}(\Omega)}\leq C 
	\quad \text{$a.e.$ on $[0,T]$}.
	\end{align}
	Furthermore, by regularity theory for the Stokes operator (see \cite[Theorem~IV.5.8]{Boyer}), we have
	\begin{align}\label{SmajorDelta}
			\|v^\fix_{m}\|_{H^{2}(\Omega)}\leq C\|\mathcal{S}v^\fix_{m}\|_{L^{2}(\Omega)}.
	\end{align}
	Form now on until the end of this proof we will frequently use this inequality without mentioning it explicitly.
	In the following, let $\eps>0$ be some real number that will be fixed later. The constants $C$ are now also allowed to depend on $\eps$.
	
	Using \eqref{smoreinterpole1} to
	estimate $\|\nabla v^\fix_{m}\|_{L^{4}(\Omega)}$, \eqref{L40bnd} to estimate $\|v^\fix_{m}\|_{L^{4}(\Omega)}$ and employing the inequality $\|\nabla^{2}v^\fix_{m}\|_{L^{2}(\Omega)}\leq C\|\Delta v^\fix_{m}\|_{L^{2}(\Omega)}$ (which holds since $v^\fix_{m}$ has trace zero at the boundary),
	we obtain the following estimate for the term $I_1$:
	\begin{align}\label{estI1}
			|I_{1}|&\leq \|\mathcal{S}v^\fix_{m}\|_{L^{2}(\Omega)}\|v^\fix_{m}\|_{L^{4}(\Omega)}\|\nabla v^\fix_{m}\|_{L^{4}(\Omega)}
			\notag\\
			& \leq C\|\mathcal{S} v^\fix_{m}\|_{L^{2}(\Omega)}\left(\|v^\fix_{m}\|_{L^{2}(\Omega)}^{\frac{1}{2}}\|\nabla v^\fix_{m}\|^{\frac{1}{2}}_{L^{2}(\Omega)}\right)\left(\|\nabla v^\fix_{m}\|_{L^{2}(\Omega)}+\|\Delta v^\fix_{m}\|^{\frac{1}{2}}_{L^{2}(\Omega)}\|\nabla v^\fix_{m}\|^{\frac{1}{2}}_{L^{2}(\Omega)}\right)
			\notag\\
			& \leq C\|\mathcal{S}v^\fix_{m}\|_{L^{2}(\Omega)}\|\nabla v^\fix_{m}\|^{\frac{3}{2}}_{L^{2}(\Omega)}+C\|\mathcal{S} v^\fix_{m}\|^{\frac{3}{2}}_{L^{2}(\Omega)}\|\nabla v^\fix_{m}\|_{L^{2}(\Omega)}
			\notag\\
			& \leq \eps \|\mathcal{S} v^\fix_{m}\|^{2}_{L^{2}(\Omega)}+C\|\nabla v^\fix_{m}\|^{4}_{L^{2}(\Omega)}+C.
	\end{align}

	We now derive two inequalities (namely \eqref{interpol1} and \eqref{interpol2}) that will be used in the subsequent approach, especially to estimate $I_2$.
	For almost every $t\in[0,T]$, we obtain
	\begin{align}\label{interpol1}
			\|\nabla M^\fix_{m}\|^{2}_{L^{4}(\Omega)}&\leq C\big(\|\nabla M^\fix_{m}\|^{2}_{L^{2}(\Omega)}+\|\Delta M^\fix_{m}\|^{2}_{L^{2}(\Omega)}\big)
			\notag\\
			& \leq C\|\Delta M^\fix_{m}\|_{L^{2}(\Omega)}+C
			\notag\\
			& \leq C\|(\Delta M^\fix_{m}-f(M^\fix_{m}))\|_{L^{2}(\Omega)}+C\|f(M^\fix_{m})\|_{L^{2}(\Omega)}+C
			\notag\\
			& \leq C\|(\Delta M^\fix_{m}-f(M^\fix_{m}))\|_{L^{2}(\Omega)}+C.
	\end{align}
	Here, we used \eqref{smoreinterpole2} and Young's inequality to estimate $\|\nabla M^\fix_{m}\|^{2}_{L^{2}(\Omega)}$, and \eqref{estfM} was employed to bound $\|f(M^\fix_{m})\|_{L^{2}(\Omega)}$ by a constant $C.$
	
	Furthermore, using once again \eqref{smoreinterpole1}, we have
	\begin{align}\label{interpol2}
			& \|(\Delta M^\fix_{m}-f(M^\fix_{m}))\|^{2}_{L^{4}(\Omega)}
			\notag\\
			&\leq C\|(\Delta M^\fix_{m}-f(M^\fix_{m}))\|^{2}_{L^{2}(\Omega)}+C\|(\Delta M^\fix_{m}-f(M^\fix_{m}))\|_{L^{2}(\Omega)}\|\nabla\big(\Delta M^\fix_{m}-f(M^\fix_{m})\big)\|_{L^{2}(\Omega)}.
	\end{align}
	Hence, we obtain the following estimate for $I_{2}:$
	\begin{align}\label{estI2}
			|I_{2}|
			&\leq \|\mathcal{S} v^\fix_{m}\|_{L^{2}(\Omega)}\|\nabla M^\fix_{m}\|_{L^{4}(\Omega)}\|(\Delta M^\fix_{m}-f(M^\fix_{m}))\|_{L^{4}(\Omega)}
			\notag\\
			&\quad +\|\mathcal{S} v^\fix_{m}\|_{L^{2}(\Omega)}\|\nabla M^\fix_{m}\|_{L^{4}(\Omega)}\|f(M^\fix_{m})\|_{L^{4}(\Omega)}
			\notag\\
			& \leq\eps\|\mathcal{S} v^\fix_{m}\|^{2}_{L^{2}(\Omega)}+C\|\nabla M^\fix_{m}\|^{2}_{L^{4}(\Omega)}\|(\Delta M^\fix_{m}-f(M^\fix_{m}))\|^{2}_{L^{4}(\Omega)}+ C\|\nabla M^\fix_{m}\|^{2}_{L^{4}(\Omega)}
			\notag\\
			& \leq\eps\|\mathcal{S} v^\fix_{m}\|^{2}_{L^{2}(\Omega)}+C\big(\|(\Delta M^\fix_{m}-f(M^\fix_{m}))\|_{L^{2}(\Omega)}+1\big)
			\notag\\
			&\quad\; \cdot \big(\|(\Delta M^\fix_{m}-f(M^\fix_{m}))\|^{2}_{L^{2}(\Omega)}+\|(\Delta M^\fix_{m}-f(M^\fix_{m}))\|_{L^{2}(\Omega)}\|\nabla\big(\Delta M^\fix_{m}-f(M^\fix_{m})\big)\|_{L^{2}(\Omega)}\big)
			\notag\\
			&\quad+C\|(\Delta M^\fix_{m}-f(M^\fix_{m}))\|_{L^{2}(\Omega)}+C
			\notag\\
			& \leq \eps\|\mathcal{S} v^\fix_{m}\|^{2}_{L^{2}(\Omega)}+\eps\|\nabla(\Delta M^\fix_{m}-f(M^\fix_{m}))\|^{2}_{L^{2}(\Omega)}+C\|(\Delta M^\fix_{m}-f(M^\fix_{m}))\|^{4}_{L^{2}(\Omega)}+C.
	\end{align}
	In this computation, the third line is obtained by using \eqref{interpol1} and \eqref{interpol2}, and the final line is deduced by means of Young's inequality.
	
	Let us now estimate $I_{3}$. Using \eqref{interpolation2} to estimate $\|M^\fix_{m}\|_{L^{\infty}(\Omega)}$, we obtain
	\begin{align}\label{I3}
			|I_{3}|&\leq \eps\|\mathcal{S} v^\fix_{m}\|^{2}_{L^{2}(\Omega)}+C\|M^\fix_{m}\|^{2}_{L^{\infty}(\Omega)}\|\nabla H\|^{2}_{L^{2}(\Omega)}
			\notag\\
			& \leq \eps\|\mathcal{S} v^\fix_{m}\|^{2}_{L^{2}(\Omega)}+C\|M^\fix_{m}\|_{L^{2}(\Omega)}\big(\|M^\fix_{m}\|_{L^{2}(\Omega)}^{2}+\|\Delta M^\fix_{m}\|^{2}_{L^{2}(\Omega)}\big)^{\frac{1}{2}}\|\nabla H\|^{2}_{L^{2}(\Omega)}
			\notag\\
			& \leq \eps\|\mathcal{S} v^\fix_{m}\|^{2}_{L^{2}(\Omega)}+C\big(\|M^\fix_{m}\|_{L^{2}(\Omega)}^{2}+\|\Delta M^\fix_{m}\|^{2}_{L^{2}(\Omega)}\big)\|\nabla H\|^{2}_{L^{2}(\Omega)}
			\notag\\
			& \leq \eps\|\mathcal{S} v^\fix_{m}\|^{2}_{L^{2}(\Omega)}+C\|\nabla H\|^{2}_{L^{2}(\Omega)}\|\big(\Delta M^\fix_{m}-f(M^\fix_{m})\big)\|^{2}_{L^{2}(\Omega)}+C\|\nabla H\|^{2}_{L^{2}(\Omega)},
	\end{align}
where we have applied \eqref{estfM}.
	The term $I_{4}$ can be estimated in a similar fashion as $I_{1}$. We get
	\begin{align}\label{I4}
			|I_{4}|\leq \eps \|\Delta F^\fix_{m}\|^{2}_{L^{2}(\Omega)}+C\|\nabla F^\fix_{m}\|^{4}_{L^{2}(\Omega)}+C\|\nabla v^\fix_{m}\|^{4}_{L^{2}(\Omega)}+C.
	\end{align}
	Invoking \eqref{smoreinterpole1} and the inequality $\|\nabla^{2}v^\fix_{m}\|_{L^{2}(\Omega)}\leq C\|\Delta v^\fix_{M}\|_{L^{2}(\Omega)}$ to estimate $\|\nabla v^\fix_{m}\|_{L^{4}(\Omega)}$, and by \eqref{SmajorDelta} we find that
	\begin{align}\label{I5}
			 |I_{5}|
			&\leq C\|\Delta F^\fix_{m}\|_{L^{2}(\Omega)}\big(\|F^\fix_{m}\|^{\frac{1}{2}}_{L^{2}(\Omega)}\|\nabla F^\fix_{m}\|^{\frac{1}{2}}_{L^{2}(\Omega)}\big)\big(\|\nabla v^\fix_{m}\|_{L^{2}(\Omega)}+\|\Delta v^\fix_{m}\|^{\frac{1}{2}}_{L^{2}(\Omega)}\|\nabla v^\fix_{m}\|^{\frac{1}{2}}_{L^{2}(\Omega)}\big) 
			\notag\\
			&\leq \eps\|\Delta F^\fix_{m}\|^{2}_{L^{2}(\Omega)}+C\|\nabla F^\fix_{m}\|_{L^{2}(\Omega)}\|\Delta v^\fix_{m}\|_{L^{2}(\Omega)}\|\nabla v^\fix_{m}\|_{L^{2}(\Omega)}+C\|\nabla F^\fix_{m}\|_{L^{2}(\Omega)}\|\nabla v^\fix_{m}\|^{2}_{L^{2}(\Omega)}
			\notag\\
			& \leq \eps\|\Delta F^\fix_{m}\|^{2}_{L^{2}(\Omega)}+\eps\|\mathcal{S} v^\fix_{m}\|^{2}_{L^{2}(\Omega)}+C\|\nabla F^\fix_{m}\|^{2}_{L^{2}(\Omega)}\|\nabla v^\fix_{m}\|^{2}_{L^{2}(\Omega)}+C\|\nabla F^\fix_{m}\|_{L^{2}(\Omega)}\|\nabla v^\fix_{m}\|^{2}_{L^{2}(\Omega)}
			\notag\\
			& \leq \eps\|\Delta F^\fix_{m}\|^{2}_{L^{2}(\Omega)}+\eps\|\mathcal{S} v^\fix_{m}\|^{2}_{L^{2}(\Omega)}+C\|\nabla F^\fix_{m}\|^{4}_{L^{2}(\Omega)}+C\|\nabla v^\fix_{m}\|^{4}_{L^{2}(\Omega)}+C.
	\end{align}

	Using \eqref{smoreinterpole1}, \eqref{interpoledir}, \eqref{SmajorDelta} and \eqref{interpol1}, we derive the estimate
	\begin{align}\label{I6}
			|I_{6}|
			&\leq \|\nabla(\Delta M^\fix_{m}-f(M^\fix_{m}))\|_{L^{2}(\Omega)}\big(\|\nabla v^\fix_{m}\|_{L^{4}(\Omega)}\|\nabla M^\fix_{m}\|_{L^{4}(\Omega)}+\|v^\fix_{m}\|_{L^{\infty}(\Omega)}\|M^\fix_{m}\|_{W^{2,2}(\Omega)}\big)
			\notag\\
			& \leq \eps\|\nabla(\Delta M^\fix_{m}-f(M^\fix_{m}))\|^{2}_{L^{2}(\Omega)}+C\|\nabla v^\fix_{m}\|^{2}_{L^{4}(\Omega)}\|\nabla M^\fix_{m}\|^{2}_{L^{4}(\Omega)}+C\|v^\fix_{m}\|^{2}_{L^{\infty}(\Omega)}\|M^\fix_{m}\|^{2}_{W^{2,2}(\Omega)}
			\notag\\
			& \leq \eps\|\nabla(\Delta M^\fix_{m}-f(M^\fix_{m}))\|^{2}_{L^{2}(\Omega)}+C\|\Delta v^\fix_{m}\|_{L^{2}(\Omega)}\|\nabla v^\fix_{m}\|_{L^{2}(\Omega)}\big(\|(\Delta M^\fix_{m}-f(M^\fix_{m}))\|_{L^{2}(\Omega)}+1\big)
			\notag\\
			&\qquad +C\|\nabla v^\fix_{m}\|^{2}_{L^{2}(\Omega)}\big(\|\big(\Delta M^\fix_{m}-f(M^\fix_{m})\big)\|_{L^{2}(\Omega)}+1\big)
			\notag\\
			&\qquad+C\|\Delta v^\fix_{m}\|_{L^{2}(\Omega)}\|v^\fix_{m}\|_{L^{2}(\Omega)}\big(\|(\Delta M^\fix_{m}-f(M^\fix_{m}))\|^{2}_{L^{2}(\Omega)}+1\big)
			\notag\\
			& \leq \eps\|\mathcal{S} v^\fix_{m}\|^{2}_{L^{2}(\Omega)}+\eps\|\nabla(\Delta M^\fix_{m}-f(M^\fix_{m}))\|^{2}_{L^{2}(\Omega)}+C\|\nabla v^\fix_{m}\|^{4}_{L^{2}(\Omega)}
			\notag\\
			&\qquad+C\|(\Delta M^\fix_{m}-f(M^\fix_{m}))\|^{4}_{L^{2}(\Omega)}+C.
	\end{align}
	
	For the term $I_{7}$ we simply obtain the bound
	\begin{align}\label{I7}
			 |I_{7}|\leq \eps\|\nabla(\Delta M^\fix_{m}-f(M^\fix_{m}))\|^{2}_{L^{2}(\Omega)}+C\|\nabla H\|^{2}_{L^{2}(\Omega)}.
	\end{align}
	Eventually, for the term $I_8$, we have
	\begin{align}\label{I8}
			|I_{8}|&\leq \left|\int_{\Omega}f'(M^\fix_{m})\partial_{t}M^\fix_{m}(\Delta M^\fix_{m}-f(M^\fix_{m}))\right|
			\notag\\
			& \leq \left|\int_{\Omega}f'(M^\fix_{m})\bigg(-v^\fix_{m}\cdot\nabla M^\fix_{m}+(\Delta M^\fix_{m}-f(M^\fix_{m}))+H\bigg)\cdot(\Delta M^\fix_{m}-f(M^\fix_{m}))\right|
			\notag\\
			&\leq C\|f'(M^\fix_{m})\|_{L^{4}(\Omega)}\|v^\fix_{m}\|_{L^{4}(\Omega)}\|\nabla M^\fix_{m}\|_{L^{4}(\Omega)}\|(\Delta M^\fix_{m}-f(M^\fix_{m}))\|_{L^{4}(\Omega)}
			\notag\\
			&\qquad+C\|f'(M^\fix_{m})\|_{L^{2}(\Omega)}\|(\Delta M^\fix_{m}-f(M^\fix_{m}))\|^{2}_{L^{4}(\Omega)}
			\notag\\
			&\qquad+C\|H\|_{L^{4}(\Omega)}\|f'(M^\fix_{m})\|_{L^{2}(\Omega)}\|(\Delta M^\fix_{m}-f(M^\fix_{m}))\|_{L^{4}(\Omega)}
			\notag\\
			& \leq C\|\nabla v^\fix_{m}\|_{L^{2}(\Omega)}\|v^\fix_{m}\|_{L^{2}(\Omega)}\|\nabla M^\fix_{m}\|^{2}_{L^{4}(\Omega)}+C\|(\Delta M^\fix_{m}-f(M^\fix_{m}))\|^{2}_{L^{4}(\Omega)}+C\|H\|^{2}_{W^{1,2}(\Omega)}
			\notag\\
			& \leq C\|\nabla v^\fix_{m}\|^{2}_{L^{2}(\Omega)}+C\|(\Delta M^\fix_{m}-f(M^\fix_{m}))\|^{2}_{L^{2}(\Omega)}
			\notag\\
			&\quad
			+\eps\|\nabla(\Delta M^\fix_{m}-f(M^\fix_{m}))\|^{2}_{L^{2}(\Omega)}
			+C\|H\|^{2}_{W^{1,2}(\Omega)}.
	\end{align}
	Here, to deduce the fourth line from the third we have used  
	$$\|f'(M^\fix_{m})\|_{L^{4}(\Omega)}\leq C\|M^\fix_{m}\|^{2}_{L^{8}(\Omega)}+C\leq C\quad\mbox{for}\,\,a.e.\,\,t\in[0,T].$$
	and the final line is obtained by using \eqref{interpol1}--\eqref{interpol2}.

	Now, fixing $\eps$ sufficiently small and using the bounds on the terms $I_1,...,I_8$ to estimate the right-hand side of \eqref{computenergy}, we eventually conclude the uniform estimate \eqref{diffinq}.
	
	Invoking Gronwall's inequality, we infer that
	\begin{align}\label{afterGron}
			\mathcal{A}(t)\leq  e^{ C\int_{0}^{T}\left(\mathcal{A}(\tau)+\|\nabla H(\tau)\|^{2}_{L^{2}(\Omega)}\right)}\left[\mathcal{A}(0)+\|H\|^{2}_{L^{2}(0,T;W^{1,2}(\Omega))}\right],
	\end{align}
	for any $t\in[0,T].$ Since $H\in L^{2}(0,T; W^{1,2}(\Omega))$, and
	$$\int_{0}^{T} \mathcal{A}(\tau)\leq C$$ 
	as a consequence of \cite[Section~3.1.3.4]{Forster},  we conclude that 
	\begin{align}\label{LinfAt}
	\mathcal{A}(t)\leq C\quad\mbox{for $a.e.$ $t\in[0,T]$}.
	\end{align}
	Integrating \eqref{diffinq} over $(0,T)$ and using \eqref{LinfAt}, we further obtain
	\begin{align}\label{mindbndmathcalB}
			\int_{0}^{T}\mathcal{B}(s)\leq C.
	\end{align}
	
	\textit{Step~3: Passage to the limit and regularity properties.}
	Due to the uniform a priori estimates \eqref{LinfAt} and \eqref{mindbndmathcalB} we can now apply the Banach--Alaoglu theorem to pass to the limit $m\to\infty$ in the weak-$^*$ sense in the corresponding function spaces. 
	In particular, we recover the estimate \eqref{strongestimate} by invoking the weak lower semicontinuity of the involved norms. 
	To show that $(v,F,M)$ actually solves the weak formulation of the system \eqref{diffviscoelastic*} we need to pass to the limit $m\rightarrow\infty$ in the weak formulation of the system solved by $(v^\fix_{m},F^\fix_{m},M^\fix_{m}).$ This can be done by proceeding exactly as in \cite[Section 3.1.4.2]{Forster}.
	
	In order to complete the proof of Theorem \ref{globalstrong} we still need to recover the time regularities of $v$, $F$ and $M$. 
	In view of \eqref{strongestimate} and $H\in L^{2}(0,T;W^{1,2}(\Omega))$ it is not difficult to check that
	$$\left\|-(v\cdot\nabla)v-\mbox{div}\left((\nabla M\odot\nabla M)-FF^{T}\right)+\nu\Delta v+(\nabla H)^TM\right\|_{L^{2}(Q_{T})}\leq C,$$
	for some positive constant $C$ depending only on $\|v_{0}\|_{W^{1,2}(\Omega)},$ $\|M_{0}\|_{W^{2,2}(\Omega)},$ $\|F_{0}\|_{W^{1,2}(\Omega)},$ $\|H\|_{L^{2}(0,T;W^{1,2}(\Omega))},$ $|\Omega|$ and the fixed final time $T.$ Hence, the time regularity $\partial_{t}v\in L^{2}(0,T;L^{2}_{\dvr}(\Omega))$ follows by a standard comparison argument in the weak formulation of \eqref{diffvlin1} written for test functions in $L^{2}_{\dvr}(\Omega)$.
	The time regularities of $F$ and $M$ can be obtained in a similar fashion by estimating the right-hand sides of the weak formulations of \eqref{diffvlin3} and \eqref{diffvlin4}.
	
	The regularity $p\in W^{-1,\infty}(0,T;L^2(\Omega))$ of the pressure can be recovered by proceeding as in the proof of Theorem~\ref{weaksolution}. Since \eqref{diffvlin1} is satisfied in the sense of distributions, we can use the regularities of $(v,F,M)$ to conclude 
	$\nabla p\in L^{2}(Q_{T})$
	by another comparison argument.
	As we already know from Theorem~\ref{weaksolution} that $\int_{\Omega}p=0$
	we eventually conclude that $p\in L^{2}(0,T;W^{1,2}(\Omega))$ by means of Poincar\'e's inequality.

    In summary, the regularity properties \eqref{strngsol} are established. Hence, the quadruplet $(v,p,F,M)$ is actually a strong solution of the system \eqref{diffviscoelastic*}.
    
	Lastly, we point out that the strong solution $(v,p,F,M)$ is unique, which is a direct consequence of the uniqueness of weak solutions established in Theorem~\ref{weaksolution}.
	
    This means that all assertions are established and thus, the proof of Theorem \ref{globalstrong} is complete.
\end{proof}


\section{Stability estimates}\label{Stabilityestimates}
In this section, we investigate the stability of strong solutions to the system \eqref{diffviscoelastic*} with respect to perturbations of the external magnetic field in both the weak energy spaces (cf.~\eqref{funframeweak}) and the strong regularity framework (cf.~\eqref{strngsol}). 

\subsection{Weak stability}
We now present a stability estimate for solutions of the system \eqref{diffviscoelastic*}, where the involved norms correspond to the regularity properties of weak solutions. Notice that in the following theorem, the initial data are taken sufficiently regular such that the existence of a unique strong solution to \eqref{diffviscoelastic*} is ensured by Theorem \ref{globalstrong}. Theorem~\ref{stabilities} will further play a crucial role in the proof of the strong stability estimate presented in Theorem \ref{stabilitiesstrong}.

\begin{thm}\label{stabilities}
Let	$T>0,$  $v_{0}\in W^{1,2}_{0,\dvr}(\Omega),$ $F_{0}\in W^{1,2}_{0}(\Omega),$ $M_{0}\in W^{2,2}_{n}(\Omega)$ be given, and let $(v_{1},F_{1}, M_{1})$ and $(v_{2},F_{2}, M_{2})$ denote the unique strong solutions of \eqref{diffviscoelastic*} to the initial data $(v_{0},F_{0},M_{0})$ and the external magnetic fields $H_{1},\, H_{2}\in L^{2}(0,T;W^{1,2}(\Omega))$, respectively. We write $\overline{H} = H_1-H_2$.

Then, the difference of these two solutions $(\overline{v},\overline{F},\overline{M})=(v_{1}-v_{2},F_{1}-F_{2},M_{1}-M_{2})$ fulfills the stability estimate
\begin{align}\label{weakeststimate}
&\|\overline{v}\|^{2}_{L^{\infty}(0,T;L^{2}(\Omega))\cap L^{2}(0,T;W^{1,2}(\Omega))}+\|\overline{F}\|^{2}_{L^{\infty}(0,T;L^{2}(\Omega))\cap L^{2}(0,T;W^{1,2}(\Omega))}+\|\overline{M}\|_{L^{\infty}(0,T;W^{1,2}(\Omega))}^{2}\notag\\
&+\int_{0}^{T}\|\Delta\overline{M}\|^{2}_{L^{2}(\Omega)}\leq \mathfrak{S}_1\bigl(\|H_{1}\|_{L^{2}(0,T;W^{1,2}(\Omega))},\|H_{2}\|_{L^{2}(0,T;W^{1,2}(\Omega))}\bigr)\|\overline{H}\|_{L^{2}(0,T;W^{1,2}(\Omega))},
\end{align}
where the function $\mathfrak{S}_1:[0,\infty)\times[0,\infty)\to(0,\infty)$ is nondecreasing in both of its variables and may depend on $T$, $|\Omega|$ and the initial data $(v_{0},F_{0},M_{0})$.
\end{thm}	
	
\begin{proof}
Our approach is motivated by \cite{anjazab}, where the authors prove the uniqueness of weak solutions to the system \eqref{diffviscoelastic*}. However, in contrast to \cite{anjazab}, we do not perturb the initial data but the external magnetic field $H.$

We consider the difference of the equations of \eqref{diffviscoelastic*} written for $(v_{1},F_{1},M_{1})$ and $(v_{2},F_{2},M_{2})$, respectively. Of the resulting equations, we test the first one by $\overline{v},$ the second one by $\overline{F}$ and the third one by $\overline{M}$ and $-\Delta\overline{M}$. This leads us to the equations
\begin{align}
    \label{eqdiffv}
	&\frac{1}{2}\frac{\mathrm d}{\mathrm dt}\int_{\Omega}|\overline{v}|^{2} +\nu\int_{\Omega}|\nabla \overline{v}|^{2}
	\notag\\
	&=-\int_{\Omega}(\overline{v}\cdot\nabla)v_{1}\cdot\overline{v}-\int_{\Omega}\mbox{div}(\nabla M_{1}\odot \nabla{M}_{1}-\nabla M_{2}\odot \nabla{M}_{2})\cdot\overline{v}
	\notag\\
	&\qquad +\int_{\Omega} \mbox{div}\left({F}_{1}F^{T}_{1}-{F}_{2}F^{T}_{2}\right)\cdot\overline{v}+\int_{\Omega} (\nabla \overline{H})^T M_{1} \cdot \overline{v}+\int_{\Omega} (\nabla{H}_{2})^T\overline{M} \cdot \overline{v}
	=:\sum_{i=1}^{5}\mathcal{I}_{i},
	\\[1ex]
	\label{eqdiffF}
	&\frac{1}{2}\frac{\mathrm d}{\mathrm dt}\int_{\Omega}|\overline{F}|^{2}+\kappa\int_{\Omega}|\nabla \overline{F}|^{2}=\int_{\Omega}(\nabla v_{1}F_{1}-\nabla v_{2}F_{2})\cdot \overline{F}-\int_{\Omega}(\overline{v}\cdot\nabla)F_{1}\cdot\overline{F}
	=:\sum_{i=6}^{7}\mathcal{I}_{i},
	\\[1ex]
	\label{eqdiffM1}
	&\frac{1}{2}\frac{\mathrm d}{\mathrm dt}\int_{\Omega}|\overline{M}|^{2}+\int_{\Omega}|\nabla \overline{M}|^{2} 
	\notag\\
	&=-\int_{\Omega}(\overline{v}\cdot\nabla)M_{1}\cdot\overline{M}-\int_{\Omega}(f(M_{1})-f(M_{2}))\cdot\overline{M}+\int_{\Omega} \overline{H}\cdot \overline{M}
	=:\sum_{i=8}^{10}\mathcal{I}_{i},
	\\[1ex]
    \label{eqdiffM}
	&\frac{1}{2}\frac{\mathrm d}{\mathrm dt}\int_{\Omega}|\nabla \overline{M}|^{2}+\int_{\Omega}|\Delta\overline{M}|^{2} 
	\notag\\
	& =\int_{\Omega}\left((v_{1}\cdot\nabla){M}_{1}-(v_{2}\cdot\nabla){M}_{2}\right)\cdot \Delta\overline{M}+\int_{\Omega}(f(M_{1})-f(M_{2}))\cdot \Delta\overline{M}
	+\int_{\Omega} \overline{H}\cdot \Delta\overline{M} 
	=:\sum_{i=11}^{13}\mathcal{I}_{i},
\end{align}
where $f$ is the function that was introduced in \eqref{fM}.

In the following, we intend to estimate the terms $\mathcal{I}_{i}$, $i=1,...,13$, in a suitable way to conclude the stability estimate. 
Let $\eps>0$ denote a small parameter that will be fixed later. The letter $C$ will denote positive constants that depend on $\eps$, $T$, $|\Omega|$ and the initial data $(v_{0},F_{0},M_{0})$ and may change their value from line to line.

For $\mathcal I_1$, we obtain the estimate
\begin{align}\label{eI1}
|\mathcal{I}_{1}|
&\leq \|\overline{v}\|^{2}_{L^{4}(\Omega)}\|\nabla v_{1}\|_{L^{2}(\Omega)} 
\leq C \|\overline{v}\|_{L^{2}(\Omega)}\|\nabla\overline{v}\|_{L^{2}(\Omega)}\|\nabla v_{1}\|_{L^{2}(\Omega)}
\notag\\
&\leq \eps\|\nabla\overline{v}\|^{2}_{L^{2}(\Omega)}
+C\|\nabla v_{1}\|_{L^{2}(\Omega)}^{2}\|\overline{v}\|^{2}_{L^{2}(\Omega)}.
\end{align}
In view of the identity \eqref{ptwiseidentity} which holds for both $M_1$ and $M_2$,
we obtain 
\begin{align}\nonumber
 \mathcal{I}_{2}+\mathcal{I}_{11}&
=-\left(\int_{\Omega}(v_{1}\cdot\nabla)\overline{M}\cdot\Delta M_{2}-(v_{2}\cdot\nabla)\overline{M}\cdot\Delta M_{1}\right)\\
&=-\left(\int_{\Omega} (\overline{v}\cdot\nabla)\overline{M}\cdot\Delta M_{2}-(v_{2}\cdot\nabla)\overline{M}\cdot\Delta\overline{M}\right)
\end{align}
via integration by parts. We thus have
\begin{align}\label{eI212}
|\mathcal{I}_{2}+\mathcal{I}_{11}|\leq \left|\int_{\Omega} (\overline{v}\cdot\nabla)\overline{M}\cdot\Delta M_{2}\right|+\left|\int_{\Omega}(v_{2}\cdot\nabla)\overline{M}\cdot\Delta\overline{M}\right|.
\end{align}
For the first term in the right-hand side of \eqref{eI212} we obtain the estimate
\begin{align}\label{eI2121}
&\left|\int_{\Omega} (\overline{v}\cdot\nabla)\overline{M}\cdot\Delta M_{2}\right|\leq \|\Delta M_{2}\|_{L^{2}(\Omega)}\|\overline{v}\|_{L^{4}(\Omega)}\|\nabla \overline{M}\|_{L^{4}(\Omega)}
\notag\\
&\quad \leq C\|\Delta M_{2}\|_{L^{2}(\Omega)}\|\overline{v}\|^{\frac{1}{2}}_{L^{2}(\Omega)}\|\nabla \overline{v}\|^{\frac{1}{2}}_{L^{2}(\Omega)}\|\nabla \overline{M}\|^{\frac{1}{2}}_{L^{2}(\Omega)}\left(\|\nabla \overline{M}\|^{2}_{L^{2}(\Omega)}+\|\Delta\overline{M}\|^{2}_{L^{2}(\Omega)}\right)^{\frac{1}{4}}
\notag\\
&\quad \leq C\|\Delta M_{2}\|^{2}_{L^{2}(\Omega)}\|\overline{v}\|_{L^{2}(\Omega)}\|\nabla \overline{M}\|_{L^{2}(\Omega)}+\eps\|\nabla\overline{v}\|_{L^{2}(\Omega)}\left(\|\nabla\overline{M}\|^{2}_{L^{2}(\Omega)}+\|\Delta\overline{M}\|^{2}_{L^{2}(\Omega)}\right)^{\frac{1}{2}}
\notag\\
&\quad  \leq C\bigg(\|\Delta M_{2}\|^{2}_{L^{2}(\Omega)}+1\bigg)\bigg(\|\overline{v}\|^{2}_{L^{2}(\Omega)}+\|\nabla\overline{M}\|^{2}_{L^{2}(\Omega)}\bigg)+\frac{\eps}{2}\|\nabla \overline{v}\|^{2}_{L^{2}(\Omega)}+\frac{\eps}{2}\|\Delta\overline{M}\|^{2}_{L^{2}(\Omega)},
\end{align}
where we used the interpolation inequalities \eqref{L40bnd} for $\|\overline{v}\|_{L^{4}(\Omega)}$ and \eqref{smoreinterpole2} for $\|\nabla\overline{M}\|_{L^{4}(\Omega)}$ to deduce the second line.
The second term in the right-hand side of \eqref{eI212} can be estimated as follows: 
\begin{align}\label{eI2122}
& \left|\int_{\Omega}(v_{2}\cdot\nabla)\overline{M}\cdot\Delta\overline{M}\right|\leq \|v_{2}\|_{L^{4}(\Omega)}\|\nabla\overline{M}\|_{L^{4}(\Omega)}\|\Delta\overline{M}\|_{L^{2}(\Omega)}
\notag\\
&\quad  \leq C\|v_{2}\|_{L^{4}(\Omega)}\|\Delta\overline{M}\|_{L^{2}(\Omega)}\|\nabla \overline{M}\|^{\frac{1}{2}}_{L^{2}(\Omega)}\left(\|\nabla \overline{M}\|^{2}_{L^{2}(\Omega)}+\|\Delta\overline{M}\|^{2}_{L^{2}(\Omega)}\right)^{\frac{1}{4}}
\notag\\
&\quad  \leq\eps\|\Delta\overline{M}\|^{2}_{L^{2}(\Omega)}+C\|v_{2}\|_{L^{2}(\Omega)}\|\nabla v_{2}\|_{L^{2}(\Omega)} \|\nabla \overline{M}\|_{L^{2}(\Omega)}\left(\|\nabla \overline{M}\|^{2}_{L^{2}(\Omega)}+\|\Delta\overline{M}\|^{2}_{L^{2}(\Omega)}\right)^{\frac{1}{2}}
\notag\\
&\quad \leq 2\eps\|\Delta\overline{M}\|^{2}_{L^{2}(\Omega)}+C\left(\|v_{2}\|^{2}_{L^{2}(\Omega)}\|\nabla v_{2}\|^{2}_{L^{2}(\Omega)}+1\right)\|\nabla\overline{M}\|^{2}_{L^{2}(\Omega)}.
\end{align}
Here, we used the interpolation inequalities \eqref{L40bnd} for $\|{v}_{2}\|_{L^{4}(\Omega)}$ and \eqref{smoreinterpole2} for $\|\nabla\overline{M}\|_{L^{4}(\Omega)}.$
Hence, using \eqref{eI2121} and \eqref{eI2122} to bound the right-hand side of \eqref{eI212}, we conclude that
\begin{align}\label{fei212}
|\mathcal{I}_{2}+\mathcal{I}_{11}|
&\leq 3\eps\|\Delta\overline{M}\|^{2}_{L^{2}(\Omega)}+\eps\|\nabla\overline{v}\|^{2}_{L^{2}(\Omega)}
\notag\\
&\quad +C\big(\|\Delta M_{2}\|^{2}_{L^{2}(\Omega)}+\|v_{2}\|^{2}_{L^{2}(\Omega)}\|\nabla v_{2}\|^{2}_{L^{2}(\Omega)}+1\big)\big(\|\overline{v}\|^{2}_{L^{2}(\Omega)}+\|\nabla\overline{M}\|^{2}_{L^{2}(\Omega)}\big).
\end{align}
After some straightforward manipulations we find that
\begin{align}\label{I36}
 \mathcal{I}_{3}+\mathcal{I}_{6}&=\int_{\Omega}\bigl(F_{1}F_{1}^{T}\bigr)\cdot\nabla v_{2}+\int_{\Omega}\bigl(F_{2}F_{2}^{T}\bigr)\cdot\nabla v_{1}-\int_{\Omega}\bigl(\nabla v_{1}F_{1}\bigr)\cdot F_{2}-\int_{\Omega}\bigl(\nabla v_{2}F_{2}\bigr)\cdot F_{1}
\notag \\
&= -\int_{\Omega}\bigl(\nabla v_{1}\overline{F}\bigr)\cdot F_{2}+\int_{\Omega}\bigl(\nabla v_{2}\overline{F}\bigr)\cdot F_{1}=-\int_{\Omega}\bigl(\nabla \overline{v}\overline{F}\bigr)\cdot F_{1}+\int_{\Omega}\bigl(\nabla v_{1}\overline{F}\bigr)\cdot \overline{F}.
\end{align}
Hence, the term $\mathcal{I}_{3}+\mathcal{I}_{6}$ can be estimated as follows:
\begin{align}\label{estI36}
|\mathcal{I}_{3}+\mathcal{I}_{6}|&\leq \|\nabla\overline{v}\|_{L^{2}(\Omega)}\|\overline{F}\|_{L^{4}(\Omega)}\|F_{1}\|_{L^{4}(\Omega)}+\|\nabla v_{1}\|_{L^{2}(\Omega)}\|\overline{F}\|^{2}_{L^{4}(\Omega)}
\notag\\
& \leq C\|\nabla \overline{v}\|_{L^{2}(\Omega)}\|\overline{F}\|^{\frac{1}{2}}_{L^{2}(\Omega)}\|\nabla\overline{F}\|^{\frac{1}{2}}_{L^{2}(\Omega)}\|{F}_{1}\|^{\frac{1}{2}}_{L^{2}(\Omega)}\|\nabla{F}_{1}\|^{\frac{1}{2}}_{L^{2}(\Omega)}
\notag\\
&\qquad+C\|\nabla v_{1}\|_{L^{2}(\Omega)}\|\overline{F}\|_{L^{2}(\Omega)}\|\nabla \overline{F}\|_{L^{2}(\Omega)}
\notag\\
& \leq \eps\|\nabla\overline{v}\|^{2}_{L^{2}(\Omega)}+C\|\overline{F}\|_{L^{2}(\Omega)}\|\nabla\overline{F}\|_{L^{2}(\Omega)}\|{F}_{1}\|_{L^{2}(\Omega)}\|\nabla{F}_{1}\|_{L^{2}(\Omega)}
\notag\\
& \qquad+\eps\|\nabla\overline{F}\|^{2}_{L^{2}(\Omega)}+C\|\nabla v_{1}\|^{2}_{L^{2}(\Omega)}\|\overline{F}\|^{2}_{L^{2}(\Omega)}
\notag\\
& \leq \eps\|\nabla\overline{v}\|^{2}_{L^{2}(\Omega)}+2\eps\|\nabla\overline{F}\|^{2}_{L^{2}(\Omega)}+C\|\nabla v_{1}\|^{2}_{L^{2}(\Omega)}\|\overline{F}\|^{2}_{L^{2}(\Omega)}
\notag\\
& \qquad+ C\|\overline{F}\|^{2}_{L^{2}(\Omega)}\bigl(\|F_{1}\|^{2}_{L^{2}(\Omega)}\|\nabla F_{1}\|^{2}_{L^{2}(\Omega)}\bigr).
\end{align}
Using \eqref{smoreinterpole1} to estimate $\|M_{1}\|_{L^{4}(\Omega)}$, we further get
\begin{align}\label{eI3}
|\mathcal{I}_{4}|&\leq \|\nabla \overline{H}\|_{L^{2}(\Omega)}\|M_{1}\|_{L^{4}(\Omega)}\|\overline{v}\|_{L^{4}(\Omega)}
\notag\\
&\leq \tfrac{1}{2}\|\nabla\overline{H}\|^{2}_{L^{2}(\Omega)}+\tfrac{1}{2}\|M_{1}\|^{2}_{L^{4}(\Omega)}\|\overline{v}\|^{2}_{L^{4}(\Omega)}
\notag\\
& \leq \tfrac{1}{2}\|\nabla\overline{H}\|^{2}_{L^{2}(\Omega)}+C\bigl(\|M_{1}\|^{2}_{L^{2}(\Omega)}+\|M_{1}\|_{L^{2}(\Omega)}\|\nabla M_{1}\|_{L^{2}(\Omega)}\bigr)\bigl(\|\overline{v}\|_{L^{2}(\Omega)}\|\nabla\overline{v}\|_{L^{2}(\Omega)}\bigr)
\notag\\
& \leq\tfrac{1}{2}\|\nabla \overline{H}\|^{2}_{L^{2}(\Omega)}+\eps\|\nabla\overline{v}\|^{2}_{L^{2}(\Omega)}+C\bigl(\|M_{1}\|^{4}_{L^{2}(\Omega)}+\|M_{1}\|_{L^{2}(\Omega)}^{2}\|\nabla M_{1}\|^{2}_{L^{2}(\Omega)}\bigr)\|\overline{v}\|^{2}_{L^{2}(\Omega)}.
\end{align}
Next, for $\mathcal{I}_{5}$, we derive the following estimate:
\begin{align}\label{eI4}
|\mathcal{I}_{5}|& \leq\|\nabla H_{2}\|_{L^{2}(\Omega)}\|\overline{M}\|_{L^{\infty}(\Omega)}\|\overline{v}\|_{L^{2}(\Omega)}
\notag\\
& \leq C\|\nabla H_{2}\|^{2}_{L^{2}(\Omega)}\|\overline{v}\|^{2}_{L^{2}(\Omega)}+\eps\|\overline{M}\|^{2}_{W^{2,2}(\Omega)}
\notag\\
& \leq C\|\nabla H_{2}\|^{2}_{L^{2}(\Omega)}\|\overline{v}\|^{2}_{L^{2}(\Omega)}+\eps C \bigl(\|\overline{M}\|^{2}_{L^{2}(\Omega)}+\|\Delta\overline{M}\|^{2}_{L^{2}(\Omega)}\bigr).
\end{align}
For the terms $\mathcal{I}_{7}$ and $\mathcal{I}_{8}$, we obtain the bounds
\begin{align}\label{eI7}
 |\mathcal{I}_{7}|& \leq \|\overline{v}\|_{L^{4}(\Omega)}\|\overline{F}\|_{L^{4}(\Omega)}\|\nabla F_{1}\|_{L^{2}(\Omega)}
\notag\\
& \leq \|\overline{v}\|^{\frac{1}{2}}_{L^{2}(\Omega)}\|\nabla\overline{v}\|^{\frac{1}{2}}_{L^{2}(\Omega)}\|\overline{F}\|^{\frac{1}{2}}_{L^{2}(\Omega)}\|\nabla\overline{F}\|^{\frac{1}{2}}_{L^{2}(\Omega)}\|\nabla F_{1}\|_{L^{2}(\Omega)}
\notag\\
& \leq C(\|\overline{v}\|^{2}_{L^{2}(\Omega)}+\|\overline{F}\|^{2}_{L^{2}(\Omega)})\|\nabla F_{1}\|^{2}_{L^{2}(\Omega)}+\eps\bigg(\tfrac{1}{2}\|\nabla\overline{v}\|^{2}_{L^{2}(\Omega)}+\tfrac{1}{2}\|\nabla\overline{F}\|^{2}_{L^{2}(\Omega)}\bigg)
\end{align}
and 
\begin{align}\label{eI8}
|\mathcal{I}_{8}|&\leq \|\overline{v}\|_{L^{2}(\Omega)}\|\nabla M_{1}\|_{L^{4}(\Omega)}\|\overline{M}\|_{L^{4}(\Omega)}
\notag\\
& \leq C\|\overline{v}\|_{L^{2}(\Omega)}\|\nabla M_{1}\|_{L^{4}(\Omega)}\bigg(\|\overline{M}\|_{L^{2}(\Omega)}+\|\overline{M}\|^{\frac{1}{2}}_{L^{2}(\Omega)}\|\nabla\overline{M}\|^{\frac{1}{2}}_{L^{2}(\Omega)}\bigg)
\notag\\
& \leq C\bigg(\|\overline{v}\|^{2}_{L^{2}(\Omega)}\|\nabla M_{1}\|^{2}_{L^{4}(\Omega)}+\|\overline{M}\|^{2}_{L^{2}(\Omega)}+\|\nabla\overline{M}\|^{2}_{L^{2}(\Omega)}\bigg),
\end{align}
where we used \eqref{smoreinterpole1} to estimate $\|\overline{M}\|_{L^{4}(\Omega)}.$
Next, using the inequality
$$(|M_{1}|^{2}M_{1}-|M_{2}|^{2}M_{2})\cdot(M_{1}-M_{2})\geqslant0,$$
we infer the estimate
\begin{align}\label{eI9}
 \mathcal{I}_{9}\leq \frac{1}{\alpha^{2}}\|\overline{M}\|^{2}_{L^{2}(\Omega)}.
\end{align}
We further get
\begin{align}\label{eI10}
|\mathcal{I}_{10}|\leq C\bigl(\|\overline{H}\|^{2}_{L^{2}(\Omega)}+\|\overline{M}\|^{2}_{L^{2}(\Omega)}\bigr).
\end{align}
Moreover, employing the inequality
$$\left||M_{1}|^{2}M_{1}-|M_{2}|^{2}M_{2}\right|\leq C|\overline{M}|\bigl(|M_{1}|^{2}+|M_{2}|^{2}\bigr),$$
as well as the interpolation inequality \eqref{smoreinterpole1} to estimate $\|\overline{M}\|_{L^{4}(\Omega)}$, we find that
\begin{align}\label{eI11}
 \mathcal{I}_{12}&\leq C\|\nabla \overline{M}\|^{2}_{L^{2}(\Omega)}+C\int_{\Omega}\bigl(|M_{1}|^{2}+|M_{2}|^{2}\bigr)|\overline{M}||\Delta\overline{M}|
\notag\\
& \leq C\bigl(\|\nabla \overline{M}\|^{2}_{L^{2}(\Omega)}+\|\overline{M}\|^{2}_{L^{4}(\Omega)}(\|M_{1}\|^{4}_{L^{8}(\Omega)}+\|M_{2}\|^{4}_{L^{8}(\Omega)})\bigr)+\eps\|\Delta \overline{M}\|^{2}_{L^{2}(\Omega)}
\notag\\
& \leq C\bigl(\|\overline{M}\|^{2}_{L^{2}(\Omega)}+\|\nabla\overline{M}\|^{2}_{L^{2}(\Omega)}\bigr)\bigl(1+\|M_{1}\|^{4}_{L^{8}(\Omega)}+\|M_{2}\|^{4}_{L^{8}(\Omega)}\bigr)+\eps\|\Delta \overline{M}\|^{2}_{L^{2}(\Omega)}.
\end{align}
Eventually, for $\mathcal{I}_{13}$, we obtain the following estimate:
\begin{align}\label{eI13}
|\mathcal{I}_{13}|\leq C\|\overline{H}\|^{2}_{L^{2}(\Omega)}+\eps\|\Delta\overline{M}\|^{2}_{L^{2}(\Omega)}.
\end{align}	
We now recall that the solutions $(v_{i},F_{i},M_{i}),$ $i\in\{1,2\}$ are bounded in the function spaces associated with weak solutions (see \eqref{funframeweak}) by a constant depending only on $T$, $\Omega$ and the initial data. 
Hence, choosing $\eps>0$ sufficiently small, adding \eqref{eqdiffv}--\eqref{eqdiffM}, and making use of the estimates \eqref{eI1}--\eqref{eI13}, we conclude that
\begin{align}\label{beforeGron}
&\frac{1}{2}\frac{\mathrm d}{\mathrm dt}\overline{\mathcal{Y}}(t)+\overline{\mathcal{B}}(t)
\le  C \mathcal Q(t) \, \overline{\mathcal{Y}}(t) + C\|\overline{H}(t)\|^{2}_{W^{1,2}(\Omega)}
\end{align}
for almost all $t\in[0,T]$, 
where 
\begin{align}\label{mathcalYB}
\overline{\mathcal{Y}}
&:=\int_{\Omega}\bigl(|\overline{v}|^{2}+|\overline{F}|^{2}+|\overline{M}|^{2}+|\nabla\overline{M}|^{2}\bigr),
\\[1ex]
\overline{\mathcal{B}}
&:=\frac{1}{2} \int_{\Omega}\bigl(\nu|\nabla \overline{v}|^{2}|+\kappa|\nabla \overline{F}|^{2}+|\nabla \overline{M}|^{2}+|\Delta \overline{M}|^{2}\bigr),
\\[1ex]
 \mathcal Q &:=
\|\nabla v_{1}\|^{2}_{L^{2}(\Omega)}+\|\Delta M_{2}\|^{2}_{L^{2}(\Omega)}+\|v_{2}\|^{2}_{L^{2}(\Omega)}\|\nabla v_{2}\|^{2}_{L^{2}(\Omega)}+\|F_{1}\|^{2}_{L^{2}(\Omega)}\|\nabla F_{1}\|^{2}_{L^{2}(\Omega)}
\notag\\
&\quad+\|M_{1}\|^{4}_{L^{2}(\Omega)}+\|M_{1}\|^{2}_{L^{2}(\Omega)}\|\nabla M_{1}\|^{2}_{L^{2}(\Omega)}+\|\nabla H_{2}\|^{2}_{L^{2}(\Omega)}+\|\nabla F_{1}\|^{2}_{L^{2}(\Omega)}
\notag\\
&\quad+\|\nabla M_{1}\|^{2}_{L^{4}(\Omega)}+\|M_{1}\|^{4}_{L^{8}(\Omega)}+\|M_{2}\|^{4}_{L^{8}(\Omega)}+1.
\end{align}
From estimate \eqref{weakestimate} we infer that $\norm{\mathcal Q}_{L^1([0,T])} \le C$.
This allows us to apply Gronwall's inequality on \eqref{beforeGron} which eventually yields the desired estimate \eqref{weakeststimate}. Thus, the proof is complete.
\end{proof}

\subsection{Strong stability}	
We now establish the stability of strong solutions to \eqref{diffviscoelastic*} with respect to perturbations of the external magnetic field in the functional spaces that correspond to the regularity framework \eqref{strngsol} of strong solutions.
\begin{thm}\label{stabilitiesstrong}
	Let	$T>0,$  $v_{0}\in W^{1,2}_{0,\dvr}(\Omega),$ $F_{0}\in W^{1,2}_{0}(\Omega),$ $M_{0}\in W^{2,2}_{n}(\Omega)$ be given, and let $(v_{1},F_{1}, M_{1})$ and $(v_{2},F_{2}, M_{2})$ denote the unique strong solutions of \eqref{diffviscoelastic*} to the initial data $(v_{0},F_{0},M_{0})$ and the external magnetic fields $H_{1},\, H_{2} \in L^{2}(0,T;W^{1,2}(\Omega))$, respectively. We write $\overline H = H_1-H_2$.
	
	Then, the difference of the two solutions  $(\overline{v},\overline{F},\overline{M})=(v_{1}-v_{2},F_{1}-F_{2},M_{1}-M_{2})$ fulfills the stability estimate
	\begin{align}\label{strongeststimate}
	&\|\overline{v}(t)\|^{2}_{W^{1,2}(\Omega)}+\|\overline{F}(t)\|^{2}_{W^{1,2}(\Omega)}+\|\overline{M}(t)\|^{2}_{W^{2,2}(\Omega)}
	\notag\\
	&\quad+\int_{0}^{t}\left(\|\overline{v}(\tau)\|^{2}_{W^{2,2}(\Omega)}+\|\overline{F}(\tau)\|^{2}_{W^{2,2}(\Omega)}+\|\overline{M}(\tau)\|^{2}_{W^{3,2}(\Omega)}\right)
	\notag\\
	&\leq  \mathfrak{S}_2\bigl(\|H_{1}\|_{L^{2}(0,T;W^{1,2}(\Omega))},\|H_{2}\|_{L^{2}(0,T;W^{1,2}(\Omega))}\bigr)\|\overline{H}\|_{L^{2}(0,T;W^{1,2}(\Omega))}
	\end{align}
	for almost all $t\in[0,T],$ where the function $\mathfrak{S}_2:[0,\infty)\times[0,\infty)\to(0,\infty)$ is nondecreasing in both of its variables and may depend on $T$, $\Omega$, and the initial data $(v_{0},F_{0},M_{0})$.
\end{thm}
\begin{proof}
    Let $\eps>0$ be arbitrary; it will be fixed later. The letter $C$ denotes generic positive constants that depend only on $\eps$, $T$, $\Omega$, and the initial data and may change their value from line to line.
    
	First using the identity \eqref{ptwiseidentity}, we reformulate \eqref{diffvlin1} (written for $(v_i,p_i,F_i,M_i)$, $i=1,2$)  as 
	 \begin{align}\label{eqsatsvi}
	 \partial_{t}v_{i}+(v_i\cdot\nabla)v_i+(\nabla M_{i})^T \Delta M_{i}-\mbox{div}(F_{i}F_{i}^{T})+\nabla p_{i}^{\#}= \nu\Delta v_{i}+(\nabla H_i)^T M_{i} \quad\mbox{in } Q_{T},
	 \end{align}
	where the term $\frac{1}{2}\nabla|\nabla M_{i}|^{2}$ is absorbed by the redefined pressure $p_{i}^{\#}.$	
	Taking the difference of \eqref{eqsatsvi} with $i=1$ and $i=2$, and testing the resulting equation by $-\Delta\overline{v},$ we infer 
	\begin{align}\label{eqdiffvstrng}
	&\frac{1}{2}\frac{\mathrm d}{\mathrm dt}\int_{\Omega}|\nabla\overline{v}|^{2}+\nu\int_{\Omega}|\Delta \overline{v}|^{2}
	\notag\\
	&\quad=\int_{\Omega}(\overline{v}\cdot\nabla)v_{1}\cdot\Delta\overline{v}+\int_{\Omega}(v_{2}\cdot\nabla)\overline{v}\cdot  \Delta\overline{v}
	+\int_{\Omega}(\nabla\overline{M})^T \Delta M_{1} \cdot \Delta\overline{v}
	\notag\\
	&\qquad+\int_{\Omega} (\nabla{M}_{2})^T \Delta \overline{M} \cdot \Delta\overline{v}-\int_{\Omega} \mbox{div}(\overline{F}F^{T}_{1})\cdot\Delta\overline{v}-\int_{\Omega}\mbox{div}({F}_{2}\overline{F}^{T})\cdot\Delta\overline{v}
	\notag\\
	&\qquad
	-\int_{\Omega} (\nabla\overline{H})^T M_{1}\cdot \Delta\overline{v}
	-\int_{\Omega} (\nabla{H}_{2})^T \overline{M} \cdot \Delta\overline{v}
	=:\sum_{i=1}^{8}I^{\overline{v}}_{i}.
	\end{align}
	The term $I^{\overline{v}}_{1}$ admits of the following estimate
	\begin{align}\nonumber
	    |I^{\overline{v}}_{1}|\displaystyle
	    &\leq\|\overline{v}\|_{L^{\infty}(\Omega)}\|\nabla v_{1}\|_{L^{2}(\Omega)}\|\Delta\overline{v}\|_{L^{2}(\Omega)}
	    \leq \|\overline{v}\|^{\frac{1}{2}}_{L^{2}(\Omega)}\|\nabla v_{1}\|_{L^{2}(\Omega)}\|\Delta \overline{v}\|^{\frac{3}{2}}_{L^{2}(\Omega)}\\
	    &\leq \frac{\epsilon}{2} \|\Delta \overline{v}\|^{2}_{L^{2}(\Omega)}+C\|\overline{v}\|^{2}_{L^{2}(\Omega)}\|\nabla v_{1}\|^{4}_{L^{2}(\Omega)},
	\end{align}
where we have used \eqref{interpoledir}	to estimate $\|\overline{v}\|_{L^{\infty}(\Omega)}.$
	Next we estimate $I^{\overline{v}}_{2}$ as
	\begin{align}\nonumber
	    |I^{\overline{v}}_{2}|&\displaystyle\leq \|v_{2}\|_{L^{4}(\Omega)}\|\nabla \overline{v}\|_{L^{4}(\Omega)}\|\Delta\overline{v}\|_{L^{2}(\Omega)}\\
	    &\leq \|v_{2}\|^{\frac{1}{2}}_{L^{2}(\Omega)}\|\nabla v_{2}\|^{\frac{1}{2}}_{L^{2}(\Omega)}\bigg(\|\nabla \overline{v}\|_{L^{2}(\Omega)}+\|\Delta \overline{v}\|^{\frac{1}{2}}_{L^{2}(\Omega)}\|\nabla\overline{v}\|^{\frac{1}{2}}_{L^{2}(\Omega)}\bigg)\|\Delta\overline{v}\|_{L^{2}(\Omega)}\\
	    & \leq\frac{\epsilon}{2} \|\Delta\overline{v}\|^{2}_{L^{2}(\Omega)}+C\|v_{2}\|_{L^{2}(\Omega)}\|\nabla v_{2}\|_{L^{2}(\Omega)}\|\nabla\overline{v}\|^{2}_{L^{2}(\Omega)}+C\|v_{2}\|^{2}_{L^{2}(\Omega)}\|\nabla v_{2}\|^{2}_{L^{2}(\Omega)}\|\nabla \overline{v}\|^{2}_{L^{2}(\Omega)}\\
	    & \leq \frac{\epsilon}{2} \|\Delta\overline{v}\|^{2}_{L^{2}(\Omega)}+C\bigg(\|\nabla v_{2}\|^{4}_{L^{2}(\Omega)}+1\bigg)\|\nabla \overline{v}\|^{2}_{L^{2}(\Omega)},
	\end{align}
	where we have used \eqref{L40bnd} to estimate $\|v_{2}\|_{L^{4}(\Omega)},$ \eqref{smoreinterpole1} to estimate $\|\nabla \overline{v}\|_{L^{4}(\Omega)}$ and the fact that $\|v_{2}\|_{L^{2}(\Omega)}\leq C\|\nabla v_{2}\|_{L^{2}(\Omega)}$ by Poincar\'{e}'s inequality.
	Summing the above two estimates and using $\|\overline{v}\|_{L^{2}(\Omega)}\leq \|\nabla\overline{v}\|_{L^{2}(\Omega)}$ we furnish
	\begin{align}\label{ssI12}
	|I^{\overline{v}}_{1}|+|I^{\overline{v}}_{2}|&\leq \eps\|\Delta \overline{v}\|^{2}_{L^{2}(\Omega)}+C\big(\|\nabla v_{1}\|^{4}_{L^{2}(\Omega)}+\|\nabla v_{2}\|^{4}_{L^{2}(\Omega)}+1\big)\|\nabla\overline{v}\|^{2}_{L^{2}(\Omega)}.	
	\end{align}
	Next, $I^{\overline{v}}_{3}$ is estimated using \eqref{interpolation3} in the following way:
	\begin{align}\label{ssI3}
    |I^{\overline{v}}_{3}|&\leq \|\nabla\overline{M}\|_{L^{\infty}(\Omega)}\|\Delta M_{1}\|_{L^{2}(\Omega)}\|\Delta\overline{v}\|_{L^{2}(\Omega)}
    \notag\\
    & \leq C\|\Delta\overline{v}\|_{L^{2}(\Omega)}\|\Delta M_{1}\|_{L^{2}(\Omega)}\|\nabla \overline{M}\|^{\frac{1}{2}}_{L^{2}(\Omega)}\bigg(\|\nabla \overline{M}\|^{2}_{L^{2}(\Omega)}+\|\Delta \overline{M}\|^{2}_{L^{2}(\Omega)}+\|\nabla\Delta \overline{M}\|^{2}_{L^{2}(\Omega)}\bigg)^{\frac{1}{4}}
    \notag\\
    &\leq \eps\|\Delta\overline{v}\|^{2}_{L^{2}(\Omega)}+ C\|\Delta M_{1}\|^{2}_{L^{2}(\Omega)}\|\nabla \overline{M}\|_{L^{2}(\Omega)}\bigg(\|\nabla \overline{M}\|^{2}_{L^{2}(\Omega)}+\|\Delta \overline{M}\|^{2}_{L^{2}(\Omega)}+\|\nabla\Delta \overline{M}\|^{2}_{L^{2}(\Omega)}\bigg)^{\frac{1}{2}}
    \notag\\
    & \leq\eps\|\Delta\overline{v}\|^{2}_{L^{2}(\Omega)}+C\|\Delta M_{1}\|^{4}_{L^{2}(\Omega)}\|\nabla\overline{M}\|^{2}_{L^{2}(\Omega)}+\eps\bigg(\|\nabla \overline{M}\|^{2}_{L^{2}(\Omega)}+\|\Delta \overline{M}\|^{2}_{L^{2}(\Omega)}+\|\nabla\Delta \overline{M}\|^{2}_{L^{2}(\Omega)}\bigg)
    \notag\\
    & \leq\eps\|\Delta\overline{v}\|^{2}_{L^{2}(\Omega)}+C\big(\|\Delta M_{1}\|^{4}_{L^{2}(\Omega)}+1\big)\|\nabla\overline{M}\|^{2}_{L^{2}(\Omega)}+\eps \|\Delta \overline{M}\|^{2}_{L^{2}(\Omega)}+\eps\|\nabla\Delta\overline{M}\|^{2}_{L^{2}(\Omega)}.
	\end{align}
	For the term $I^{\overline{v}}_{4}$, using \eqref{interpolation} we have
	\begin{align}\label{ssI4}
	|I^{\overline{v}}_{4}|&\leq \|\nabla{M}_{2}\|_{L^{4}(\Omega)}\|\Delta \overline{M}\|_{L^{4}(\Omega)}\|\Delta\overline{v}\|_{L^{2}(\Omega)}
	\notag\\
	& \leq \eps\|\Delta\overline{v}\|^{2}_{L^{2}(\Omega)}+C\|\nabla M_{2}\|^{2}_{L^{4}(\Omega)}\|\Delta\overline{M}\|^{2}_{L^{4}(\Omega)}
	\notag\\
	& \leq \eps\|\Delta\overline{v}\|^{2}_{L^{2}(\Omega)}+C\|\nabla M_{2}\|^{2}_{L^{4}(\Omega)}\|\Delta\overline{M}\|_{L^{2}(\Omega)}\big(\|\Delta\overline{M}\|^{2}_{L^{2}(\Omega)}+\|\nabla\Delta\overline{M}\|^{2}_{L^{2}(\Omega)}\big)^{\frac{1}{2}}
	\notag\\
	& \leq \eps\|\Delta\overline{v}\|^{2}_{L^{2}(\Omega)}+C\|\nabla M_{2}\|^{4}_{L^{4}(\Omega)}\|\Delta\overline{M}\|^{2}_{L^{2}(\Omega)}+\eps\big(\|\Delta\overline{M}\|^{2}_{L^{2}(\Omega)}+\|\nabla\Delta\overline{M}\|^{2}_{L^{2}(\Omega)}\big)
	\notag\\
	& \leq \eps\|\Delta\overline{v}\|^{2}_{L^{2}(\Omega)}+C\big(\|\nabla M_{2}\|^{4}_{L^{4}(\Omega)}+1\big)\|\Delta\overline{M}\|^{2}_{L^{2}(\Omega)}+\eps\|\nabla\Delta\overline{M}\|^{2}_{L^{2}(\Omega)}.
	\end{align}
	For $I^{\overline{v}}_{5},$ using \eqref{smoreinterpole1} and \eqref{interpoledir} we derive the estimate
	\begin{align}\label{ssI5}
	|I^{\overline{v}}_{5}|& \leq \eps\|\Delta\overline{v}\|^{2}_{L^{2}(\Omega)}+C\|\nabla\overline{F}\|^{2}_{L^{4}(\Omega)}\|F_{1}\|^{2}_{L^{4}(\Omega)}+C\|\nabla F_{1}\|^{2}_{L^{2}(\Omega)}\|\overline{F}\|^{2}_{L^{\infty}(\Omega)}
	\notag\\
	& \leq \eps\|\Delta\overline{v}\|^{2}_{L^{2}(\Omega)}+C\|F_{1}\|^{2}_{L^{4}(\Omega)}\big(\|\nabla\overline{F}\|^{2}_{L^{2}(\Omega)}+\|\nabla \overline{F}\|_{L^{2}(\Omega)}\|\Delta\overline{F}\|_{L^{2}(\Omega)}\big)
	\notag\\
	&\qquad+C\|\nabla F_{1}\|^{2}_{L^{2}(\Omega)}\|\overline{F}\|_{L^{2}(\Omega)}\|\Delta\overline{F}\|_{L^{2}(\Omega)}
	\notag\\
	&\leq \eps\|\Delta\overline{v}\|^{2}_{L^{2}(\Omega)}+C\|\nabla\overline{F}\|^{2}_{L^{2}(\Omega)}\|F_{1}\|^{2}_{L^{4}(\Omega)}+C\|\nabla\overline{F}\|_{L^{2}(\Omega)}\|\Delta\overline{F}\|_{L^{2}(\Omega)}\|F_{1}\|^{2}_{L^{4}(\Omega)}
	\notag\\
	&\qquad +\eps\|\Delta\overline{F}\|^{2}_{L^{2}(\Omega)}+C\|\nabla F_{1}\|^{4}_{L^{2}(\Omega)}\|\overline{F}\|^{2}_{L^{2}(\Omega)}
	\notag\\
	& \leq  \eps\|\Delta\overline{v}\|^{2}_{L^{2}(\Omega)}+2\eps\|\Delta\overline{F}\|^{2}_{L^{2}(\Omega)}+C\|\nabla\overline{F}\|^{2}_{L^{2}(\Omega)}\big(\|F_{1}\|^{2}_{L^{4}(\Omega)}+\|F_{1}\|^{4}_{L^{4}(\Omega)}\big)
	\notag\\
	&\qquad+ C\|\nabla F_{1}\|^{4}_{L^{2}(\Omega)}\|\overline{F}\|^{2}_{L^{2}(\Omega)}.
	\end{align}
	Similar to \eqref{ssI5}, the term $I^{\overline{v}}_{6}$ can be estimated as follows:
	\begin{align}\label{ssI6}
	|I^{\overline{v}}_{6}|&\leq \eps\|\Delta\overline{v}\|^{2}_{L^{2}(\Omega)}+2\eps\|\Delta\overline{F}\|^{2}_{L^{2}(\Omega)}+C\|\nabla\overline{F}\|^{2}_{L^{2}(\Omega)}\big(\|F_{2}\|^{2}_{L^{4}(\Omega)}+\|F_{2}\|^{4}_{L^{4}(\Omega)}\big)
	\notag\\
	& \qquad+C\|\nabla F_{2}\|^{4}_{L^{2}(\Omega)}\|\overline{F}\|^{2}_{L^{2}(\Omega)}.
	\end{align}
	Eventually, for the terms $I^{\overline{v}}_{7}$ and $I^{\overline{v}}_{8}$, we obtain the estimates
	\begin{align}\label{ssI7}
	|I^{\overline{v}}_{7}| \leq \|\nabla \overline{H}\|_{L^{2}(\Omega)}\|M_{1}\|_{L^{\infty}(\Omega)}\|\Delta\overline{v}\|_{L^{2}(\Omega)}
	\leq \eps\|\Delta\overline{v}\|^{2}_{L^{2}(\Omega)}+C\|M_{1}\|^{2}_{L^{\infty}(\Omega)}\|\nabla \overline{H}\|_{L^{2}(\Omega)}^{2}
	\end{align}
	and using \eqref{interpolation2}
	\begin{align}\label{ssI8}
	|I^{\overline{v}}_{8}|& \leq \|\Delta\overline{v}\|_{L^{2}(\Omega)}\|\overline{M}\|_{L^{\infty}(\Omega)}\|\nabla H_{2}\|_{L^{2}(\Omega)}
	\notag\\
	&\leq \eps\|\Delta\overline{v}\|^{2}_{L^{2}(\Omega)}+C\|\overline{M}\|_{L^{2}(\Omega)}\big(\|\overline{M}\|^{2}_{L^{2}(\Omega)}+\|\Delta\overline{M}\|^{2}_{L^{2}(\Omega)}\big)^{\frac{1}{2}}\|\nabla H_{2}\|^{2}_{L^{2}(\Omega)}
	\notag\\
	& \leq \eps\|\Delta\overline{v}\|^{2}_{L^{2}(\Omega)} +C\|\nabla H_{2}\|^{2}_{L^{2}(\Omega)}\|\overline{M}\|^{2}_{L^{2}(\Omega)}+C\|\nabla H_{2}\|^{2}_{L^{2}(\Omega)}\|\Delta\overline{M}\|^{2}_{L^{2}(\Omega)}.
	\end{align}

Next, we consider the difference of \eqref{diffvlin4} written for $(v_1,p_1,F_1,M_1)$ and for $(v_2,p_2,F_2,M_2)$. 
After taking the gradient of the resulting equation, we test it with $-\nabla\Delta\overline{M}$. This leads to the identity
\begin{align}\label{strngstabbM}
& \frac{1}{2}\frac{\mathrm d}{\mathrm dt}\int_{\Omega}|\Delta\overline{M}|^{2}+\int_{\Omega}|\nabla\Delta\overline{M}|^{2}
\notag\\
&\quad =\int_{\Omega}\nabla(\overline{v}\cdot\nabla)M_{1}\cdot\nabla\Delta\overline{M}+\int_{\Omega}\nabla(v_{2}\cdot\nabla)\overline{M}\cdot\nabla\Delta\overline{M}
\notag\\
&\qquad
-\int_{\Omega}\nabla\overline{H}\cdot\nabla\Delta\overline{M}
+\int_{\Omega}\nabla\big(-f(M_{1})+f(M_{2})\big)\cdot\nabla\Delta\overline{M}
=:\sum_{i=1}^{4}I^{\overline{M}}_{i},
\end{align}
where $f$ is the function that was introduced in \eqref{fM}.

The term $I^{\overline{M}}_1$ can be estimated as 
\begin{align}\label{IbM1}
|I^{\overline{M}}_{1}|
&\leq \eps\|\nabla\Delta\overline{M}\|^{2}_{L^2(\Omega)}+C\|\nabla\overline{v}\|^{2}_{L^{2}(\Omega)}\|\nabla M_{1}\|^{2}_{L^{\infty}(\Omega)}+C\|\overline{v}\|^{2}_{L^{4}(\Omega)}\|\nabla^{2} M_{1}\|^{2}_{L^{4}(\Omega)}
\notag\\
& \leq\eps\|\nabla\Delta\overline{M}\|^{2}_{L^2(\Omega)}+C\|\nabla\overline{v}\|^{2}_{L^{2}(\Omega)}\|\nabla M_{1}\|^{2}_{L^{\infty}(\Omega)}+C\|\nabla^{2}M_{1}\|^{4}_{L^{4}(\Omega)}\|\nabla\overline{v}\|^{2}_{L^{2}(\Omega)}+C\|\overline{v}\|^{2}_{L^{2}(\Omega)},
\end{align}
where \eqref{L40bnd} was employed to estimate $\|\overline{v}\|_{L^{4}(\Omega)}.$
Using \eqref{smoreinterpole2} and Young's inequality to estimate $\|\nabla\overline{M}\|^{2}_{L^{4}(\Omega)}$, we obtain
\begin{align}\label{IbM2}
|I^{\overline{M}}_{2}|
&\leq \eps\|\nabla\Delta\overline{M}\|^{2}_{L^{2}(\Omega)}+C\|\nabla v_{2}\|^{2}_{L^{4}(\Omega)}\|\nabla\overline{M}\|^{2}_{L^{4}(\Omega)}+C\|v_{2}\|^{2}_{L^{\infty}(\Omega)}\|\nabla^{2}\overline{M}\|^{2}_{L^{2}(\Omega)}
\notag\\
& \leq\eps\|\nabla\Delta\overline{M}\|^{2}_{L^{2}(\Omega)}+C\|\nabla v_{2}\|^{2}_{L^{4}(\Omega)}\big(\|\nabla\overline{M}\|^{2}_{L^{2}(\Omega)}+\|\Delta\overline{M}\|^{2}_{L^{2}(\Omega)}\big)
\notag\\
& \qquad+C\|v_{2}\|^{2}_{L^{\infty}(\Omega)}\big(\|\Delta\overline{M}\|^{2}_{L^2(\Omega)}+\|\overline{M}\|^{2}_{L^{2}(\Omega)}\big)
\notag\\
& \leq \eps\|\nabla\Delta\overline{M}\|^{2}_{L^{2}(\Omega)}+C\|\nabla v_{2}\|^{2}_{L^{4}(\Omega)}\|\nabla\overline{M}\|^{2}_{L^{2}(\Omega)}+C\|\nabla v_{2}\|^{2}_{L^{4}(\Omega)}\|\Delta\overline{M}\|^{2}_{L^{2}(\Omega)}
\notag\\
&\qquad+C\|v_{2}\|^{2}_{L^{\infty}(\Omega)}\|\Delta\overline{M}\|^{2}_{L^{2}(\Omega)}+C\|v_{2}\|^{2}_{L^{\infty}(\Omega)}\|\overline{M}\|^{2}_{L^{2}(\Omega)}.
\end{align}
For $I_{3}^{\overline{M}}$, we obtain the simple estimate
\begin{align}\label{barM3}
|I^{\overline{M}}_{3}|\leq \eps\|\nabla\Delta\overline{M}\|^{2}_{L^{2}(\Omega)}+C\|\nabla \overline{H}\|^{2}_{L^{2}(\Omega)}.
\end{align}
For the term $I^{\overline{M}}_{4}$, we compute the following estimate:
\begin{align}\label{barM4}
|I^{\overline{M}}_{4}|
& \leq \eps\|\nabla\Delta\overline{M}\|_{L^{2}(\Omega)}^{2}+\|\nabla\big(f(M_{1})-f(M_{2})\big)\|^{2}_{L^{2}(\Omega)}
\notag\\
& \leq \eps\|\nabla\Delta\overline{M}\|_{L^{2}(\Omega)}^{2}+C\|\nabla\overline{M}\|^{2}_{L^{2}(\Omega)}+C\|\nabla M_{2}\|^{2}_{L^{\infty}(\Omega)}\|\overline{M}\|^{2}_{L^{2}(\Omega)}.
\end{align}
Here, to infer the third line from the second one, we employed the relations
\begin{equation}\nonumber
\begin{array}{llll}
   \nabla \bigg(f(M_{1})-f(M_{2})\bigg)
   & =f'(M_{1})\left(\nabla M_{1}-\nabla M_{2}\right)+\left(f'(M_{1})-f'(M_{2})\right)\nabla M_{2}
    &&\quad\text{$a.e.$ in $Q_T$},\\
    |f'(M_{1})-f'(M_{2})|
    &\leq C|\overline{M}|
    &&\quad\text{$a.e.$ in $Q_T$}.
    \end{array}
\end{equation} 
The above inequality follows directly from the mean value theorem.\\
Now, we consider the difference of \eqref{diffvlin3} written for $(v_1,p_1,F_1,M_1)$ and for $(v_2,p_2,F_2,M_2)$.
Testing the resulting equation with $-\Delta\overline{F}$, we obtain
 \begin{align}\label{barF}
 &\frac{1}{2}\frac{\mathrm d}{\mathrm dt}\int_{\Omega}|\nabla \overline{F}|^{2}+\kappa\int_{\Omega}|\Delta\overline{F}|^{2}
 \notag\\
 &\;=\int_{\Omega}(\overline{v}\cdot\nabla)F_{1}\cdot\Delta\overline{F}+\int_{\Omega}(v_{2}\cdot\nabla)\overline{F}\cdot\Delta\overline{F}
 -\int_{\Omega}\nabla\overline{v}F_{1}\Delta\overline{F}-\int_{\Omega}\nabla v_{2}\overline{F}\Delta \overline{F}=:\sum_{i=1}^{4}I^{\overline{F}}_{4}.
 \end{align}
 
 Using \eqref{L40bnd} to estimate $\|\overline{v}\|^{2}_{L^{4}(\Omega)}$, we obtain
 \begin{align}\label{IbF1}
 |I^{\overline{F}}_{1}|
 & \leq \eps\|\Delta\overline{F}\|^{2}_{L^{2}(\Omega)}+C\|\overline{v}\|^{2}_{L^{4}(\Omega)}\|\nabla F_{1}\|^{2}_{L^{4}(\Omega)}
 \notag\\
 &\leq \eps\|\Delta\overline{F}\|^{2}_{L^{2}(\Omega)}+C\|\overline{v}\|_{L^{2}(\Omega)}\|\nabla\overline{v}\|_{L^{2}(\Omega)}\|\nabla F_{1}\|^{2}_{L^{4}(\Omega)}
 \notag\\
 & \leq \eps\|\Delta\overline{F}\|^{2}_{L^{2}(\Omega)}+C\|\nabla F_{1}\|^{4}_{L^{4}(\Omega)}\|\nabla\overline{v}\|^{2}_{L^{2}(\Omega)}+C\|\overline{v}\|^{2}_{L^{2}(\Omega)}.
 \end{align}
 The terms $I^{\overline{F}}_{2}$ and $I^{\overline{F}}_{3}$ can be estimated as follows:
 \begin{align}\label{IbF2}
 |I^{\overline{F}}_{2}|
 &\leq \eps\|\Delta\overline{F}\|^{2}_{L^{2}(\Omega)}+C\|v_{2}\|^{2}_{L^{\infty}(\Omega)}\|\nabla\overline{F}\|^{2}_{L^{2}(\Omega)},
 \\
 \label{IbF3}
 |I^{\overline{F}}_{3}|
 &\leq \eps\|\Delta\overline{F}\|^{2}_{L^{2}(\Omega)}+C\|F_{1}\|^{2}_{L^{\infty}(\Omega)}\|\nabla\overline{v}\|^{2}_{L^{2}(\Omega)}.
 \end{align}
 Since $v_{i}$ and $F_{i}$ both satisfy a homogeneous Dirichlet boundary condition, $I^{\overline{F}}_{4}$ can be estimated similarly as $I^{\overline{F}}_{1}$. We thus have
 \begin{align}\label{IbF4}
 |I^{\overline{F}}_{4}|\leq \eps\|\Delta\overline{F}\|^{2}_{L^{2}(\Omega)}+C\|\nabla v_{2}\|^{4}_{L^{4}(\Omega)}\|\nabla\overline{F}\|^{2}_{L^{2}(\Omega)}+C\|\overline{F}\|^{2}_{L^{2}(\Omega)}.
 \end{align}
Now, fixing $\eps>0$ sufficiently small, adding the \eqref{eqdiffvstrng}, \eqref{strngstabbM}, \eqref{barF} and using the above estimates for the terms $I^{\overline{v}}_{i}$,  $I^{\overline{M}}_{i}$, and $I^{\overline{F}}_{i}$ to bound the right-hand side of the resulting equation, we eventually obtain
\begin{align}\label{strngstabpenultimate}
&\frac{1}{2}\frac{\mathrm d}{\mathrm dt}
\overline{\mathcal{Y}}_{s}(t)+\overline{\mathcal{B}}_{s}(t)
\le C \mathcal Q_s(t)\, \overline{\mathcal{Y}}_{s}(t)
+ C \mathcal R_s(t)\, \overline{\mathcal{Y}}(t)
+ C\big(\|M_{1}\|^{2}_{L^{\infty}(\Omega)}+1\big)
    \|\nabla\overline{H}\|^{2}_{L^{2}(\Omega)},
\end{align} 
for almost all $t\in[0,T]$, where
\begin{align}
    \label{defy1B1}
    \overline{\mathcal{Y}}_{s}
    &:=\int_{\Omega}\left(|\nabla\overline{v}|^{2}+|\Delta \overline{M}|^{2}+|\nabla\overline{F}|^{2}\right),
    \\[1ex]
    \overline{\mathcal{B}}_{s}
    &:=\frac{1}{2}\int_{\Omega}\left(\nu|\Delta\overline{v}|^{2}+|\nabla\Delta\overline{M}|^{2}+\kappa|\Delta \overline{F}|^{2}\right)
    \\[1ex]
    \mathcal Q_s 
    &:= \|\nabla M_{1}\|^{2}_{L^{\infty}(\Omega)}
    +\|\nabla^{2}M_{1}\|^{4}_{L^{4}(\Omega)}
    +\|\nabla F_{1}\|^{4}_{L^{4}(\Omega)}
    +\|F_{1}\|^{2}_{L^{\infty}(\Omega)}
    +\|F_{1}\|^{2}_{L^{4}(\Omega)}
    \notag\\
    &\qquad
    +\|F_{1}\|^{4}_{L^{4}(\Omega)}
    +\|F_{2}\|^{2}_{L^{4}(\Omega)}
    +\|F_{2}\|^{4}_{L^{4}(\Omega)}
    +\|v_{2}\|^{2}_{L^{\infty}(\Omega)}+\|\nabla v_{1}\|^{4}_{L^{2}(\Omega)}+\|\nabla v_{2}\|^{4}_{L^{2}(\Omega)}
    \notag\\
    &\qquad
    +\|\nabla v_{2}\|^{4}_{L^{4}(\Omega)}
    +\|\nabla M_{2}\|^{4}_{L^{4}(\Omega)}
    +\|\nabla v_{2}\|^{2}_{L^{4}(\Omega)}+\|\nabla H_{2}\|^{2}_{L^{2}(\Omega)}+1,
    \\[1ex]
    \mathcal R_s 
    &:= \|\Delta M_{1}\|^{4}_{L^{2}(\Omega)}
    +\|\nabla M_{2}\|^{2}_{L^{\infty}(\Omega)}
    +\|\nabla v_{2}\|^{2}_{L^{4}(\Omega)}
    \notag\\
    &\qquad+\|v_{2}\|^{2}_{L^{\infty}(\Omega)}
    +\|\nabla F_{1}\|^{4}_{L^{2}(\Omega)}
    +\|\nabla F_{2}\|^{4}_{L^{2}(\Omega)}
    +\|\nabla H_{2}\|^{2}_{L^{2}(\Omega)}+1
\end{align}
 and $\overline{\mathcal{Y}}$ is as introduced in \eqref{mathcalYB}. Hence, $\overline{\mathcal{Y}}$ can be bounded by means of estimate \eqref{weakeststimate}.
 
 

It remains to show that 
\begin{align}
    \label{EST:QSRS}
    \norm{\mathcal Q_s}_{L^1([0,T])}\le C
    \quad\text{and}\quad
    \norm{\mathcal R_s}_{L^1([0,T])}\le C.
\end{align}
Using Sobolev's embedding theorem as well as interpolation between Sobolev spaces, we conclude that
 \begin{align}
 \label{EST:L4}
 &\|\nabla^{2} M_{1}\|^{4}_{L^1(0,T;L^4(\Omega))}
 \leq C \|\nabla^{2} M_{1}\|^{4}_{L^{4}(Q_{T})}
 \leq C\|\nabla^{2} M_{1}\|^{4}_{L^{4}(0,T;W^{1/2,2}(\Omega))} 
 \notag\\
 &\quad \leq C\|M_{1}\|^{4}_{L^{4}(0,T;W^{5/2,2}(\Omega))} 
 \leq C\| M_{1}\|^{2}_{L^{\infty}(0,T;W^{2,2}(\Omega))}\| M_{1}\|^{2}_{L^{2}(0,T;W^{3,2}(\Omega))}.
 \end{align}
 We point out that the norms appearing on the right-hand side of this inequality can be bounded by the norms of the initial data since $(v_{1},F_{1},M_{1})$ is a strong solution.
 The terms $\|\nabla F_{1}\|_{L^{4}(Q_{T})}$ and $\|\nabla v_{2}\|_{L^{4}(Q_{T})}$ can be estimated analogously.
 All further summands of $\mathcal Q_s$ and $\mathcal R_s$ are relatively easy to deal with and hence one can show \eqref{EST:QSRS}.\\
 Eventually, we add the inequalities \eqref{strngstabpenultimate} and \eqref{beforeGron}, and we apply Gronwall's lemma on the resulting estimate. Using \eqref{EST:QSRS} we conclude the estimate \eqref{strongeststimate} and hence, the proof is complete.
\end{proof}


\section{The control-to-state operator and its properties}\label{Control2state}

In this section, we fix an arbitrary final time $T>0$ as well as initial data $v_{0}\in W^{1,2}_{0,\dvr}(\Omega),$ $F_{0}\in W^{1,2}_{0}(\Omega),$ $M_{0}\in W^{2,2}_{n}(\Omega)$. We further introduce several function spaces to simplify the notation in the subsequent approach:
\begin{align}
\label{DEF:HH}
				\HH &:= L^2(0,T;W^{1,2}(\Omega)),\\
\label{DEF:VV}
				\VV &:= \big[ L^2(0,T;V(\Omega)) 
				    \cap L^\infty(0,T;W^{1,2}_{0,\dvr}(\Omega)) \big]
				    \times L^2(0,T;W^{1,2}(\Omega)) \notag\\
				&\qquad \times \big[ L^2(0,T;W^{2,2}(\Omega)) 
				\cap L^\infty(0,T;W^{1,2}_{0}(\Omega)) \big]
				    \times \big[ L^2(0,T;W^{3,2}(\Omega))
				\cap L^\infty(0,T;W^{2,2}_{n}(\Omega)) \big],\\
\label{DEF:SS}
			    \SSS &:= L^2(0,T;L^2(\Omega)) \times L^2(0,T;L^2(\Omega)) \times L^2(0,T;L^2(\Omega)) \times L^2(0,T;L^2(\Omega)).
\end{align}

The space $\HH$ can be considered as the \emph{space of admissible controls}.
In view of Theorem~\ref{globalstrong} we can define an operator mapping any admissible control $H\in\HH$ to the corresponding solution of the system \eqref{diffviscoelastic*}, the so-called \emph{state}.

\begin{mydef}\label{DEF:CSO} 
    For any field $H\in\HH$, let $(v_H,p_H,F_H,M_H) \in \VV$ denote the unique strong solution of the state equation \eqref{diffviscoelastic*}. The operator 
    \begin{align}
        \FF:\HH\to \SSS,\quad H\mapsto (v_H,p_H,F_H,M_H)
    \end{align}
    is referred to as the \emph{control-to-state operator}.
\end{mydef}

In the following, we will discuss some properties of the control-to-state operator $\FF$ which are essential to investigate the optimal control problems. We point out that actually $\FF(\HH)\subset \VV\subset\SSS$. However, for some of the properties established below ($e.g.$, Fr\'echet differentiability), it is more suitable to use the larger space $\SSS$ in the definition of $\FF$. 

\subsection{Lipschitz continuity}

We first observe that the control-to-state operator $\FF$ is Lipschitz continuous with respect to the norm of $\VV$.
In fact, this is a direct consequence of the strong stability result presented in Theorem~\ref{stabilitiesstrong}.

\begin{corollary} \label{COR:LIP}
    The control-to-state operator $\FF$ is locally Lipschitz continuous with respect to the norm of $\VV$. It even holds that for every $R>0$, there exists a positive constant $L_R>0$ depending only on $R$, $T$, $\Omega$ and the initial data such that for all $H_1,H_2\in\HH$ with $\norm{H_1}_\HH\le R$ and $\norm{H_2}_\HH\le R$ it holds that
    \begin{align}
        \label{LIP}
        \norm{\FF(H_1)-\FF(H_2)}_\VV \le L_R \norm{H_1-H_2}_\HH \, .
    \end{align}
\end{corollary}

\subsection{Weak sequential continuity}

We next show that the control-to-state operator $\FF$ is weakly (sequentially) continuous with respect to the norm of $\VV$, and the components $H\mapsto v_H$, $H\mapsto F_H$, and $H\mapsto M_H$ are even strongly weakly (sequentially) continuous with respect to the norm of $C([0,T];L^2(\Omega))$.

\begin{prop} \label{PROP:WSC}
    The control-to-state operator $\FF$ is sequentially continuous in the following sense: For any sequence $(H_k)_{k\in\N} \in \HH$ with $H_k\wto H^* $ in $\HH$ as $k\to\infty$ it holds that
    \begin{align}
        \label{WSC}
        \begin{aligned}
        &\FF(H_k) \wto \FF(H^*) 
        &&\quad\text{in $\VV$},\\
        &v_{H_k} \to v_{H^*}, \;
        F_{H_k} \to F_{H^*},\;
        M_{H_k} \to M_{H^*}
        &&\quad\text{in $C([0,T];L^2(\Omega))$ and $a.e.$~in $Q_T$}
        \end{aligned}
    \end{align}
    as $k\to\infty$.
\end{prop}

\begin{proof}
    Let $(H_k)_{k\in\N} \in \HH$ be an arbitrary sequence converging weakly in $\HH$ to a limit $H^*\in\HH$, $i.e.$, $H_k\wto H^*\in\HH$ in $\HH$ as $k\to\infty$. Since weakly convergent sequences are bounded, there exists a radius $R>0$ such that $\norm{H_k}_\HH \le R$ for all $k\in\N$. We then infer from \eqref{strongestimate} that the sequence $\FF(H_k)_{k\in\N}$ is bounded in $\VV$. 
    Hence, there exists a quadruplet $\FF^*=(v^*,p^*,F^*,M^*) \in \VV$ such that
    \begin{align}
        \label{CONV:V}
        \FF(H_k)=(v_{H_k},p_{H_k},F_{H_k},M_{H_k}) \wto (v^*,p^*,F^*,M^*) = \FF^*
        \quad\text{in $\VV$ as $k\to\infty$}
    \end{align}
    along a non-relabeled subsequence. 
    Moreover, by a comparison argument in the strong formulation \eqref{diffviscoelastic*}, we infer that the time derivatives are also bounded uniformly in $k$. To be precise, we obtain
    \begin{align}
        \norm{\partial_t v_{H_k}}_{L^2(0,T;L^2(\Omega))}
        + \norm{\partial_t F_{H_k}}_{L^2(0,T;L^2(\Omega))}
        + \norm{\partial_t M_{H_k}}_{L^2(0,T;L^2(\Omega))}
        \le C
    \end{align}
    for some constant $C>0$ depending only on $T$, $R$, $\Omega$ and the initial data.
    Using the Banach--Alaoglu theorem and the Aubin--Lions lemma, we conclude, possibly after another subsequence extraction, that 
    \begin{align}
        \label{CONV:L}
        v_{H_k} \to v_{*}, \;
        F_{H_k} \to F_{*},\;
        M_{H_k} \to M_{*}
        \quad\text{in $C([0,T];L^2(\Omega))$ and $a.e.$~in $Q_T$}.
    \end{align}
    Due to these convergence properties, we can pass to the limit in the weak formulation of \eqref{diffviscoelastic*} to verify that $\FF^*=(v^*,p^*,F^*,M^*)$ is the unique weak solution of \eqref{diffviscoelastic*} to the magnetic field $H^*$ and the given initial data. According to Theorem~\ref{globalstrong}, the solution $\FF^*=(v^*,p^*,F^*,M^*)$ is actually strong, and hence, $\FF^* = \FF(H^*)$. Furthermore, this means that the limit does not depend on the subsequence extractions and thus, the above convergence properties hold true for the whole sequence. In view of \eqref{CONV:V} and \eqref{CONV:L}, this completes the proof.
\end{proof}

\subsection{Fr\'echet differentiability}

To prove Fr\'echet differentiability of the control-to-state operator, we linearize the state equation. 
For any $H\in\HH$, the corresponding state $(v_H,p_H,F_H,M_H)$, and general source terms
$S_1$, $S_2$ and $S_3$ belonging to $L^2(0,T;L^2(\Omega))$,
we consider the following system of equations:
\begin{subequations}
\label{LIN}
\begin{alignat}{2}
	&\partial_{t} \dv 
	+ (\dv\cdot\nabla)v_H + (v_H\cdot\nabla)\dv 
	+ \mbox{div}\left((\nabla \dM \odot\nabla M_H)-\dF F_H^{T}\right) \nonumber\\
	&\qquad + \mbox{div}\left((\nabla M_H\odot\nabla \dM)-F_H \dF^{T}\right)
	+\nabla \dpr 
	= \nu\Delta \dv + (\nabla H)^T \dM + S_1 \;
	&&\mbox{ in } Q_{T},\label{LIN:1}\\[1ex]
	\label{LIN:2}
	&\dvr \dv=0
	&&\mbox{ in } Q_{T},\\[1ex]
	&\partial_{t}\dF 
	+ (\dv\cdot\nabla)F_H + (v_H\cdot\nabla)\dF
	-\nabla \dv F_H -\nabla v_H \dF 
	= \Delta \dF + S_2
	&&\mbox{ in } Q_{T},\label{LIN:3}\\[1ex]
	&\partial_{t}\dM
	+ (\dv\cdot\nabla)M_H + (v_H\cdot\nabla)\dM \nonumber\\
	&\qquad =\Delta \dM 
	- \frac{1}{\alpha^{2}}(|M_H|^{2}-1)\dM 
	- \frac{2}{\alpha^{2}}(\dM\cdot M_H)M_H
	+ S_3
	&&\mbox{ in } Q_{T},\label{LIN:4}\\[1ex]
	\label{LIN:5}
	&\dv=0,\ \dF=0, \ \partial_{n}\dM=0
	&&\mbox{ on } \Sigma_{T},\\[1ex]
	\label{LIN:6}
	&(\dv,\dM,\dF)(\cdot,0)=(0,0,0)
	&& \mbox{ in }\Omega.
\end{alignat}
\end{subequations} 

We next show that the system \eqref{LIN} actually has a unique weak solution.
\begin{prop}\label{WP:LIN}
	Let $H\in\HH$ be arbitrary, let $(v_H,p_H,F_H,M_H)$ denote the corresponding state, and suppose that the source terms $S_1$, $S_2$ and $S_3$ belong to $L^2(0,T;L^2(\Omega))$. Then the system \eqref{LIN} has a unique weak solution
	\begin{align}\label{REG:LIN}
	\left\{ \begin{aligned}
	&  \dv\in W^{1,2}(0,T;(W^{1}_{0,\dvr}(\Omega))') \cap 
	L^{\infty}(0,T;L^{2}_{\dvr}(\Omega))\cap L^{2}(0,T;W^{1,2}_{0,\dvr}(\Omega));\\
	&  \dpr \in L^2(0,T;L^2(\Omega)); \\
	&  \dF\in W^{1,2}(0,T;(W^{1,2}(\Omega))') \cap 
	L^{\infty}(0,T;L^{2}(\Omega))\cap L^{2}(0,T;W^{1,2}_{0}(\Omega));\\
	&  \dM\in  W^{1,2}(0,T;L^2(\Omega)) \cap 
	L^{\infty}(0,T;W^{1,2}(\Omega))\cap L^{2}(0,T;W^{2,2}_{n}(\Omega)).
	\end{aligned}\right.
	\end{align}
	Moreover, there exists a constant $C>0$ depending only on $T$, $\Omega$, $\norm{H}_\HH$ and the initial data of the state, such that
	\begin{align}
	\label{EST:LIN}
	\bignorm{(\dv,\dpr,\dF,\dM)}_\SSS 
	\le C\sum_{i=1}^3 \norm{S_i}_{L^2(0,T;L^2(\Omega))}.
	\end{align}
\end{prop}

The proof of this proposition will be presented in the Appendix. 

We will now establish the Fr\'echet differentiability of the control-to-state operator. In particular, the Fr\'echet derivative can be expressed as the unique weak solution of \eqref{LIN} with a special choice of source terms. 

\begin{prop}
    \label{PROP:FD}
    The control-to-state operator $\FF$ is Fr\'echet differentiable, $i.e.$, for all $H\in\HH$, there exists a linear and bounded operator
    \begin{align*}
        \FF'(H): \HH \to \SSS,\quad
    \end{align*}
    such that 
    \begin{align*}
        \frac{\bignorm{\FF(H+\dH)-\FF(H)-\FF'(H)}_\SSS}{\bignorm{\dH}_\HH} \to 0
        \quad\text{as $\bignorm{\dH}_\HH\to 0$.}
    \end{align*}
    The Frechet derivative at the point $H\in\HH$ in direction $\dH\in\HH$ is then given as
    \begin{align}
        \label{FR:DER}
        \big(v_H'[\dH],p_H'[\dH],F_H'[\dH],M_H'[\dH]\big):= \FF'(H)[\dH] = (\dv,\dpr,\dF,\dM)
    \end{align}
    where the quadruplet $(\dv,\dpr,\dF,\dM)$ is the unique weak solution of the linearized system \eqref{LIN} to the source terms
    \begin{align*}
        S_1 = (\nabla \dH)^T M_H,\quad 
        S_2 = 0,\quad
        S_3 = \dH.
    \end{align*}
\end{prop}

\begin{proof}
    Let us fix an arbitrary field $H\in\HH$. Moreover, let $\dH\in\HH$ be arbitrary, and without loss of generality, we assume that $\tnorm{H-\dH}_\HH< 1$. Hence, defining 
    $R:= \norm{H}_\HH + 1$, we have
    $\tnorm{H}_\HH\le R$ and $\tnorm{\dH}_\HH\le R$, and thus, Corollary~\ref{COR:LIP} can be applied with Lipschitz constant $L_R$.
    Let now $C>0$ denote generic constants that depends only on $T$, $R$, $\Omega$ and the initial data, and may change their value from line to line.
    To prove Fr\'echet differentiability, we have to consider the difference
    \begin{align*}
        (v,p,F,M) 
        &= (v_{H+\dH},p_{H+\dH},F_{H+\dH},M_{H+\dH})
        - (v_H,p_H,F_H,M_H)
    \end{align*}
    To express $(v,p,F,M)$, we expand the nonlinear terms in the state equation. We obtain
    \begin{align*}
        (v_{H+\dH}\cdot\nabla)v_{H+\dH} - (v_H\cdot\nabla)v_H
        &= (v\cdot \nabla)v_H + (v_H\cdot\nabla)v + \RR_1,\\[1ex]
        \dvr\big(\nabla M_{H+\dH} \odot \nabla M_{H+\dH} \big)
        - \dvr\big(\nabla M_{H} \odot \nabla M_{H} \big) 
        &= \dvr\big(\nabla M \odot \nabla M_{H} \big) 
        + \dvr\big(\nabla M_{H} \odot \nabla M \big) + \RR_2, \\[1ex]
        \dvr(F_{H+\dH}F^T_{H+\dH}) - \dvr(F_{H}F^T_{H})
        &= \dvr(F F^T_{H}) + \dvr(F_{H}F^T) + \RR_3,\\[1ex]
        \big[\nabla(H+\dH)\big]^T M_{H+\dH} - (\nabla H)^T M_{H}
        &= (\nabla \dH)^T M_{H} + (\nabla H)^T M + \RR_4,\\[1ex]
        (v_{H+\dH}\cdot\nabla)F_{H+\dH} - (v_H\cdot\nabla)F_H
        &= (v\cdot \nabla)F_H + (v_H\cdot\nabla)F + \RR_5,\\[1ex]
        \nabla v_{H+\dH} F_{H+\dH} - \nabla v_{H} F_{H}
        &= \nabla v F_{H} + \nabla v_{H} F + \RR_6,\\[1ex]
        (v_{H+\dH}\cdot\nabla)M_{H+\dH} - (v_H\cdot\nabla)M_H
        &= (v\cdot \nabla)M_H + (v_H\cdot\nabla)M + \RR_7,\\[1ex]
        \alpha^{-2} \big(|M_{H+\dH}|^2 - 1\big)M_{H+\dH} 
        - \alpha^{-2} \big(|M_{H}|^2 - 1\big)M_{H} 
        &= 2\alpha^{-2} (M\cdot M_H)M_H + \alpha^{-2} \big(|M_{H}|^2-1\big)M + \RR_8
    \end{align*}
    with
    \begin{align*}
        \RR_1 
        &:= \big[(v_{H+\dH} - v_H)\cdot\nabla\big](v_{H+\dH}-v_H),\\
        \RR_2 
        &:= \dvr\big[ \nabla (M_{H+\dH}-M_H) 
            \odot \nabla (M_{H+\dH}-M_H) \big],\\
        \RR_3
        &:= \dvr\big[(F_{H+\dH}-F_H)(F_{H+\dH}-F_H)^T\big],\\
        \RR_4
        &:= (\nabla \dH)^T (M_{H+\dH} - M_H),\\
        \RR_5
        &:= \big[(v_{H+\dH} - v_H)\cdot\nabla\big](F_{H+\dH}-F_H),\\
        \RR_6
        &:= (\nabla v_{H+\dH} - \nabla v_{H})(F_{H+\dH}-F_H),\\
        \RR_7
        &:= \big[(v_{H+\dH} - v_H)\cdot\nabla\big](M_{H+\dH}-M_H),\\
        \RR_8
        &:= \alpha^{-2} \big(|M_{H+\dH}| - |M_{H}|\big)^2 \\
        &\quad + \alpha^{-2} \big(|M_{H+\dH}| + |M_{H}|\big)
        \big(|M_{H+\dH}| - |M_{H}|\big)(M_{H+\dH}-M_{H}).
    \end{align*}
    By means of the estimates \eqref{strongestimate} from Theorem~\ref{globalstrong} and \eqref{LIP} from Corollary~\ref{COR:LIP}, we deduce that
	\begin{align*}
	\norm{\RR_8}_{L^2(0,T;L^2(\Omega))}
	&\le \alpha^{-2} \left(1 + \bignorm{M_{H+\dH}}_{L^\infty(0,T;W^{2,2}(\Omega))}    + \bignorm{M_{H}}_{L^\infty(0,T;W^{2,2}(\Omega))} \right)\\
	&\qquad \cdot \bignorm{M_{H+\dH}-M_{H}}_{L^\infty(0,T;L^4(\Omega))}
	\bignorm{M_{H+\dH}-M_{H}}_{L^2(0,T;L^4(\Omega))}   \\
	&\le C \bignorm{\dH}_\HH^2.
	\end{align*}
	Proceeding similarly with $\RR_i$, $i=1,...,7$, we conclude that
	\begin{align}
	\label{EST:RI}
	\bignorm{\RR_i}_{L^2(0,T;L^2(\Omega))}
	\le C \bignorm{\dH}_\HH^2
	\quad\text{for all $i\in\{1,...,8\}$}.
	\end{align}
	Let now $(\dv,\dpr,\dF,\dM)$ denote the unique weak solution of the system \eqref{LIN} to the source terms
	\begin{align*}
	S_1 = (\nabla \dH)^T M_H,\quad
	S_2 = 0, \quad
	S_3 = \dH,
	\end{align*}
	and let $(v_\RR,p_\RR,F_\RR,M_\RR)$ denote the unique weak solution of \eqref{LIN} to the source terms
	\begin{align*}
	S_1 = - \RR_1 - \RR_2 + \RR_3 + \RR_4 ,\quad
	S_2 = - \RR_5 + \RR_6, \quad
	S_3 = - \RR_7 - \RR_8 .
	\end{align*}
	Due to linearity, and recalling the above considerations, we infer that both $(v,p,F,M)$ and the sum $(\dv,\dpr,\dF,\dM) + (v_\RR,p_\RR,F_\RR,M_\RR)$ are a weak solution of \eqref{LIN} 
	to the source terms
	\begin{align*}
	S_1 = - \RR_1 - \RR_2 + \RR_3 + \RR_4 + (\nabla \dH)^T M_H,\quad
	S_2 = - \RR_5 + \RR_6, \quad
	S_3 = - \RR_7 - \RR_8 + \dH.
	\end{align*}
	Because of uniqueness of the weak solution, this directly implies that
	\begin{align*}
	(v_{H+\dH},p_{H+\dH},F_{H+\dH},M_{H+\dH})
	- (v_H,p_H,F_H,M_H)
	= (\dv,\dpr,\dF,\dM) 
	+ (v_\RR,p_\RR,F_\RR,M_\RR).
	\end{align*}
	Consequently, recalling \eqref{EST:RI} and the estimate \eqref{EST:LIN} from Proposition~\ref{WP:LIN}, we obtain
	\begin{align*}
	\frac{\bignorm{\FF(H+\dH)-\FF(H)-(\dv,\dpr,\dF,\dM)}_\SSS}{\bignorm{\dH}_\HH}
	= \frac{\bignorm{(v_\RR,p_\RR,F_\RR,M_\RR)}_\SSS}{\bignorm{\dH}_\HH}
	\le C \bignorm{\dH}_\HH
	\to 0
	\end{align*}
	as $\|\dH\|_\HH \to 0$. This means that the operator $\FF$ is Fr\'echet differentiable at the point $H\in\HH$, and the Fr\'echet derivative in any direction $\dH\in\HH$ is given by $\FF'(H)[\dH] = (\dv,\dpr,\dF,\dH)$. Thus, the proof is complete.
\end{proof}

\section{Optimal control via unconstrained external magnetic fields}\label{SECT:OC}

In this section we investigate an optimal control problem where the control is represented by the external magnetic field $H\in\HH$ (see \eqref{DEF:HH} for the definition of $\HH$). As no other constraints are imposed on the control $H$, the optimal control problem can be classified as an \emph{unconstrained optimization problem}.

We fix arbitrary $T>0,$ $v_{0}\in W^{1,2}_{0,\dvr}(\Omega),$ $F_{0}\in W^{1,2}_{0}(\Omega),$ $M_{0}\in W^{2,2}_{n}(\Omega)$.
The goal is to control the strong solution $(v,p,F,M)$ of \eqref{diffviscoelastic*} in such a way that the functions $v$, $F$ and $M$ are close to given desired functions $v_d$, $F_d$, and  $M_d$, which belong to $L^2(0,T;L^2(\Omega))$, in a certain sense. 
To formulate this more precisely, let $a_1,a_2,a_3\ge 0$, $\lambda > 0$ be any given real numbers.
We intend to minimize the cost functional
\begin{align}
    \label{DEF:I:1}
    \begin{aligned}
    I(v,p,F,M,H) &:= 
    \frac{a_1}{2} \norm{v-v_d}_{L^2(Q_T)}^2
    + \frac{a_2}{2}  \norm{F-F_d}_{L^2(Q_T)}^2 
     + \frac{a_3}{2}  \norm{M-M_d}_{L^2(Q_T)}^2
    + \frac{\lambda}{2}  \norm{H}_{\HH}^2
    \end{aligned}
\end{align}
subject to the constraints\\[1ex]
$\phantom{.}\quad \bullet \quad H\in\HH$, $i.e.$, $H$ is an admissible control;\\
$\phantom{.}\quad \bullet \quad (v,p,F,M)$ is the unique strong solution of the state equation \eqref{diffviscoelastic*} to the control $H$.\\[1ex]
By means of the control-to-state operator $\FF$, we can equivalently formulate this problem as
\begin{align}
    \label{OCP}
    \left\{
    \begin{aligned}
        &\text{Minimize} && J(H) := I(\FF(H),H),\\
        &\text{subject to} && H\in\HH.
    \end{aligned}
    \right.
\end{align}
This is referred to as the \emph{reduced formulation} of the optimal control problem, and $J$ is called the \emph{reduced cost functional}.

Exploiting the properties of the control-to-state operator established in Section \ref{Control2state}, we will first show in Section \ref{existenceopcontrol} that the optimal control problem has at least one (globally) optimal solution. Of course, since our optimization problem is non-convex (as $\FF$ is a nonlinear operator), such an optimal solution will usually not be unique. There might be more than one globally optimal solution but also several locally optimal solutions. In general, due to the non-convex structure, numerical methods will not be able to detect a globally optimal solution but only find a local one. To this end we will derive a characterization of locally optimal solutions by necessary first-order optimality conditions in Section \ref{Forderopcond}. 

\subsection{Existence of an optimal control}\label{existenceopcontrol}

In the following, we will frequently use the spaces $\HH$, $\VV$ and $\SSS$ that were introduced in \eqref{DEF:HH}--\eqref{DEF:SS}.

\begin{thm} \label{THM:EX:1}
    The optimal control problem \eqref{OCP} has at least one (globally) optimal solution $H^*\in\HH$, $i.e.$, it holds that $J(H^*) \le J(H)$ for all $H\in\HH$.
\end{thm}

\begin{proof}
    To prove the assertion, we employ the direct method of the calculus of variations. We first notice that $J$ is nonnegative, and thus, the infimum
    \begin{align*}
        J^* := \underset{H\in\HH}{\inf}\; J(H)
    \end{align*}
    exists. Consequently, there exists a minimizing sequence $(H_k)_{k\in\N}$ such that $J(H_k)\to J^*$ as $k\to\infty$. In particular, this means that
    \begin{align*}
        \norm{H_k}_\HH \le \frac{2}{\lambda}J(H_k) \le \frac{2}{\lambda}\big( J^* + 1 \big)
    \end{align*}
    if $k\in\N$ is sufficiently large. 
    Hence, there exists a field $H^*\in\HH$ such that $H_k\wto H^*$ in $\HH$ as $k\to\infty$ along a non-relabeled subsequence. From Proposition~\ref{PROP:WSC} we infer that $\FF(H_k)\to \FF(H^*)$ in $\SSS$. Eventually, due to weak lower semicontinuity of the norms, we obtain
    \begin{align*}
        J(H^*) = I\big(\FF(H^*),H^*\big) \le \underset{k\to\infty}{\lim\inf}\; I\big(\FF(H_k),H_k\big)
        = \underset{k\to\infty}{\lim} J(H_k) = J^*
    \end{align*}
    which yields $J^*=J(H^*)$ as $J^*$ was defined as the infimum. This means that $J$ attains its minimum at $H^*\in\HH$ and thus, the proof is complete.
\end{proof}

\subsection{First-order necessary optimality conditions}\label{Forderopcond}

We now derive first-order necessary optimality conditions for locally optimal solutions ($i.e.$, local minimizers of the cost functional). Since the control-to-state operator $\FF$ is Fr\'echet differentiable, so is the cost functional $J$ due to the chain rule. The Fr\'echet derivative at any point $H\in\HH$ can be written as
\begin{align}
    \label{FDJ}
    \begin{aligned}
        J'(H)[\dH]&= \lambda \big( \nabla H, \nabla \dH \big)_{L^2(Q_T)}
            + \lambda \big( H, \dH \big)_{L^2(Q_T)} 
            + a_1 \big(v_{H}-v_d,v_{H}'[\dH]\,\big)_{L^2(Q_T)}\\
        &\qquad
            + a_2 \big(F_{H}-F_d,F_{H}'[\dH]\,\big)_{L^2(Q_T)} 
            + a_3 \big(M_{H}-M_d,M_{H}'[\dH]\,\big)_{L^2(Q_T)}
    \end{aligned}
\end{align}
for all directions $\dH\in\HH$. If $H\in\HH$ is a locally optimal solution, it directly follows that the derivative of $J$ at $H$ necessarily vanishes, $i.e.$,
\begin{align}
    \label{NOC:SENS}
    J'(H)[\dH] = 0 \quad\text{for all $\dH\in\HH$.}
\end{align}
Note that \eqref{FDJ} does not provide an explicit description of the derivative $J'(H)$ since in some of the summands, the direction $\dH$ appears only implicitly within the inner products. However, many computational methods for solving such optimal control problems numerically require an explicit representation of the derivative to compute a suitable descent direction.
We notice that if we could find an adjoint operator 
\begin{align*}
    \big(\FF'(H)\big)^* = \big( (v'_{H})^* , (p'_{H})^*, (F'_{H})^*, (M'_{H})^* \big)
\end{align*}
we could rewrite the condition \eqref{NOC:SENS} as
\begin{align}
    \label{NOC:ALT}
    \begin{aligned}
        J'(H)[\dH]&= \lambda \big( \nabla H, \nabla \dH \big)_{L^2(Q_T)}
            + \lambda \big( H, \dH \big)_{L^2(Q_T)} \\
        &\;\;+ \Big(a_1 (v'_{H})^*[v_{H}-v_d] + a_2 (F'_{H})^*[F_{H}-F_d]
            + a_3 (M'_{H})^*[M_{H}-M_d] 
            \,,\,
            \dH \,\Big)_{L^2(Q_T)}.
    \end{aligned}
\end{align}
Following a standard approach in optimal control theory, we intend to express the first argument of the inner product in the second line of \eqref{NOC:ALT} by means of so-called \emph{adjoint variables}. They can be constructed as the solution of a certain \emph{adjoint system}, which can be derived via the \emph{formal Lagrangian technique} (see, $e.g.$, \cite{Troltzsch}).
We will write $(w,q,G,N)$ to denote the adjoint variables.
For any given field $H\in\HH$ and the corresponding state $\FF(H) = (v_H,p_H,F_H,M_H)$, our adjoint system reads as 
\begin{subequations}
\label{ADJ}
\begin{alignat}{2}
    \label{ADJ:1}
	&\partial_{t} w + (v_H \cdot \nabla) w - (\nabla v_H)^T \, w + \nu \Delta w + \nabla q \notag \\
	&\qquad = (\nabla F_H)^T G + \dvr(G F_H^T) + (\nabla M_H)^T N 
	    - a_1(v_H-v_d) \;\; 
	&&\mbox{ in } Q_{T},\\[1ex]
	\label{ADJ:2}
	&\dvr w=0
	&&\mbox{ in } Q_{T},\\[1ex]
	\label{ADJ:3}
	&\partial_{t} G + (v_H\cdot\nabla) G + (\nabla v_H)^T G + \Delta G = 2\,\D w\, F_H 
	    - a_2(F_H-F_d)
	&&\mbox{ in } Q_{T},\\[1ex]
	\label{ADJ:4}
	&\partial_{t} N + (v_H\cdot \nabla) N - 2\alpha^{-2} (M_H\cdot N) M_H  
	    - \alpha^{-2} \big(|M_H|^2 - 1\big)N + \Delta N \notag\\
	&\qquad =  2 \dvr( \nabla M_H\, \D w) 
	    - \nabla H\, w 
	    - a_3(M_H-M_d)
	&&\mbox{ in } Q_{T},\\[1ex]
	\label{ADJ:5}
	&w =0,\ G=0,\ \partial_{n} N =0
	&&\mbox{ on } \Sigma_{T},\\[1ex]
	\label{ADJ:6}
	&(w,G,N)(\cdot,T)=(0,0,0)
	&& \mbox{ in }\Omega.
\end{alignat}
\end{subequations} 
Here 
\begin{align*}
    \D w := \frac 12 \big( \nabla w + (\nabla w)^T \big) 
\end{align*}
denotes the symmetrized gradient of the function $w$. We further point out that the term $(\nabla F)^T$ stands for the transpose of the associated linear map, and thus
\begin{align*}
    \big[(\nabla F_H)^T G\big]_i = \sum_{j,k=1}^2 [\partial_i F_H]_{jk}\, G_{jk}, \quad i=1,...,d.
\end{align*}

We first need to ensure that the adjoint system \eqref{ADJ} is well-posed. The following proposition is a direct consequence of Proposition~\ref{WP:ADJ*} which is established in the Appendix. 
\begin{prop}\label{WP:ADJ}
    Let $H\in\HH$ be arbitrary with corresponding state $\FF(H)=(v_H,p_H,F_H,M_H)$. Then the system \eqref{ADJ} has a unique weak solution $(w,q,G,N)$ having the regularity
    \begin{align}\label{REG:ADJ}
    	\left\{ \begin{aligned}
    	&  w\in L^{2}(0,T;V(\Omega))
    	\cap L^{\infty}(0,T;W^{1,2}_{0,\dvr}(\Omega))
    	\cap W^{1,2}(0,T;L^2_{\dvr}(\Omega));\\
    	&  q \in L^2(0,T;W^{1,2}(\Omega)); \\
    	&  G \in L^{2}(0,T;W^{2,2}(\Omega))
    	\cap L^{\infty}(0,T;W^{1,2}_{0}(\Omega))
    	\cap W^{1,2}(0,T;L^2(\Omega));\\
    	&  N \in  
    	L^{\infty}(0,T;L^{2}(\Omega))\cap L^{2}(0,T;W^{1,2}(\Omega)) 
    	\cap L^{\frac{3}{2}}(0,T;W^{2,\frac{3}{2}}(\Omega))\cap W^{1,\frac{3}{2}}(0,T;L^{\frac{3}{2}}(\Omega)).
    	\end{aligned}\right.
	\end{align}
	This weak solution $(w,q,G,N)$ is called the \emph{adjoint state} or \emph{costate} associated with the field $H$ and the state $\FF(H)$.
\end{prop}

Similar to the definition of the control-to-state operator $\FF$, the above well-posedness result allows us to define an operator mapping any field $\HH$ onto its corresponding adjoint state.

\begin{mydef}
\label{DEF:CCS}
    For any field $H\in\HH$, let $\FF(H)=(v_H,p_H,F_H,M_H)$ denote the corresponding state, and let $(w_H,q_H,G_H,N_H)$ denote the corresponding adjoint state. We define
    \begin{align}
        \label{DEF:ASO}
        \AD:\HH\to \SSS, \quad H\mapsto (w_H,q_H,G_H,N_H)
    \end{align}
    which we refer to as the \emph{control-to-costate operator}.
\end{mydef}

The adjoint variables can now be used to provide an explicit representation of the Fr\'echet derivative of $J$. This description can then be used to reformulate the first-order necessary optimality condition \eqref{NOC:SENS}.

\begin{thm}
\label{THM:NOC}
For any $H\in\HH$, let $\FF(H) = (v_{H},p_{H},F_{H},M_{H})$ denote the corresponding state and let $\AD(H) = (w_{H},q_{H},G_{H},N_{H})$ denote the corresponding adjoint state.
\begin{enumerate}[label = $\mathrm{(\alph*)}$, ref = $\mathrm{(\alph*)}$]
    \item The Fr\'echet derivative of the cost functional $J$ at any point $H\in\HH$ satisfies
    \begin{align}
        \label{DJW}
        \begin{aligned}
        J'(H)[\dH] &= \lambda \big( \nabla H, \nabla \dH \big)_{L^2(Q_T)}
            + \lambda \big( H, \dH \big)_{L^2(Q_T)} \\
        &\quad + \big( N_H  - \nabla M_H w_H , \dH \big)_{L^2(Q_T)}
        \end{aligned}
    \end{align}
    for all $\dH\in\HH$, meaning that
    \begin{align}
        \label{DJ}
        J'(H) =  -\lambda\Delta_\mathrm{N} H + \lambda H + N_H -\nabla M_H w_H \in \HH',
    \end{align}
    where $\Delta_\mathrm{N}$ is to be understood as the Laplace--Neumann operator.
    \item Suppose that $H^*\in\HH$ is a locally optimal solution of the optimal control problem \eqref{OCP}. Then it necessarily holds that
    \begin{align}
        \label{NOC:H:WF}
        \begin{aligned}
        &\lambda \big( \nabla H^*, \nabla \dH \big)_{L^2(Q_T)}
            + \lambda \big( H^*, \dH \big)_{L^2(Q_T)} \\
        &\qquad + \big( N_{H^*}  - \nabla M_{H^*} w_{H^*} , \dH \big)_{L^2(Q_T)} = 0
        \end{aligned}
    \end{align}
    for all $\dH\in\HH$. This entails that $H^*\in L^2(0,T;W^{2,2}(\Omega))$ is a strong solution of the semilinear vector-valued Helmholtz equation
    \begin{subequations}
    \label{NOC:H:SF}
    \begin{alignat}{2}
        \label{NOC:H:HH}
        -\Delta H^* + H^* 
        &=  \frac 1\lambda \big(\nabla M_{H^*} w_{H^*} - N_{H^*} \big)
        &&\quad\text{in $Q_T$,}\\
        \label{NOC:H:BC}
        \partial_n H^* &= 0
        &&\quad\text{on $\Sigma_T$.}
    \end{alignat}
    \end{subequations}
\end{enumerate}
\end{thm}

\medskip

\begin{remark}\label{REM:3D}
    \textnormal{(a)}
    Provided that the global strong-well posedness of system~\eqref{diffviscoelastic*} established in Theorem~\ref{globalstrong} and the stability estimates in Theorem~\ref{weakeststimate} and Theorem~\ref{stabilitiesstrong} could be established in the three-dimensional setting, the results in Section~5 would also hold true.
    Moreover, the optimal control problem \eqref{OCP} could then also be investigated in three dimensions and the results of the present section would still be valid as they do not depend on the dimension. In particular, the optimality conditions presented in Theorem~\ref{THM:NOC} would remain the same.
    
    \textnormal{(b)} In principle, the theory of Section~5 and Section~6 would also remain true if $\HH=L^2(0,T;W^{1,2}(\Omega))$ were replaced by a linear subspace of $L^2(0,T;W^{1,2}(\Omega))$. The only exceptions are the representations \eqref{DJ} and \eqref{NOC:H:SF}. They would not remain valid since under additional linear restrictions on the controls we would not recover the homogeneous Neumann boundary condition from the variational formulations \eqref{DJW} and \eqref{NOC:H:WF}, respectively.
\end{remark}

\begin{proof}[Proof of Theorem~\ref{THM:NOC}]
    To prove (a), let $H,\dH \in \HH$ be arbitrary. For more clarity, we will write
    \begin{align*}
        (v,p,F,M) := \FF(H), 
        \quad
        (w,q,G,N) := \AD(H),
        \quad
        (\dv,\dpr,\dF,\dM) := \FF'(H)[\dH].
    \end{align*}
    Then, according to \eqref{FDJ}, the Fr\'echet derivative of the cost functional $J$ at the point $H$ in direction $\dH$ can be expressed as
    \begin{align*}
        J'(H)[\dH]&= \lambda \big( \nabla H, \nabla \dH \big)_{L^2(Q_T)}
            + \lambda \big( H, \dH \big)_{L^2(Q_T)} 
            \\
        &\qquad
            + \big(a_1(v-v_d),\dv\big)_{L^2(Q_T)}
            + \big(a_2(F-F_d),\dF\big)_{L^2(Q_T)} 
            + \big(a_3(M-M_d),\dM\big)_{L^2(Q_T)}.
    \end{align*}
    Replacing the terms $a_1(v-v_d)$, $a_2(F-F_d)$ and $a_3(M-M_d)$ by means of the adjoint equations \eqref{ADJ:1}, \eqref{ADJ:3} and \eqref{ADJ:4} 
    , we obtain
    \begin{align*}
        J'(H)[\dH] 
        &= \lambda \big( \nabla H, \nabla \dH \big)_{L^2(Q_T)}
            + \lambda \big( H, \dH \big)_{L^2(Q_T)} 
            \\
        &\quad
        - \int_0^T\int_\Omega 
            \Big( \partial_{t} w + (v \cdot \nabla) w - (\nabla v)^T \, w  + \nu \Delta w + \nabla q 
	        - (\nabla F)^T G 
	        \\
	    &\quad\qquad\qquad\qquad - \dvr(G F^T)  - (\nabla M)^T N \Big)\cdot
	        \dv \dx\dt \\
	    &\quad
	    - \int_0^T\int_\Omega \Big( \partial_{t} G + (v\cdot\nabla) G  + (\nabla v)^T G 
	        + \Delta G - 2\,\D w\, F \Big)\, 
	        \cdot \dF \dx\dt \\
	    &\quad
	    - \int_0^T\int_\Omega \Big( \partial_{t} N + (v\cdot \nabla) N 
	        - 2\alpha^{-2} (M\cdot N) M  
	        - \alpha^{-2} \big(|M|^2 - 1\big)N \\
	    &\quad\qquad\qquad\qquad 
	        + \Delta N 
	        - 2 \dvr( \nabla M_H\, \D w) 
	        - \nabla H \, w \Big)\, 
	        \cdot \dM \dx\dt.
    \end{align*}
    Recalling that the functions $(w,q,G,N)$ further satisfy \eqref{ADJ:2}, \eqref{ADJ:5} and \eqref{ADJ:6}, whereas the functions $(\dv,\dpr,\dF,\dM)$ fulfill \eqref{LIN:2}, \eqref{LIN:5} and \eqref{LIN:6} we conclude via integration by parts that
    \begin{align*}
        J'(H)[\dH] 
        &= \lambda \big( \nabla H, \nabla \dH \big)_{L^2(Q_T)}
            + \lambda \big( H, \dH \big)_{L^2(Q_T)} 
            \\
        &\;\;
        + \int_0^T\int_\Omega 
            \Big( \partial_{t} \dv 
        	+ (\dv\cdot\nabla)v + (v\cdot\nabla)\dv 
        	+ \dvr\big((\nabla \dM \odot\nabla M)-\dF F^{T}\big) \\
        &\;\;\qquad\qquad\qquad
        	+ \mbox{div}\big((\nabla M\odot\nabla \dM)-F \dF^{T}\big)
        	+\nabla \dpr 
        	 - (\nabla H)^T \dM \Big)\,
	        \cdot w + \nu \nabla \dv \cdot \nabla w\dx\dt \\
	    &\;\;
	   + \int_0^T\int_\Omega \Big( \partial_{t}\dF 
	        + (\dv\cdot\nabla)F + (v\cdot\nabla)\dF
	        -\nabla \dv F -\nabla v \dF \Big)\, 
	        \cdot G + \nabla \dF \cdot \nabla G\dx\dt \\
	    &\;\;
	    + \int_0^T\int_\Omega \Big( \partial_{t}\dM
	        + (\dv\cdot\nabla)M + (v\cdot\nabla)\dM 
	        - \Delta \dM 
	        + \frac{1}{\alpha^{2}}(|M|^{2}-1)\dM \\
	   &\;\;\qquad\qquad\qquad
	        + \frac{2}{\alpha^{2}}(\dM\cdot M)M
	        \Big)\, 
	        \cdot N \dx\dt.
    \end{align*}
    Since, according to Proposition~\ref{PROP:FD}, $(\dv,\dpr,\dF,\dM)$ is the unique weak solution of the linearized system \eqref{LIN} with $S_1 = (\nabla \dH)^T M$, $S_2 = 0$ and $S_3 = \dH$, the integrands in the above identity can be replace by means of \eqref{LIN:1}, \eqref{LIN:3} and \eqref{LIN:4}. We eventually obtain
    \begin{align*}
        J'(H)[\dH] &= \lambda \big( \nabla H, \nabla \dH \big)_{L^2(Q_T)}
            + \lambda \big( H, \dH \big)_{L^2(Q_T)} 
        + \big( N  - \nabla M w , \dH \big)_{L^2(Q_T)}.
    \end{align*}
    Recalling \eqref{NOC:SENS}, this proves \eqref{DJW}, and the representation \eqref{DJ} directly follows.
    
    To prove (b), we assume that $H^*\in\HH$ is a locally optimal solution. Hence, according to \eqref{NOC:SENS}, we know that $J'(H^*)[\dH]=0$ for all $\dH\in\HH$. Expressing $J'(H^*)[\dH]$ via \eqref{DJW} (written for $H^*$ instead of $H$), we obtain \eqref{NOC:H:WF}. This means that $H^*$ is a weak solution of the problem \eqref{NOC:H:SF}. Recalling the regularity of the state variables (see \eqref{strngsol}) and the adjoint variables (see \eqref{REG:ADJ}), we infer that the right-hand side of the Helmholtz equation \eqref{NOC:H:HH} belongs to $L^2(Q_T)$. Using elliptic regularity theory, we conclude that $H^*\in L^2(0,T;W^{2,2}(\Omega))$. This means that $H^*$ is actually a strong solution of \eqref{NOC:H:SF}. Thus, (b) is established and the proof is complete.
\end{proof}

\section{Optimal control via fixed magnetic field coils}

In real applications, it might not be possible to create the magnetic field ad libitum as assumed in Section~\ref{SECT:OC}.
Therefore, in this section we investigate an optimal control problem where the external magnetic field is not the control itself but rather it is generated by a finite number $n\in\N$ of fixed magnetic field coils. This means that the geometry ($i.e.$, shape and position) of the field coils is not going to be optimized, but only the intensities of their generated magnetic fields are to be adjusted. 

We assume that the magnetic field of the $i$-th field coil is given as
\begin{align*}
    H_i: Q_T\to \R^3,\quad (x,t) \mapsto H_i(x,t):= u_i(t) h_i(x).
\end{align*}
Here, the factor
\begin{align*}
    u_i: [0,T] \to \R,\quad t \mapsto u_i(t)
\end{align*}
is related to the intensity of the magnetic field, $i.e.$, is proportional to the current of electricity that flows through the $i$-th coil.
The factor 
\begin{align*}
    h_i: \Omega \to \R^3,\quad x\mapsto h_i(x).
\end{align*}
is related to the geometry of the $i$-th field coil.
It is a common technique in physics and engineering science to compute the function $h_i$ via the Biot-Savart law which is a magnetostatic approximation of Maxwell's equations. 
For more details, we refer the reader to \cite{knopf-weber} where such an ansatz was made for the optimal control of a plasma via external magnetic field coils.

By linear superposition, the total external magnetic field $H$ can be expressed as
\begin{align*}
    H(x,t) = \sum_{i=1}^n u_i(t)\, h_i(x)
\end{align*}
for all $t\in[0,T]$ and $x\in\Omega$.

In the optimal control problem, we assume that the $h_i$, $i=1,...,n$ are prescribed functions that belong to $H^1(\Omega;\R^3)$. 
The vector $u=(u_1,...,u_N)$ of intensity functions will now represent the control parameters. 
It is supposed to belong to the \emph{set of admissible control parameters} which is defined as
\begin{align}
    \label{SAC}
    \UAD := \left\{ 
    u\in L^2(0,T;\R^n) 
    \suchthat 
    \begin{aligned} 
    &a_i(t) \le u_i(t) \le b_i(t) \;\; 
    \text{for all $i\in\{1,...,n\}$} \\ 
    &\text{and almost all $t\in[0,T]$}
    \end{aligned}
    \right\}\;,
\end{align}
where $a,b\in L^2(0,T;\R^n)$ are given functions with $a_i\le b_i$ for all $i\in\{1,...,n\}$ and almost all $t\in[0,T]$. Note that the set $\UAD$ is a bounded, closed, convex subset of $L^2(0,T;\R^n)$ and thus, it is weakly sequentially compact. Due to the boundedness of $\UAD$, there exists a radius $r>0$ (depending on $a$ and $b$) such that 
\begin{align*}
    \UAD \subset \UR := \big\{ u\in L^2(0,T;\R^n) \suchthat \norm{u}_{L^2(0,T;\R^n)} < r \big\}.
\end{align*}
This further implies the existence of a radius $R>0$ (depending on $r$ and $h_i$, $i=1,...,n$) such that the corresponding field $H\in\HH$ satisfies $\norm{H}_\HH < R$. 

To formulate and analyze the optimal control problem, we first define several operators.
\begin{mydef}
    We define the operators
    \begin{gather}
        \CC: \UR \to \HH, \quad \CC(u)(x,t) = \sum_{i=1}^n u_i(t)\, h_i(x),\\
        \CS := \FF\circ\CC: \, \UR \to \SSS, \quad
        \CA := \AD\circ\CC: \, \UR \to \SSS.
    \end{gather}
\end{mydef}

Note that the operator $\CC$ is linear and bounded. Hence, the results on $\FF$ established in Section~3 can easily be adapted to the operator $\CS$.

We now fix arbitrary $T>0,$ $v_{0}\in W^{1,2}_{0,\dvr}(\Omega),$ $F_{0}\in W^{1,2}_{0}(\Omega),$ $M_{0}\in W^{2,2}_{n}(\Omega)$.
Moreover, let $v_d$, $F_d$, and $M_d$ be given functions belonging to $L^2(0,T;L^2(\Omega))$, and let $a_1,a_2,a_3\ge 0$ and $\lambda>0$ be given real numbers.
In the spirit of Section~\ref{SECT:OC}, we now want to study the following (reduced) optimal control problem:
\begin{align}
    \label{OCP:2}
    \left\{
    \begin{aligned}
        &\text{Minimize} && \tilde J(u) := \tilde I(\CS(u),u),\\
        &\text{subject to} && u\in\UAD.
    \end{aligned}
    \right.
\end{align}
Here, the functional $I$ is defined as
\begin{align}
    \label{DEF:I:2}
    \begin{aligned}
    \tilde I(v,p,F,M,u) &:= 
    \frac{a_1}{2} \norm{v-v_d}_{L^2(Q_T)}^2
    + \frac{a_2}{2} \norm{F-F_d}_{L^2(Q_T)}^2 \\
    &\qquad + \frac{a_3}{2} \norm{M-M_d}_{L^2(Q_T)}^2
    + \frac{\lambda}{2} \norm{u}_{L^2(0,T;\R^n)}^2.
    \end{aligned}
\end{align}
As in Section~\ref{SECT:OC}, we first prove that our optimal control problem \eqref{OCP:2} has at least one globally optimal solution. Then we derive first-order necessary optimality conditions for locally optimal solutions.

\subsection{Existence of an optimal control}

\begin{thm}
    \label{THM:EX:2}
    The optimal control problem \eqref{OCP:2} has at least one (globally) optimal solution $u^*\in\UAD$, $i.e.$, it holds that $\tilde J(u^*) \le \tilde J(u)$ for all $u\in\UAD$.
\end{thm}

\begin{proof}
We recall that the operator $\CC$ is linear and bounded, and the control-to-state operator is weakly sequentially continuous (see Proposition~\ref{PROP:WSC}). It is thus easy to see that the operator $\CS$ is also weakly sequentially continuous. As the set $\UAD$ is a bounded, closed, convex subset of the Hilbert space $L^2(0,T;\R^n)$, it follows that $\UAD$ is weakly sequentially compact (see \cite[Thm.~2.11]{Troltzsch}). Hence, the proof can be completed by proceeding exactly as in the proof of Theorem~\ref{THM:EX:1}.
\end{proof}

\subsection{First-order necessary optimality conditions}

\begin{thm}
\label{THM:NOC:2}
For any $u\in\UR$, let $\CS(u) = (v_{\CC(u)},p_{\CC(u)},F_{\CC(u)},M_{\CC(u)})$ denote the corresponding state and let $\CA(u) = (w_{\CC(u)},q_{\CC(u)},G_{\CC(u)},N_{\CC(u)})$ denote the corresponding adjoint state.\\[1ex]
We define the operator 
\begin{align*}
    &\mathcal D: \UR \to L^2(0,T;\R^n),\quad
    u\mapsto \mathcal D(u)
    := \big(\mathcal D_1(u),...,\mathcal D_n(u)\big)^T\\
    &\text{with}\quad \mathcal D_i(u):= 
    \int_\Omega 
        \big( N_{\CC(u)}  
        - \nabla M_{\CC(u)} w_{\CC(u)} \big) \cdot h_i
    \dx,
    \quad i=1,...,n.
\end{align*}
Then the following holds:
\begin{enumerate}[label = $\mathrm{(\alph*)}$, ref = $\mathrm{(\alph*)}$]
    \item The Fr\'echet derivative of the cost functional $\tilde J$ at any point $u\in\UR$ satisfies
    \begin{align}
        \label{DJW:2}
        \begin{aligned}
        \tilde J'(u)[\du] = \int_0^T \big(\lambda u + \mathcal D(u) \big) \cdot \du \dt
        \quad\text{for all $\du \in L^2(0,T;\R^n)$,}
        \end{aligned}
    \end{align}
    meaning that
    \begin{align}
        \label{DJ:2}
        \begin{aligned}
        \tilde J'(u) &= \lambda u + \mathcal D(u) 
        \;\in L^2([0,T];\R^n).
        \end{aligned}
    \end{align}
    \item Suppose that $u^*\in\UAD$ is a locally optimal solution of the optimal control problem \eqref{OCP}. Then for $u^*$ necessarily satisfies the \emph{variational inequality}
    \begin{align}
        \label{NOC:VI}
        \int_0^T \big(\lambda u^* + \mathcal D(u^*) \big) \cdot (u-u^*) \dt \; 
        \ge 0
        \quad\text{for all $u \in \UAD$.}
    \end{align}
    As a consequence, $u^*$ can be described by the $L^2([0,T];\R^n)$-orthogonal projection of $-\lambda^{-1} \mathcal D(u^*)$ onto the set $\UAD$. This means that for all $i\in\{1,...,n\}$, the $i$-th component $u_i^*$ can be expressed by the \emph{projection formula}
    \begin{align}
        \label{PROJ}
        u_i^*(t) 
        = \mathcal P_{[a_i(t),b_i(t)]}\left( -\frac{1}{\lambda}
        \mathcal D_i\big(u^*(t)\big)
        \right)
        \quad\text{for almost all $t\in [0,T]$.}
    \end{align}
    Here, for any real numbers $c\le d$, the function $\mathcal P_{[c,d]}$ denotes the projection of $\R$ onto the interval $[c,d]$ that is given by
    \begin{align*}
        \mathcal P_{[c,d]}(s) = \max\big\{ c , \min\{s,d\} \big\}, 
        \quad s\in\R.
    \end{align*}
\end{enumerate}
\end{thm}

\medskip

\begin{remark}
    As already discussed in Remark~\ref{REM:3D}, the optimal control problem \eqref{OCP:2} could also be investigated in three dimensions provided that the global strong well-posedness of the state equation \eqref{diffviscoelastic*} would be known and that the strong stability estimate could also be established.
    In this case, the necessary optimality conditions presented in Theorem~\ref{THM:NOC:2} would remain the same.
\end{remark}

\begin{proof}[Proof of Theorem~\ref{THM:NOC:2}]
    Since $\CC$ is a linear and bounded operator, it is continuously Fr\'echet differentiable with $\CC'(u)[\du] = \CC(\du)$ for all $u\in\UR$ and $\du\in L^2(0,T;\R^n)$. By the chain rule, this implies the functional $\tilde J$ is also continuously Fr\'echet differentiable. For any arbitrary $u\in\UR$ and $\du\in L^2(0,T;\R^n)$, we have
    \begin{align*}
        \tilde J'(u)[\du] 
        &= a_1 \big( v_{\CC(u)} - v_d , v'_{\CC(u)}[\CC(\du)] \big)_{L^2(Q_T)}
        + a_2 \big( F_{\CC(u)} - F_d , F'_{\CC(u)}[\CC(\du)] \big)_{L^2(Q_T)} \\
        &\quad 
        + a_3 \big( M_{\CC(u)} - M_d , v'_{\CC(u)}[\CC(\du)] \big)_{L^2(Q_T)}
        + \lambda \big( u , \du \big)_{L^2(0,T;\R^n)}.
    \end{align*}
    We now proceed exactly as in the proof of Theorem~\ref{THM:NOC} to express the first three summands on the right-hand side by means of the adjoint state. We obtain
    \begin{align*}
        \tilde J'(u)[\du] 
        &= \big( N_{\CC(u)} - \dvr(w_{\CC(u)}) M_{\CC(u)} - \nabla M_{\CC(u)} w_{\CC(u)} , \CC(\du) \big)_{L^2(Q_T)}
        + \lambda \big( u , \du \big)_{L^2(Q_T)} \\
        & = \int_0^T \lambda u \cdot \du + \sum_{i=1}^n \left[ \int_\Omega \big(N_{\CC(u)}  - \nabla M_{\CC(u)} w_{\CC(u)} \big) \cdot h_i \dx \; \du_i \right] \dt \\
        & = \int_0^T \big(\lambda u + \mathcal D(u) \big) \cdot \du \dt,
    \end{align*}
    which proves (a). 
    
    To prove (b), we assume that $u^*\in\UAD$ is a locally optimal solution. Since $\UAD$ is convex, we know that for all $u\in\UAD$ and $\tau\in[0,1]$, it holds that $u^* + \tau(u-u^*) \in \UAD$. 
    Let now $u\in\UAD$ be arbitrary.
    As $u^*$ is a local minimizer of the cost functional $\tilde J$, we have
    \begin{align*}
        \tilde J\big(u^*+ \tau(u-u^*)) - \tilde J(u^*) \ge 0 
    \end{align*}
    for all $u\in\UAD$ and all sufficiently small $t>0$. This implies that
    \begin{align*}
        0 \le \frac{\mathrm d}{\mathrm d\tau} \tilde J\big(u^*+ \tau(u-u^*))\Big\vert_{\tau = 0}
        = \tilde J'\big(u^*+ \tau(u-u^*))[u-u^*]\Big\vert_{\tau = 0}
        = \tilde J'\big(u^*)[u-u^*].
    \end{align*}
    Due to the representation \eqref{DJW:2}, this proves \eqref{NOC:VI}. 
    It is a well-known result of optimal control theory, that $u^*$ can be expressed as the orthogonal projection onto the set of admissible control parameters. We refer to \cite[pp.~67--71]{Troltzsch} where such a projection formula was derived in a similar situation. Hence, all assertions of (b) are established and thus, the proof is complete.
\end{proof}

\section{Appendix}

\subsection{Well-posedness of the linearized system}
\begin{proof}[Proof of Proposition~\ref{WP:LIN}]
    The existence of a weak solution can be established rigorously via a standard Galerkin approximation where some of the arguments involved in the proof of Theorem~\ref{weaksolution} can also be applied. 
    Since the system is linear and the involved solution $(v_H,p_H,F_H,M_H)$ of the state equation is sufficiently regular, the convergence of suitable approximate solutions to a weak solution of the system \eqref{LIN} can be shown very easily. For the same reason, the uniqueness of a weak solution can be established without any problems.
	
	Therefore, we will just formally establish the a priori estimates that would be the most significant part of a Galerkin approach. 
	In the following, $\eps>0$ is any real number that will be adjusted later. Moreover, $C$ stands for a generic positive constant that depends on $\eps$, $T$, $\Omega$, $\|H\|_{L^2(0,T;W^{1,2}(\Omega)}$, and the initial data of the state $(v_H,p_H,F_H,M_H)$,
	and may change its value from line to line.
	Testing \eqref{LIN:1} by $\dv$, \eqref{LIN:3} by $\dF$, and \eqref{LIN:4} by $\dM$ and $-\Delta \dM$, respectively, we derive the following identities: 
\begin{align}
	  \label{testdv}
	  &\frac{1}{2}\frac{\mathrm d}{\mathrm dt}
	  \int_{\Omega}|\dv|^{2}+\nu\int_{\Omega}|\nabla\dv|^{2}
	  \notag\\
	  &= -\int_{\Omega} (\dv\cdot\nabla)v_{H}\cdot\dv
	  - \int_{\Omega}(v_{H}\cdot\nabla)\dv\cdot\dv
	  - \int_{\Omega}\dvr(\nabla\dM\odot\nabla M_{H})\cdot\dv
	  +\int_{\Omega}\dvr(\dF F^{T}_{H})\cdot\dv
	  \notag\\
	  &\qquad
	  -\int_{\Omega}\dvr(\nabla M_{H}\odot\nabla \dM)\cdot\dv
	  +\int_{\Omega}\dvr(F_{H}\dF^{T})\cdot\dv
	  +\int_{\Omega}(\nabla H)^{T}\dM\cdot\dv
	  +\int_{\Omega}S_{1} \cdot\dv
      \;=:\sum_{i=1}^{8}I^{\dv}_{i},
	  \\[1ex]
     \label{testdF}
	 & \frac{1}{2}\frac{\mathrm d}{\mathrm dt}\int_{\Omega}|\dF|^{2}+\int_{\Omega}|\nabla\dF|^{2}
	 \notag\\
	 &=\int_{\Omega} -(\dv\cdot\nabla)F_{H}\cdot\dF
	 -\int_{\Omega}(v_{H}\cdot\nabla)\dF\cdot\dF
	 +\int_{\Omega}\nabla\dv F_{H}\cdot\dF
	 +\int_{\Omega}\nabla v_{H}\dF\cdot\dF
	 +\int_{\Omega}S_{2} \cdot\dF
	 =:\sum_{i=1}^{5}I^{\dF}_{i},
     \\[1ex]
	 \label{testdM}
	 &\frac{1}{2}\frac{\mathrm d}{\mathrm dt}\int_{\Omega} |\dM|^{2}+\int_{\Omega}|\nabla\dM|^{2}
	 \notag\\
	 &= -\int_{\Omega}(\dv\cdot\nabla)M_{H}\cdot\dM
	 -\int_{\Omega}(v_{H}\cdot\nabla)\dM\cdot\dM
	 -\int_{\Omega}\frac{1}{\alpha^{2}}(|M_{H}|^{2}-1)\dM\cdot\dM
	 \notag\\
	 &\quad-\int_{\Omega}\frac{2}{\alpha^{2}}(\dM\cdot M_{H})M_{H}\cdot\dM
	 +\int_{\Omega} S_{3}\cdot\dM\cdot\dM
     =: \sum_{i=1}^{5}I^{\dM}_{i},
	 \\[1ex]
    \label{testDdM}
	&\frac{1}{2}\frac{\mathrm d}{\mathrm dt}\int_{\Omega} |\nabla\dM|^{2}+\int_{\Omega}|\Delta\dM|^{2}
	\notag\\
	&= \int_{\Omega} (\dv\cdot\nabla)M_{H}\cdot\Delta\dM
	+\int_{\Omega}(v_{H}\cdot\nabla)\dM\cdot\Delta\dM
	+\int_{\Omega}\frac{1}{\alpha^{2}}(|M_{H}|^{2}-1)\dM\cdot\Delta\dM
	\notag\\
    &\quad +\int_{\Omega}\frac{2}{\alpha^{2}}(\dM\cdot M_{H})M_{H}\cdot\Delta\dM
    -\int_{\Omega}S_{3}\cdot\Delta\dM
    =:\sum_{i=6}^{10}I^{\dM}_{i}.
\end{align}
	In the following, we estimate the terms $I_{i}^{\dv},$ $I_{i}^{\dF}$ and $I_{i}^{\dM}$ for all indices $i$. We first estimate the terms appearing in the right-hand side of \eqref{testdv}.
	Using \eqref{L40bnd} to estimate $\|\dv\|_{L^{4}(\Omega)}$, we find that
	\begin{align}\label{Idv1}
	|I^{\dv}_{1}|&\leq C\|\nabla v_{H}\|_{L^{2}(\Omega)}\|\dv\|^{2}_{L^{4}(\Omega)}
	\leq C\|\nabla v_{H}\|_{L^{2}(\Omega)}\|\dv\|_{L^{2}(\Omega)}\|\nabla\dv\|_{L^{2}(\Omega)}
	\notag\\
	&\leq \eps\|\nabla\dv\|^{2}_{L^{2}(\Omega)}+C\|\nabla v_{H}\|^{2}_{L^{2}(\Omega)}\|\dv\|^{2}_{L^{2}(\Omega)},
	\end{align}
    and
    \begin{align}\label{Idv2}
    |I^{\dv}_{2}|&\leq  C\|v_{H}\|_{L^{4}(\Omega)}\|\dv\|_{L^{4}(\Omega)}\|\nabla\dv\|_{L^{2}(\Omega)}
    \leq \eps\|\nabla\dv\|^{2}_{L^{2}(\Omega)}+C\|v_{H}\|^{2}_{L^{4}(\Omega)}\|\dv\|^{2}_{L^{4}(\Omega)}
    \notag\\
    & \leq 2\eps\|\nabla\dv\|^{2}_{L^{2}(\Omega)}+C\|v_{H}\|^{4}_{L^{4}(\Omega)}\|\dv\|^{2}_{L^{2}(\Omega)}.
    \end{align}
    For the third term, we obtain
    \begin{align}\label{Idv3}
    |I^{\dv}_{3}|
    &\leq C\|\nabla^{2}\dM\|_{L^{2}(\Omega)}\|\nabla M_{H}\|_{L^{4}(\Omega)}\|\dv\|_{L^{4}(\Omega)}+C\|\nabla\dM\|_{L^{4}(\Omega)}\|\nabla^{2}M_{H}\|_{L^{2}(\Omega)}\|\dv\|_{L^{4}(\Omega)}
    \notag\\
    &\leq \eps\|\nabla^{2}\dM\|^{2}_{L^{2}(\Omega)}+C\|\nabla M_{H}\|^{2}_{L^{4}(\Omega)}\|\dv\|_{L^{2}(\Omega)}\|\nabla\dv\|_{L^{2}(\Omega)}
    \notag\\
    &\qquad+C\|\nabla\dM\|_{L^{2}(\Omega)}\left(\|\nabla\dM\|^{2}_{L^{2}(\Omega)}+\|\Delta\dM\|^{2}_{L^{2}(\Omega)}\right)^{\frac{1}{2}}+C\|\nabla^{2}M_{H}\|^{2}_{L^{2}(\Omega)}\|\dv\|_{L^{2}(\Omega)}\|\nabla\dv\|_{L^{2}(\Omega)}
    \notag\\
    & \leq 2\eps\|\Delta\dM\|^{2}_{L^{2}(\Omega)}+C\|\dM\|^{2}_{L^{2}(\Omega)}+2\eps\|\nabla\dv\|^{2}_{L^{2}(\Omega)}+C\|\nabla\dM\|^{2}_{L^{2}(\Omega)}
    \notag\\
    &\qquad+C\left(\|\nabla M_{H}\|^{4}_{L^{4}(\Omega)}+\|\nabla^{2}M_{H}\|^{4}_{L^{2}(\Omega)}\right)\|\dv\|^{2}_{L^{2}(\Omega)}.
    \end{align}
    Here, we used \eqref{smoreinterpole2} and \eqref{L40bnd} to estimate $\|\nabla\dM\|_{L^{4}(\Omega)}$ and $\|\dv\|_{L^{4}(\Omega)}$, and we applied elliptic regularity theory.
    We further get
    \begin{align}\label{Idv4}
    |I^{\dv}_{4}|
    &\leq C\|\dF\|_{L^{4}(\Omega)}\|\nabla F_{H}\|_{L^{2}(\Omega)}\|\dv\|_{L^{4}(\Omega)}+C\|\nabla\dF\|_{L^{2}(\Omega)}\|F_{H}\|_{L^{4}(\Omega)}\|\dv\|_{L^{4}(\Omega)}
    \notag\\
    & \leq C\|\nabla F_{H}\|^{2}_{L^{2}(\Omega)}\|\dF\|_{L^{2}(\Omega)}\|\nabla\dF\|_{L^{2}(\Omega)}+C\|\dv\|_{L^{2}(\Omega)}\|\nabla\dv\|_{L^{2}(\Omega)}+\eps\|\nabla\dF\|^{2}_{L^{2}(\Omega)}\\
    &\qquad+C\|F_{H}\|^{2}_{L^{4}(\Omega)}\|\dv\|_{L^{2}(\Omega)}\|\nabla\dv\|_{L^{2}(\Omega)}
    \notag\\
    & \leq 2\eps\|\nabla\dF\|^{2}_{L^{2}(\Omega)}+2\eps\|\nabla\dv\|^{2}_{L^{2}(\Omega)}+C\|\nabla F_{H}\|^{4}_{L^{2}(\Omega)}\|\dF\|^{2}_{L^{2}(\Omega)}+C(1+\|F_{H}\|^{4}_{L^{4}(\Omega)})\|\dv\|^{2}_{L^{2}(\Omega)},
    \end{align}
    where both $\|\dF\|_{L^{4}(\Omega)}$ and $\|\dv\|_{L^{4}(\Omega)}$ are estimated by means of \eqref{L40bnd}.
    Next one observes that $|I^{\dv}_{5}|$ and $|I^{\dv}_{6}|$ admit the same estimates as that of \eqref{Idv3} and \eqref{Idv4} respectively.
    Employing \eqref{smoreinterpole1} to estimate $\|\dM\|_{L^{4}(\Omega)}$, we obtain
    \begin{align}\label{Idv7}
    |I^{\dv}_{7}|
    & \leq C\|\nabla H\|_{L^{2}(\Omega)}\|\dM\|_{L^{4}(\Omega)}\|\dv\|_{L^{4}(\Omega)}
    \notag\\
    & \leq C\|\nabla H\|_{L^{2}(\Omega)}\big(\|\dM\|^{2}_{L^{2}(\Omega)}+\|\nabla\dM\|^{2}_{L^{2}(\Omega)}\big)+C \|\nabla H\|_{L^{2}(\Omega)} \|\dv\|_{L^{2}(\Omega)} \|\nabla\dv\|_{L^{2}(\Omega)}
    \notag\\
    & \leq C\|\nabla H\|_{L^{2}(\Omega)}\|\dM\|^{2}_{L^{2}(\Omega)}
    +C\|\nabla H\|_{L^{2}(\Omega)}\|\nabla\dM\|^{2}_{L^{2}(\Omega)}
    \notag\\
    & \qquad
    +C\|\nabla H\|^2_{L^{2}(\Omega)}\|\dv\|^{2}_{L^{2}(\Omega)}+\eps\|\nabla\dv\|^{2}_{L^{2}(\Omega)}.
    \end{align}
    Eventually, for the eighth term, we simply have
    \begin{align}\label{Idv8}
    |I^{\dv}_{8}|\leq C\|S_{1}\|^{2}_{L^{2}(\Omega)}+C\|v\|^{2}_{L^{2}(\Omega)}.
    \end{align}
    
    We next estimate the terms appearing in the right-hand side of \eqref{testdF}. Using \eqref{L40bnd} to estimate $\|\widehat{F}\|_{L^{4}(\Omega)}$, we deduce the estimate
	 \begin{align}\label{IhF1}
	 |I^{\widehat{F}}_{1}|&\leq C\|\widehat{v}\|_{L^{2}(\Omega)}\|\nabla F_{H}\|_{L^{4}(\Omega)}\|\widehat{F}\|_{L^{4}(\Omega)}
	 \notag\\
	 &\leq C\|\nabla F_{H}\|^{2}_{L^{4}(\Omega)}\|\widehat{v}\|^{2}_{L^{2}(\Omega)}+C\|\widehat{F}\|_{L^{2}(\Omega)}\|\nabla\widehat{F}\|_{L^{2}(\Omega)}
	 \notag\\
	 & \leq C\|\nabla F_{H}\|^{2}_{L^{4}(\Omega)}\|\widehat{v}\|^{2}_{L^{2}(\Omega)}+C\|\widehat{F}\|^{2}_{L^{2}(\Omega)}+\eps\|\nabla\widehat{F}\|^{2}_{L^{2}(\Omega)}
	 \end{align} 
	 For the second term, we get
	 \begin{align}
	 \label{IhF2}
	 |I^{\widehat{F}}_{2}|&\leq C\|\nabla\widehat{F}\|_{L^{2}(\Omega)}\|v_{H}\|_{L^{4}(\Omega)}\|\widehat{F}\|_{L^{4}(\Omega)}
	 \notag\\
	 &\leq C\|v_{H}\|^{2}_{L^{4}(\Omega)}\|\widehat{F}\|^{2}_{L^{4}(\Omega)}+\eps\|\nabla\widehat{F}\|^{2}_{L^{2}(\Omega)}
	 \notag\\
	 &\leq C\|v_{H}\|^{2}_{L^{4}(\Omega)}\|\widehat{F}\|_{L^{2}(\Omega)}\|\nabla\widehat{F}\|_{L^{2}(\Omega)}+\eps\|\nabla\widehat{F}\|^{2}_{L^{2}(\Omega)}
	 \notag\\
	 &\leq 2\eps\|\nabla\widehat{F}\|^{2}_{L^{2}(\Omega)}+C\|v_{H}\|^{4}_{L^{4}(\Omega)}\|\widehat{F}\|^{2}_{L^{2}(\Omega)}.
	 \end{align}  
	 We further obtain 
	 \begin{align}\label{IhF3}
	 |I^{\widehat{F}}_{3}|& \leq C \|\nabla\widehat{v}\|_{L^{2}(\Omega)}\|F_{H}\|_{L^{4}(\Omega)}\|\widehat{F}\|_{L^{4}(\Omega)}
	 \notag\\
	 &\leq \eps\|\nabla\widehat{v}\|^{2}_{L^{2}(\Omega)}+C\|F_{H}\|^{2}_{L^{4}(\Omega)}\|\widehat{F}\|^{2}_{L^{4}(\Omega)}
	 \notag\\
	 &\leq \eps\|\nabla\widehat{v}\|^{2}_{L^{2}(\Omega)}+C\|F_{H}\|^{2}_{L^{4}(\Omega)}\|\widehat{F}\|_{L^{2}(\Omega)}\|\nabla\widehat{F}\|_{L^{2}(\Omega)}
	 \notag\\
	 &\leq \eps\|\nabla\widehat{v}\|^{2}_{L^{2}(\Omega)}+\eps\|\nabla\widehat{F}\|^{2}_{L^{2}(\Omega)}+C\|F_{H}\|^{4}_{L^{4}(\Omega)}\|\widehat{F}\|^{2}_{L^{2}(\Omega)}.
	 \end{align}	
	 The fourth term can be bounded as follows:
	 \begin{align}\label{IhF4}
	 |I^{\widehat{F}}_{4}|&\leq\|\nabla v_{H}\|_{L^{4}(\Omega)}\|\widehat{F}\|_{L^{4}(\Omega)}\|\widehat{F}\|_{L^{2}(\Omega)}
	 \notag\\
	 & \leq C\|\nabla v_{H}\|^{2}_{L^{4}(\Omega)}\|\widehat{F}\|^{2}_{L^{2}(\Omega)}+C\|\widehat{F}\|^{2}_{L^{4}(\Omega)}
	 \notag\\
	 &\leq C\|\nabla v_{H}\|^{2}_{L^{4}(\Omega)}\|\widehat{F}\|^{2}_{L^{2}(\Omega)}+C\|\widehat{F}\|_{L^{2}(\Omega)}\|\nabla\widehat{F}\|_{L^{2}(\Omega)}
	 \notag\\
	 &\leq \eps\|\nabla\widehat{F}\|^{2}_{L^{2}(\Omega)}+C\|\widehat{F}\|^{2}_{L^{2}(\Omega)}+C\|\nabla v_{H}\|^{2}_{L^{4}(\Omega)}\|\widehat{F}\|^{2}_{L^{2}(\Omega)}.
	 \end{align}
	 Eventually, for the fifth term, we simply get 
	 \begin{align}\label{IhF5}
	 |I^{\widehat{F}}_{5}| \leq C\|S_{2}\|^{2}_{L^{2}(\Omega)}+C\|\widehat{F}\|^{2}_{L^{2}(\Omega)}.
	 \end{align}
	 Now, we estimate the summands appearing in the right-hand side of \eqref{testdM}. Using \eqref{smoreinterpole1} to estimate $\|\widehat{M}\|_{L^{4}(\Omega)}$, we obtain
	 \begin{align}\label{IhM1}
	 |I^{\widehat{M}}_{1}|&\leq C\|\nabla M_{H}\|_{L^{4}(\Omega)}\|\widehat{v}\|_{L^{2}(\Omega)}\|\widehat{M}\|_{L^{4}(\Omega)}\\
	 & \leq C\|\nabla M_{H}\|^{2}_{L^{4}(\Omega)}\|\widehat{v}\|^{2}_{L^{2}(\Omega)}+C\|\widehat{M}\|^{2}_{L^{2}(\Omega)}+\eps\|\nabla\widehat{M}\|^{2}_{L^{2}(\Omega)}
	 \end{align}
     and
	 \begin{align}\label{IhM2}
	  |I^{\widehat{M}}_{2}|& \leq \|v_{H}\|_{L^{4}(\Omega)}\|\nabla\widehat{M}\|_{L^{2}(\Omega)}\|\widehat{M}\|_{L^{4}(\Omega)}
	  \notag\\
	  &\leq \eps\|\nabla\widehat{M}\|^{2}_{L^{2}(\Omega)}+C\|v_{H}\|^{2}_{L^{4}(\Omega)}\|\widehat{M}\|^{2}_{L^{2}(\Omega)}+C\|v_{H}\|^{2}_{L^{4}(\Omega)}\|\widehat{M}\|_{L^{2}(\Omega)}\|\nabla\widehat{M}\|_{L^{2}(\Omega)}
	  \notag\\
	  & \leq 2\eps\|\nabla\widehat{M}\|^{2}_{L^{2}(\Omega)}+C\big(\|v_{H}\|^{4}_{L^{4}(\Omega)}+1\big)\|\widehat{M}\|^{2}_{L^{2}(\Omega)}.
	 \end{align}
	 The terms $I^{\widehat{M}}_{i}$, $i=3,4,5$ are estimated as follows:
	 \begin{align}\label{IhM3}
	 |I^{\widehat{M}}_{3}|
	 &\leq C\big(\|M_{H}\|^{2}_{L^{\infty}(\Omega)}+1\big)\|\widehat{M}\|^{2}_{L^{2}(\Omega)}, \\
    \label{IhM4}
	  |I^{\widehat{M}}_{4}|
	  &\leq C\|M_{H}\|^{2}_{L^{\infty}(\Omega)}\|\widehat{M}\|^{2}_{L^{2}(\Omega)}, \\
	 \label{IhM5}
	  |I^{\widehat{M}}_{5}|
	  &\leq C\|S_{3}\|^{2}_{L^{2}(\Omega)}+C\|\widehat{M}\|^{2}_{L^{2}(\Omega)}.
	 \end{align}
	 Finally, we estimate the terms appearing in the right-hand side of \eqref{testDdM}. Using \eqref{L40bnd} to estimate $\|\widehat{v}\|_{L^{4}(\Omega)}$, we deduce
	 \begin{align}\label{IhM6}
	 |I^{\widehat{M}}_{6}|&\leq C\|\nabla M_{H}\|_{L^{4}(\Omega)}\|\widehat{v}\|_{L^{4}(\Omega)}\|\Delta\widehat{M}\|_{L^{2}(\Omega)}
	 \notag\\
	 &\leq \eps\|\Delta\widehat{M}\|^{2}_{L^{2}(\Omega)}+C\|\nabla M_{H}\|^{2}_{L^{4}(\Omega)}\|\widehat{v}\|_{L^{2}(\Omega)}\|\nabla\widehat{v}\|_{L^{2}(\Omega)}
	 \notag\\
	 &\leq \eps\|\Delta\widehat{M}\|^{2}_{L^{2}(\Omega)}+\eps\|\nabla\widehat{v}\|^{2}_{L^{2}(\Omega)}+C\|\nabla M_{H}\|^{4}_{L^{4}(\Omega)}\|\widehat{v}\|^{2}_{L^{2}(\Omega)}.
	 \end{align} 
	 Moreover, employing \eqref{smoreinterpole2} and Young's inequality to estimate $\|\nabla\widehat{M}\|_{L^{4}(\Omega)}$, we get
	 \begin{align}\label{IhM7}
	 |I^{\widehat{M}}_{7}|& \leq \eps \|\Delta\widehat{M}\|^{2}_{L^{2}(\Omega)}+\|v_{H}\|^{2}_{L^{4}(\Omega)}\|\nabla\widehat{M}\|^{2}_{L^{4}(\Omega)}\\
	 &\leq \eps\|\Delta\widehat{M}\|^{2}_{L^{2}(\Omega)}+\eps\|v_{H}\|^{2}_{L^{4}(\Omega)}\|\Delta\widehat{M}\|^{2}_{L^{2}(\Omega)}+C\|v_{H}\|^{2}_{L^{4}(\Omega)}\|\nabla\widehat{M}\|^{2}_{L^{2}(\Omega)}.
	 \end{align}
	 The remaining terms can easily be estimated as follows:
	 \begin{align}
	 \label{IhM8}
	  |I^{\widehat{M}}_{8}|
	  &\leq\eps\|\Delta\widehat{M}\|^{2}_{L^{2}(\Omega)}+C\big(\|M_{H}\|^{2}_{L^{\infty}(\Omega)}+1\big)^{2}\|\widehat{M}\|^{2}_{L^{2}(\Omega)}, \\
	 \label{IhM9}
	  |I^{\widehat{M}}_{9}|
	  &\leq \eps\|\Delta\widehat{M}\|^{2}_{L^{2}(\Omega)}+C\|M_{H}\|^{4}_{L^{\infty}(\Omega)}\|\widehat{M}\|^{2}_{L^{2}(\Omega)},\\
	 \label{IhM10}
	 |I^{\widehat{M}}_{10}|
	 &\leq \eps\|\Delta\widehat{M}\|^{2}_{L^{2}(\Omega)}+C\|S_{3}\|^{2}_{L^{2}(\Omega)}.
	 \end{align}
	 Now, choosing $\eps>0$ sufficiently small, adding \eqref{testdv}--\eqref{testDdM}, and using the estimates \eqref{Idv1}--\eqref{IhM10} to estimate the right-hand side of the resulting equation, we conclude that
	 \begin{align}\label{esthvFM}
	 \frac{1}{2}\frac{\mathrm d}{\mathrm dt}
	 \widehat{\mathcal{Y}}(t)+\widehat{\mathcal{B}}(t)
	 \leq C \widehat{\mathcal{Q}}(t)\, \widehat{\mathcal{Y}}(t)
        +C\big(\|S_{1}\|^{2}_{L^{2}(\Omega)}
	    +\|S_{2}\|_{L^{2}(\Omega)}^2+\|S_{3}\|_{L^{2}(\Omega)}^2\big),
	 \end{align}
	 for almost all $t\in[0,T]$, where 
	 \begin{align}\label{whatYB}
	 \widehat{\mathcal{Y}}&:=\int_{\Omega}\bigl(|\widehat{v}|^{2}+|\widehat{F}|^{2}+|\widehat{M}|^{2}+|\nabla\widehat{M}|^{2}\bigr),
	 \\
	 \widehat{\mathcal{B}}&:=\frac{1}{2} \int_{\Omega}\bigl(\nu|\nabla \widehat{v}|^{2}|+|\nabla \widehat{F}|^{2}+|\nabla \widehat{M}|^{2}+|\Delta \widehat{M}|^{2}\bigr),
	 \\
	 \widehat{\mathcal{Q}}&:=
	 \|\nabla v_{H}\|^{2}_{L^{2}(\Omega)}+\|v_{H}\|^{4}_{L^{4}(\Omega)}+\|\nabla v_{H}\|^{2}_{L^{4}(\Omega)}+\|\nabla M_{H}\|^{4}_{L^{4}(\Omega)}+\|\nabla^{2}M_{H}\|^{4}_{L^{2}(\Omega)}
	 \notag\\
	 &\quad+\|M_{H}\|^{4}_{L^{\infty}(\Omega)}
	 +\|\nabla F_{H}\|^{4}_{L^{2}(\Omega)}+\|F_{H}\|^{4}_{L^{4}(\Omega)}+\|\nabla F_{H}\|^{2}_{L^{4}(\Omega)}+\|\nabla H\|^{2}_{L^{2}(\Omega)}+1.
	 \end{align}
	 We point out that we also made use of the fact that $\|v_{H}\|_{L^{\infty}(0,T;L^{2}(\Omega))} \le C$ (cf.~\eqref{IhM7}) to derive the estimate \eqref{esthvFM}.
	 
	 Since $H\in\HH$ and $(v_H,p_H,F_H,M_H) \in \VV$, it is straightforward to check that $$\bignorm{\widehat{\mathcal{Q}}}_{L^1([0,T])}\le C.$$ 
	 Applying Gronwall's lemma, we finally obtain
	 \begin{align}
	    \bignorm{\widehat{\mathcal{Y}}}_{L^\infty([0,T])}
	    + \int_0^T \widehat{\mathcal{B}}(t) \dt
	    \le  C\sum_{i=1}^3 \norm{S_i}_{L^2(0,T;L^2(\Omega))}.
	 \end{align}
	 This a priori estimate can eventually be used to establish the spatial regularity properties collected in \eqref{REG:LIN} and, in particular, it implies the estimate \eqref{EST:LIN}.
	 Eventually, the time regularity properties stated in \eqref{REG:LIN} then follow by standard comparison arguments.
\end{proof}	

\subsection{Well-posedness of the adjoint system}
Instead of proving well-posedness for the adjoint system, we consider the equivalent initial value problem instead:
\begin{subequations}\label{ADJ*}
	\begin{alignat}{2}
		&\partial_{t} w  - \nu \Delta w - \nabla q \notag \label{line1}
		=  (v_H \cdot \nabla) w - (\nabla v_H)^T \, w -(\nabla F_{H})^T G \\
		&\qquad\qquad\qquad\qquad\quad  - \dvr(G F_H^T) - (\nabla M_H)^T N + a_1(v_H-v_d) \;\; 
		&&\mbox{ in } Q_{T},\\[1ex]
		&\dvr w=0
		&&\mbox{ in } Q_{T},\\[1ex]
		&\partial_{t} G -\Delta G =  (v_{H}\cdot\nabla) G + (\nabla v_H)^T G -2\,\D w\, F_H + a_2(F_H-F_d)
		&&\mbox{ in } Q_{T},\label{line2}
		\\[1ex]
		&\partial_{t} N   -\Delta N 
		=  - 2\alpha^{-2} (M_H\cdot N) M_H  
		- \alpha^{-2} \big(|M_H|^2 - 1\big)N+(v_H\cdot \nabla) N&&\notag\\
		&\qquad\qquad\qquad\quad -2 \dvr( \nabla M_H\, \D w) 
		 + \nabla H\, w  + a_3(M_H-M_d)
		&&\mbox{ in } Q_{T},\label{line3}\\[1ex]
		&w =0,\ G=0,\ \partial_{n} N =0
		&&\mbox{ on } \Sigma_{T},\\[1ex]
		&(w,G,N)(\cdot,0)=(0,0,0)
		&& \mbox{ in }\Omega.
	\end{alignat}
\end{subequations} 
We point out that \eqref{ADJ*} is indeed equivalent to the adjoint system \eqref{ADJ} through the transformation $t\mapsto T-t$. This means that by proving the weak well-posedness of the system \eqref{ADJ*}, the weak well-posedness of \eqref{ADJ} is also established. 

\begin{prop}\label{WP:ADJ*}
	Let $H\in\HH$ be arbitrary with corresponding state $\FF(H)=(v_H,p_H,F_H,M_H)$. Then the system \eqref{ADJ*} has a unique weak solution $(w,q,G,N)$ having the regularity
	\begin{align}\label{REG:ADJ*}
	\left\{ \begin{aligned}
	&  w\in L^{2}(0,T;V(\Omega))
	\cap L^{\infty}(0,T;W^{1,2}_{0,\dvr}(\Omega))
	\cap W^{1,2}(0,T;L^2_{\dvr}(\Omega));\\
	&  q \in L^2(0,T;W^{1,2}(\Omega)); \\
	&  G \in L^{2}(0,T;W^{2,2}(\Omega))
	\cap L^{\infty}(0,T;W^{1,2}_{0}(\Omega))
	\cap W^{1,2}(0,T;L^2(\Omega));\\
	&  N \in  
	L^{\infty}(0,T;L^{2}(\Omega))\cap L^{2}(0,T;W^{1,2}(\Omega)) 
	\end{aligned}
	\right.
	\end{align}
	Moreover, the function $N$ has the additional regularity 
	\begin{align}
	\label{furtherregN}
	    N\in L^{\frac{3}{2}}(0,T;W^{2,\frac{3}{2}}(\Omega))\cap W^{1,\frac{3}{2}}(0,T;L^{\frac{3}{2}}(\Omega)).
	\end{align}
\end{prop}

\begin{proof}
    For the same reasons as in the proof of Proposition~\ref{WP:LIN}, we only present the formal a priori estimates. We point out that a rigorous proof can be carried out by means of a Galerkin approximation. Passing to the limit in a Galerkin scheme and proving uniqueness of the weak solution is straightforward due to the linearity of system \eqref{ADJ*}.
    
    In the following, let $\eps>0$ be any real number that will be fixed later. Moreover, the letter $C$ denotes a generic positive constant that depends on $\eps$, $T$, $\Omega$, $\|H\|_{L^2(0,T;W^{1,2}(\Omega)}$, and the initial data of the state $(v_H,p_H,F_H,M_H)$,
	and may change its value from line to line. 
	Since the estimates in this proof are derived using very similar ideas as in in previous proofs, we will mostly present the formal computations without further comments. 

	We first test \eqref{line1} by ${\mathcal{S}}w,$ where ${\mathcal{S}}w$ is the Stokes operator that is defined as
	$${\mathcal{S}}w=-\nu\Delta w-\nabla q.$$
	This yields the identity
	\begin{align}\label{momentumtesting}
	&\frac{\mathrm d}{\mathrm dt}\int_{\Omega}|\nabla w|^{2}+\int_{\Omega}|{\mathcal{S}}w|^{2} 
	\notag\\
	&= \int_{\Omega} (v_H \cdot \nabla) w \cdot \mathcal S w 
	- \int_{\Omega} (\nabla v_H)^T \, w \cdot \mathcal S w
	- \int_{\Omega} (\nabla F_{H})^T G \cdot \mathcal S w
	- \int_{\Omega} \dvr(G F_H^T) \cdot \mathcal S w
	\notag\\
	&\quad
	- \int_{\Omega} (\nabla M_H)^T N \cdot \mathcal S w 
	+ \int_{\Omega} a_1(v_H-v_d) \cdot \mathcal S w
	\; =:\sum\limits_{i=1}^{6}I_{i}^{w}.
	\end{align}
	The terms $I^{w}_{i}$, $i=1,...,6$ are estimated as follows:
	\begin{align}\label{Iw1}
	|I^{w}_{1}|&\leq \eps\|{\mathcal{S}}w\|^{2}_{L^{2}(\Omega)}
	+C\|\nabla w \|^{2}_{L^{4}(\Omega)}\|v_{H}\|^{2}_{L^{4}(\Omega)}
	\notag\\
	&\leq \eps\|{\mathcal{S}}w\|^{2}_{L^{2}(\Omega)}+C\big(\|\nabla w\|^{2}_{L^{2}(\Omega)}+\|\nabla w\|_{L^{2}(\Omega)}\|\Delta w\|_{L^{2}(\Omega)}\big)\|v_{H}\|^{2}_{L^{4}(\Omega)}
	\notag\\
	& \leq\eps \|{\mathcal{S}}\, w\|^{2}_{L^{2}(\Omega)}+C\|\nabla w\|^{2}_{L^{2}(\Omega)}\|v_{H}\|^{2}_{L^{4}(\Omega)}+C\|\nabla w\|_{L^{2}(\Omega)}\|{\mathcal{S}}\, w\|_{L^{2}(\Omega)}\|v_{H}\|^{2}_{L^{4}(\Omega)}
	\notag\\
	&\leq 2\eps\|{\mathcal{S}}w\|^{2}_{L^{2}(\Omega)}+C\|\nabla w\|^{2}_{L^{2}(\Omega)}\|v_{H}\|^{2}_{L^{4}(\Omega)}+C\|\nabla w\|^{2}_{L^{2}(\Omega)}\|v_{H}\|^{4}_{L^{4}(\Omega)},
	\\[1ex]
	|I^{w}_{2}|&  
	\leq \eps\|{\mathcal{S}}w\|^{2}_{L^{2}(\Omega)}
	+C\|\nabla v_H\|^{2}_{L^{4}(\Omega)}\|w\|^{2}_{L^{4}(\Omega)}
	\notag\\
	& \leq \eps\|{\mathcal{S}}w\|^{2}_{L^{2}(\Omega)}
	+C\|\nabla v_H\|^{2}_{L^{4}(\Omega)} \|\nabla w\|^{2}_{L^{2}(\Omega)},
	\\[1ex]
	|I_{3}^{w}|&\leq \eps\|{\mathcal{S}}w\|^{2}_{L^{2}(\Omega)}+C\|\nabla F_{H}\|^{2}_{L^{4}(\Omega)}\|G\|^{2}_{L^{4}(\Omega)}
	\notag\\
	& \leq \eps\|{\mathcal{S}}w\|^{2}_{L^{2}(\Omega)}+C\|\nabla F_{H}\|^{2}_{L^{4}(\Omega)}\|G\|_{L^{2}(\Omega)}\|\nabla G\|_{L^{2}(\Omega)}
	\notag\\
	& \leq \eps\|{\mathcal{S}}w\|^{2}_{L^{2}(\Omega)}+\eps\|\nabla G\|^{2}_{L^{2}(\Omega)}+C\|\nabla F_{H}\|^{4}_{L^{4}(\Omega)}\|G\|^{2}_{L^{2}(\Omega)},
	\\[1ex]
	|I^{w}_{4}|&\leq 2\eps\|{\mathcal{S}}w\|^{2}_{L^{2}(\Omega)}+C\|\nabla F_{H}\|^{2}_{L^{4}(\Omega)}\|G\|_{L^{2}(\Omega)}\|\nabla G\|_{L^{2}(\Omega)}
	\notag\\
	&\qquad+ C\|\nabla G\|^{2}_{L^{4}(\Omega)}\|F_{H}\|^{2}_{L^{4}(\Omega)}
	\notag\\
	&\leq 2\eps\|{\mathcal{S}}w\|^{2}_{L^{2}(\Omega)}+\eps\|\nabla G\|^{2}_{L^{2}(\Omega)}+C\|\nabla F_{H}\|^{4}_{L^{4}(\Omega)}\|G\|^{2}_{L^{2}(\Omega)}
	\notag\\
	&\qquad +C\big(\|\nabla G\|^{2}_{L^{2}(\Omega)}+\|\nabla G\|_{L^{2}(\Omega)}\|\Delta G\|_{L^{2}(\Omega)}\big)\|F_{H}\|^{2}_{L^{4}(\Omega)}
	\notag\\
	&\leq 2\eps\|{\mathcal{S}}w\|^{2}_{L^{2}(\Omega)}+\eps\|\nabla G\|^{2}_{L^{2}(\Omega)}+C\|\nabla F_{H}\|^{4}_{L^{4}(\Omega)}\|G\|^{2}_{L^{2}(\Omega)}
	\notag\\
	&\qquad +C\|\nabla G\|^{2}_{L^{2}(\Omega)}\|F_{H}\|^{2}_{L^{4}(\Omega)}+\eps\|\Delta G\|^{2}_{L^{2}(\Omega)}+C\|\nabla G\|^{2}_{L^{2}(\Omega)}\|F_{H}\|^{4}_{L^{4}(\Omega)},
	\\[1ex]
	\label{Iw*}
	|I^{w}_{5}|
    &\leq\eps\|{\mathcal{S}}w\|^{2}_{L^{2}(\Omega)}+C\|\nabla M_{H}\|^{2}_{L^{4}(\Omega)}\|N\|^{2}_{L^{4}(\Omega)}\\
    &\leq \eps\|{\mathcal{S}}w\|^{2}_{L^{2}(\Omega)}+C\|\nabla M_{H}\|^{2}_{L^{4}(\Omega)}\left(\|N\|^{2}_{L^{2}(\Omega)}+\eps\|\nabla N\|^{2}_{L^{2}(\Omega)}\right)\nonumber,
    \\[1ex]
    \label{Iw5}
    |I^{w}_{6}|
    &\leq\eps\|{\mathcal{S}}w\|^{2}_{L^{2}(\Omega)}+C\|v_{H}-v_{d}\|^{2}_{L^{2}(\Omega)}.
	\end{align}
	In the estimate \eqref{Iw*} we have used \eqref{smoreinterpole1} and Young's inequality.\\
	Next, testing \eqref{line2} by $-\Delta G$ leads to
	\begin{align}\label{Gtesting}
	&\frac{\mathrm d}{\mathrm dt}\int_{\Omega}|\nabla G|^{2}+\int_{\Omega}|\Delta G|^{2}
	\notag\\
	&= -\int_\Omega (v_{H}\cdot\nabla) G \cdot \Delta G
	- \int_\Omega(\nabla v_H)^T G \cdot \Delta G
	\notag\\
	&\quad
	+ \int_\Omega 2\D w\, F_H \cdot \Delta G
	- \int_\Omega a_2(F_H-F_d) \cdot \Delta G
	=:\sum_{i=1}^{4}I^{G}_{i}.
	\end{align}
	Now, the terms $I^{G}_{i}$, $i=1,...,4$ are estimated as follows:
	\begin{align}\label{IG1}
	|I^{G}_{1}|&\leq \eps\|\Delta G\|^{2}_{L^{2}(\Omega)}+C\|v_{H}\|^{2}_{L^{4}(\Omega)}\|\nabla G\|^{2}_{L^{4}(\Omega)}
	\notag\\
	& \leq \eps\|\Delta G\|^{2}_{L^{2}(\Omega)}+C\|v_{H}\|^{2}_{L^{4}(\Omega)}\big(\|\nabla G\|^{2}_{L^{2}(\Omega)}+\|\nabla G\|_{L^{2}(\Omega)}\|\Delta G\|_{L^{2}(\Omega)}\big)
	\notag\\
	& \leq 2\eps\|\Delta G\|^{2}_{L^{2}(\Omega)}+C\|v_{H}\|^{2}_{L^{4}(\Omega)}\|\nabla G\|^{2}_{L^{2}(\Omega)}+C\|v_{H}\|^{4}_{L^{4}(\Omega)}\|\nabla G\|^{2}_{L^{2}(\Omega)},
	\\[1ex]
	|I^{G}_{2}|&\leq \eps\|\Delta G\|^{2}_{L^{2}(\Omega)}+C\|\nabla v_{H}\|^{2}_{L^{4}(\Omega)}\|G\|^{2}_{L^{4}(\Omega)}
	\notag\\
	&\leq \eps\|\Delta G\|^{2}_{L^{2}(\Omega)}+C\|\nabla v_{H}\|^{2}_{L^{4}(\Omega)}\|G\|_{L^{2}(\Omega)}\|\nabla G\|_{L^{2}(\Omega)}
	\notag\\
	&\leq \eps\|\Delta G\|^{2}_{L^{2}(\Omega)}+\eps\|\nabla G\|^{2}_{L^{2}(\Omega)}+C\|\nabla v_{H}\|^{4}_{L^{4}(\Omega)}\|G\|^{2}_{L^{2}(\Omega)}, 
	\\[1ex]
	|I^{G}_{3}|&\leq\eps\|\Delta G\|^{2}_{L^{2}(\Omega)}+C\|\nabla w\|^{2}_{L^{4}(\Omega)}\|F_{H}\|^{2}_{L^{4}(\Omega)}
	\notag\\
	& \leq\eps\|\Delta G\|^{2}_{L^{2}(\Omega)}+C\big(\|\nabla w\|^{2}_{L^{2}(\Omega)}+\|\nabla w\|_{L^{2}(\Omega)}\|\Delta w\|_{L^{2}(\Omega)}\big)\|F_{H}\|^{2}_{L^{4}(\Omega)}
	\notag\\
	& \leq \eps\|\Delta G\|^{2}_{L^{2}(\Omega)}+C\|\nabla w\|^{2}_{L^{2}(\Omega)}\|F_{H}\|^{2}_{L^{4}(\Omega)}+\eps\|\Delta w\|^{2}_{L^{2}(\Omega)}+C\|\nabla w\|^{2}_{L^{2}(\Omega)}\|F_{H}\|^{4}_{L^{4}(\Omega)},
    \\[1ex]
    \label{IG4}
    |I^{G}_{4}|
    &\leq \eps\|\Delta G\|^{2}_{L^{2}(\Omega)}+C\|F_{H}-F_{d}\|^{2}_{L^{2}(\Omega)}.
	\end{align}
	
	Finally, testing \eqref{line3} by $N$, we obtain
	\begin{align}\label{Ntesting}
	&\frac{\mathrm d}{\mathrm dt}\int_{\Omega}| N|^{2}+\int_{\Omega}|\nabla N|^{2}
	\notag\\
	&=  - \int_\Omega 2\alpha^{-2} (M_H\cdot N) M_H \cdot N 
	    - \int_\Omega \alpha^{-2} \big(|M_H|^2 - 1\big)N \cdot N 
	    + \int_\Omega (v_H\cdot \nabla) N \cdot N 
	\notag\\
	&\quad
        - \int_\Omega 2 \dvr( \nabla M_H\, \D w) \cdot N 
		+ \int_\Omega \nabla H\, w \cdot N 
		+ \int_\Omega a_3(M_H-M_d) \cdot N 
	=:\sum_{i=1}^{6}I^{N}_{i}.
	\end{align}
	
	For the terms $I^{N}_{i}$, $i=1,...,6$ we obtain the following estimates:
	\begin{align}\label{IN1}
    |I^{N}_{1}|
    &\leq \|M_{H}\|^{2}_{L^{\infty}(\Omega)}\|N\|^{2}_{L^{2}(\Omega)},
    \\[1ex]
    |I^{N}_{2}|
    &\leq \|M_{H}\|^{2}_{L^{\infty}(\Omega)}\|N\|^{2}_{L^{2}(\Omega)}+\|N\|^{2}_{L^{2}(\Omega)},
    \\[1ex]
	 |I^{N}_{3}|& \leq\eps\|\nabla N\|^{2}_{L^{2}(\Omega)}+C\|v_{H}\|^{2}_{L^{4}(\Omega)}\|N\|^{2}_{L^{4}(\Omega)}
	\notag\\
	& \leq \eps\|\nabla N\|^{2}_{L^{2}(\Omega)}+C\|v_{H}\|^{2}_{L^{4}(\Omega)}\big(\|N\|^{2}_{L^{2}(\Omega)}+\|N\|_{L^{2}(\Omega)}\|\nabla N\|_{L^{2}(\Omega)}\big)
	\notag\\
	& \leq 2\eps\|\nabla N\|^{2}_{L^{2}(\Omega)}+C\|v_{H}\|^{2}_{L^{4}(\Omega)}\|N\|^{2}_{L^{2}(\Omega)}+C\|v_{H}\|^{4}_{L^{4}(\Omega)}\|N\|^{2}_{L^{2}(\Omega)},
    \\[1ex]
	 |I^{N}_{4}|& \leq \|\nabla^{2}M_{H}\|_{L^{4}(\Omega)}\|\nabla w\|_{L^{2}(\Omega)}\|N\|_{L^{4}(\Omega)}+\|\nabla M_{H}\|_{L^{4}(\Omega)}\|\nabla^{2}w\|_{L^{2}(\Omega)}\|N\|_{L^{4}(\Omega)}
	\notag\\
	& \leq\|\nabla^{2}M_{H}\|^{2}_{L^{4}(\Omega)}\|\nabla w\|^{2}_{L^{2}(\Omega)}+\big(\|N\|^{2}_{L^{2}(\Omega)}+\|N\|_{L^{2}(\Omega)}\|\nabla N\|_{L^{2}(\Omega)}\big)
	\notag\\
	&\qquad +\eps\|\Delta w\|^{2}_{L^{2}(\Omega)}+C\|\nabla M_{H}\|^{2}_{L^{4}(\Omega)}\|N\|^{2}_{L^{4}(\Omega)}
	\notag\\
	& \leq\|\nabla^{2}M_{H}\|^{2}_{L^{4}(\Omega)}\|\nabla w\|^{2}_{L^{2}(\Omega)}+C\|N\|^{2}_{L^{2}(\Omega)}+\eps\|\nabla N\|^{2}_{L^{2}(\Omega)}
	\notag\\
	&\qquad +\eps\|\Delta w\|^{2}_{L^{2}(\Omega)}+C\|\nabla M_{H}\|^{2}_{L^{4}(\Omega)}\big(\|N\|^{2}_{L^{2}(\Omega)}+\|N\|_{L^{2}(\Omega)}\|\nabla N\|_{L^{2}(\Omega)}\big)
	\notag\\
	&\leq \|\nabla^{2}M_{H}\|^{2}_{L^{4}(\Omega)}\|\nabla w\|^{2}_{L^{2}(\Omega)}+C\|N\|^{2}_{L^{2}(\Omega)}+\eps\|\nabla N\|^{2}_{L^{2}(\Omega)}
	+\eps\|\Delta w\|^{2}_{L^{2}(\Omega)}
	\notag\\
	&\qquad 
	+C\|\nabla M_{H}\|^{2}_{L^{4}(\Omega)}\|N\|^{2}_{L^{2}(\Omega)}
	+C\|\nabla M_{H}\|^{4}_{L^{4}(\Omega)}\|N\|^{2}_{L^{2}(\Omega)}
	+\eps\|\nabla N\|^{2}_{L^{2}(\Omega)},
    \\[1ex]
	|I^{N}_{5}|&\leq \|\nabla H\|_{L^{2}(\Omega)}\|w\|_{L^{\infty}(\Omega)}\|N\|_{L^{2}(\Omega)}
	\notag\\
	& \leq \eps\|\Delta w\|^{2}_{L^{2}(\Omega)}+C\|\nabla H\|^{2}_{L^{2}(\Omega)}\|N\|^{2}_{L^{2}(\Omega)}
    \\[1ex]
    \label{IN7}
	|I^{N}_{6}|
	&\leq \frac{1}{2}\|N\|^{2}_{L^{2}(\Omega)}+\frac{1}{2}\|M_{H}-M_{d}\|^{2}_{L^{2}(\Omega)}.
	\end{align}
	Choosing $\eps>0$ sufficiently small, summing \eqref{momentumtesting}, \eqref{Gtesting}, \eqref{Ntesting}, and using the inequalities \eqref{Iw1}--\eqref{Iw5}, \eqref{IG1}--\eqref{IG4} and \eqref{IN1}--\eqref{IN7}, 
	we conclude that
	\begin{align}\label{wGNformatGronwall}
	&\frac{\mathrm d}{\mathrm dt} \mathcal Y_{a}(t) 
	+ \mathcal B_a(t) 
	\notag\\
	&\quad 
	\leq 
	C \mathcal Q_a(t) \mathcal Y_{a}(t) 
	+ C \big(\|v_{H}(t)-v_{d}\|^{2}_{L^{2}(\Omega)}
	    +\|F_{H}(t)-F_{d}\|^{2}_{L^{2}(\Omega)}
	    +\|M_{H}(t)-M_{d}\|^{2}_{L^{2}(\Omega)}\big),
	\end{align}
	for almost all $t\in[0,T]$, where 
	\begin{align*}
	    \mathcal Y_{a} 
	    &:= \int_{\Omega}\big(|\nabla w|^{2}+|\nabla G|^{2}+|N|^{2}\big),
	    \\
	    \mathcal B_a
	    &:= \int_{\Omega}\big(|{\mathcal{S}}w|^{2}+|\Delta G|^{2}+|\nabla N|^{2}\big),
	    \\
	    \mathcal Q_a
	    &:= \|v_{H}\|^{2}_{L^{4}(\Omega)}+\|v_{H}\|^{4}_{L^{4}(\Omega)}+\|\nabla F_{H}\|^{4}_{L^{4}(\Omega)}+\|\nabla F_{H}\|^{4}_{L^{4}(\Omega)}+\|\nabla M_{H}\|^{2}_{L^{4}(\Omega)}+\|\nabla v_{H}\|^{4}_{L^{4}(\Omega)}
    	\notag\\
    	&\quad+\|F_{H}\|^{2}_{L^{4}(\Omega)}+\|F_{H}\|^{4}_{L^{4}(\Omega)}+\|M_{H}\|^{2}_{L^{\infty}(\Omega)}+\|\nabla^{2}M_{H}\|^{2}_{L^{4}(\Omega)}+\|\nabla M_{H}\|^{4}_{L^{4}(\Omega)}+\|H\|^{2}_{L^{4}(\Omega)}
    	\notag\\
    	&\quad+\|\nabla H\|^{2}_{L^{2}(\Omega)}+1.
	\end{align*}
	We further point out that for the derivation of \eqref{wGNformatGronwall}, we also used the estimates 
	\begin{align*}
	    \|\nabla G\|^{2}_{L^{2}(\Omega)}\leq C\|\Delta G\|^{2}_{L^{2}(\Omega)}
	    \quad\text{and}\quad
	    \|G\|^{2}_{L^{2}(\Omega)}\leq C\|\nabla G\|^{2}_{L^{2}(\Omega)}
	\end{align*}
    which follow from Poincar\'{e}'s inequality since $G\vert_{\partial\Omega}=0$ $a.e.$ on $\partial\Omega$.
    
    It is not hard to check that $\norm{\mathcal Q_a}_{L^1([0,T])} \le C$. Specifically, we point out that the terms $\|\nabla F_{H}\|^{4}_{L^{1}(0,T;L^{4}(\Omega)},$  $\|\nabla v_{H}\|^{4}_{L^{1}(0,TL^{4}(\Omega)}$ and $\|\nabla^{2}M_{H}\|^{2}_{L^{1}(0,T;L^{4}(\Omega))}$ can be bounded by proceeding as in \eqref{EST:L4}. We can thus apply Gronwall's lemma to conclude the a priori estimate
	\begin{align*}
	    \norm{\mathcal Y_a}_{L^\infty([0,T])}
	    + \int_0^T \mathcal B_a(t) \dt 
	    \le C.
	\end{align*}
	which can be used to recover the spatial regularity properties collected in \eqref{REG:ADJ*}. The time regularity properties stated in \eqref{REG:ADJ*} then follow by standard comparison arguments.
	
	We still have to show the additional regularity of $N$ stated in \eqref{furtherregN}. 
	In order to apply maximal parabolic regularity theory, we intend to estimate the right-hand side of \eqref{line3} in the $L^{3/2}(Q_{T})$-norm. By straightforward computations, we obtain the following estimates:
    \begin{align*}\label{estimateline3}
        \|2\alpha^{-2} (M_H\cdot N) M_H \|_{L^{2}(Q_{T})}
        &\leq C\|M_{H}\|^{2}_{L^{\infty}(Q_{T})}\|N\|_{L^{2}(Q_{T})},
        \\[1ex]
         \|\alpha^{-2} \big(|M_H|^2 - 1\big)N\|_{L^{2}(Q_{T})}
         &\leq C\big(\|M_{H}\|^{2}_{L^{\infty}(Q_{T})}+1\big)
         \|N\|_{L^{2}(Q_{T})},
        \notag\\[1ex]
         \|(v_H\cdot \nabla) N\|_{L^{2}(0,T;L^{3/2}(\Omega))}
         &\leq C\|v_{H}\|_{L^{\infty}(0,T;L^{6}(\Omega))}
         \|\nabla N\|_{L^{2}(Q_{T})},
        \notag\\[1ex]
        \|2 \dvr( \nabla M_H\, \D w)\|_{L^{2}(0,T;L^{3/2}(\Omega))}
        &\leq C\big(\|\nabla^{2}M_{H}\|_{L^{\infty}(0,T;L^{2}(\Omega))}
        \|\nabla w\|_{L^{2}(0,T;L^{6}(\Omega))}\\
        &\quad\qquad 
        +\|\nabla M_{H}\|_{L^{\infty}(0,T;L^{6}(\Omega))}
        \|\nabla^{2}w\|_{L^{2}(Q_{T})}\big),
        \notag\\[1ex]
        \|\nabla H w\|_{L^{2}(0,T;L^{3/2}(\Omega))}
        &\leq C\|\nabla H\|_{L^{2}(Q_{T})}
        \|w\|_{L^{\infty}(0,T;L^{6}(\Omega))}.
    \end{align*}
	Due to the regularities $H\in \HH$, $(v_{H},p_{H},F_{H},M_{H})\in\VV$ and \eqref{REG:ADJ*} we conclude that the right-hand side of \eqref{line3} is bounded in the $L^{3/2}(\Omega)$-norm. 
	Hence, by employing maximal parabolic regularity, we eventually conclude \eqref{furtherregN}.
\end{proof}
	
\section*{Acknowledgement}
Harald Garcke and Patrik Knopf are partially supported by the RTG 2339 ``Interfaces, Complex Structures, and Singular Limits''
of the Deutsche Forschungsgemeinschaft (DFG, German Research Foundation).  Sourav Mitra and Anja Schl\"omer\-kemper are funded by the Deutsche Forschungsgemeinschaft (DFG, German Research Foundation), grant SCHL 1706/4-2, project number 391682204. S.M. also received partial funding from the Alexander von Humboldt foundation. All the supports are gratefully acknowledged.


\end{document}